\documentclass[12pt]{amsart}

\usepackage{tgpagella}
\usepackage{euler}
\usepackage[T1]{fontenc}
\usepackage{amsmath, amssymb}
\usepackage[hidelinks]{hyperref}
\usepackage[english]{babel}
\usepackage{mathrsfs}
\usepackage{eucal}
\usepackage[all]{xy}
\usepackage{tikz}
\usetikzlibrary{decorations.pathreplacing,decorations.markings}

\usepackage{enumitem}

\newtheorem{thm}{Theorem}[section]
\newtheorem*{thm*}{Theorem}
\newtheorem{lem}[thm]{Lemma}
\newtheorem{fact}[thm]{Fact}

\newtheorem{prop}[thm]{Proposition}
\newtheorem*{prop*}{Proposition}

\newtheorem{cor}[thm]{Corollary}
\newtheorem*{cor*}{Corollary}

\theoremstyle{definition}
\newtheorem{defn}[thm]{Definition}
\newtheorem*{defn*}{Definition}
\newtheorem{notation}[thm]{Notation}
\newtheorem{remark}[thm]{Remark}

\newtheorem{question}[thm]{Question}

\newtheorem*{question*}{Question}
\newtheorem*{Pquestion*}{Popa's question}
\newtheorem{conv}[thm]{Convention}
\newtheorem*{conv*}{Convention}

\def\bb{\mathbb}

\def\bb{\mathbb}

\def\cal{\mathcal}

\def\u{\mathsf 1}

\newcommand{\cstar}{$\mathrm{C}^*$}

\makeatletter

\def\dotminussym#1#2{%
  \setbox0=\hbox{$\m@th#1-$}%
  \kern.5\wd0%
  \hbox to 0pt{\hss\hbox{$\m@th#1-$}\hss}%
  \raise.6\ht0\hbox to 0pt{\hss$\m@th#1.$\hss}%
  \kern.5\wd0}

\DeclareMathOperator{\id}{id}

\DeclareMathOperator{\Aut}{Aut}

\newcommand\bN{{\mathbb N}}

\def \u{\mathcal U}

\def \Malg{\operatorname{MALG}}

\def\Cocy{\operatorname{Cocy}}

\newcommand{\cF}{\mathcal{F}}
\newcommand{\pmp}{p{$.$}m{$.$}p{$.$}}
\newcommand{\acts}{\curvearrowright}
\newcommand{\salg}{\sigma \text{-}\mathrm{alg}}
\newcommand{\sH}{\mathrm{H}}
\newcommand{\rh}{h}
\newcommand{\Borel}{\mathcal{B}}
\newcommand{\given}{\mathbin{|}}
\newcommand{\Given}{\mathbin{\Big|}}
\newcommand{\symd}{\triangle}

\newcommand{\Prob}{\mathrm{Prob}}

\newcommand{\Pow}[1]{\mathrm{Pow}(#1)}

\tikzset{every picture/.style={line width=0.11mm}}

\allowdisplaybreaks
\makeatletter
\def\l@subsection{\@tocline{2}{0pt}{2.5pc}{5pc}{}}
\def\l@subsubsection{\@tocline{2}{0pt}{5pc}{7.5pc}{}}
\makeatother

\usepackage{mathtools}

\begin{document}


\title[E.c.\ p.m.p. actions of approximately treeable groups]{Existentially closed measure-preserving actions of approximately treeable groups}

\author{Isaac Goldbring, Brandon Seward, and Robin Tucker-Drob}
\address{Department of Mathematics\\University of California, Irvine, 340 Rowland Hall (Bldg.\# 400),
Irvine, CA 92697-3875}
\email{isaac@math.uci.edu}
\urladdr{http://www.math.uci.edu/~isaac}
\thanks{Goldbring was partially supported by NSF grant DMS-2054477.}
\address{Department of Mathematics\\University of California, San Diego, 9500 Gilman Drive {\#}0112, La Jolla, CA 92093-0112}
\email{bseward@ucsd.edu}
\urladdr{https://mathweb.ucsd.edu/~bseward/}
\thanks{Seward was partially supported by NSF grants DMS-1955090 and DMS-2054302 as well as Sloan Research Fellowship FG-2021-16246.}
\address{Department of Mathematics\\University of Florida, 1400 Stadium Rd, Gainesville, FL 32611-1927}
\email{r.tuckerdrob@ufl.edu}
\urladdr{https://people.clas.ufl.edu/r-tuckerdrob/}
\thanks{Tucker-Drob was partially supported by NSF grant DMS-2246684}

\begin{abstract}
Given a countable group $\Gamma$, letting $\mathcal{K}_\Gamma$ denote the class of {\pmp} actions of $\Gamma$, we study the question of when the model companion of $\mathcal{K}_\Gamma$ exists.  Berenstein, Henson, and Ibarluc\'ia showed that the model companion of $\mathcal{K}_\Gamma$ exists when $\Gamma$ is a nonabelian free group on a countable number of generators.  We significantly generalize their result by showing that the model companion of $\cal K_\Gamma$ exists whenever $\Gamma$ is an approximately treeable group.  The class of approximately treeable groups contain the class of treeable groups as well as the class of universally free groups, that is, the class of groups with the same universal theory as nonabelian free groups.  We prove this result using an open mapping characterization of when the model companion exists; moreover, this open mapping characterization provides concrete, ergodic-theoretic axioms for the model companion when it exists.  We show how to simplify these axioms in the case of treeable groups, providing an alternate axiomatization for the model companion in the case of the free group, which was first axiomatized by Berenstein, Henson, and Ibarluc\'ia using techniques from model-theoretic stability theory.  Along the way, we prove a purely ergodic-theoretic result of independent interest, namely that finitely generated universally free groups (also known as limit groups) have Kechris' property MD.  We also show that for groups with Kechris' EMD property, the profinite completion action is existentially closed, and for groups without property (T), the generic existentially closed action is weakly mixing, generalizing results of Berenstein, Henson, and Ibarluc\'ia for the case of nonabelian free groups.
\end{abstract}

\maketitle

\tableofcontents
\section{Introduction}

The work presented in this paper lie at the intersection of ergodic theory and model theory.  Our main objects of study are ``model-theoretically generic'' probability measure-preserving (\pmp) actions of discrete countable groups on probability spaces.  Here, the precise formulation of a model-theoretically generic {\pmp} action is that of an \textbf{existentially closed} action.

Given some axiomatizable class $\cal K$ of structures in an appropriate (classical or continuous) language $L$, for example, the class of fields, the class of groups, the class of graphs, the class of Banach spaces, the class of tracial von Neumann algebras, the class of \cstar-algebras etc..., the Robinsonian philosophy for obtaining a good model-theoretic understanding of $\cal K$ is to understand the class of existentially closed members of $\cal K$, for these members of $\cal K$ represent ``universal domains'' for which one should work in when studying members of $\cal K$.  Roughly speaking, e.c.\ members $\cal M$ of $\cal K$ contain all ``solutions'' to ``equations'' with ``coefficients'' from $\cal M$ that ``should have'' solutions, that is, which have solutions in some extension of $\cal M$ belonging to $\cal K$.  For example, e.c.\ fields are precisely the algebraically closed fields, e.c.\ graphs are precisely the ``random'' graphs, and e.c.\ Banach spaces are precisely the Gurarij Banach spaces.  In these examples, something fortuitous occurs:  the class of e.c.\ members of $\cal K$, while a priori not expressible using first-order axioms, do actually themselves form an axiomatizable class.  When this situation happens, we say that the \textbf{model companion} for $\cal K$ exists.  An alternative formulation avoiding the first-order formalism states that the model companion for $\cal K$ exists precisely when the e.c.\ members of $\cal K$ are closed under ultraproducts.  On the other hand, the classes of groups, tracial von Neumann algebras, and \cstar-algebras do \emph{not} admit model companions; from the model-theoretic point of view, these classes are ``wild.''

In this paper, we consider the class $\cal K_\Gamma$ of {\pmp} actions of a given discrete, countable group $\Gamma$ (which is indeed an axiomatizable class in an appropriate continuous language $L_\Gamma$).  Our motivating question is the following:

\begin{question}\label{mainquestion}
For which groups $\Gamma$ does the class $\cal K_\Gamma$ admit a model companion?
\end{question}

Right at the outset, we mention that there does not exist a single group $\Gamma$ for which the answer to the above question is known to be negative.  Our main contribution is to significantly enlarge the class of groups for which the answer to this question is positive.

The first example of a group $\Gamma$ for which the previous question was shown to have a positive answer was the case $\Gamma=\bb Z$, as demonstrated by Berenstein and Henson in \cite{BH}.  There, they show that the e.c.\ actions of $\bb Z$ are precisely the aperiodic ones and they use the classical Rohklin lemma to give a concrete axiomatization of this class.  In unpublished work of Berenstein and Henson, they use the Ornstein-Weiss version of the Rohklin lemma to extend their result to the case of all amenable groups $\Gamma$, whose model companion is simply the axiomatizable class of free actions.

Besides the case of amenable groups, the only other instance of Question \ref{mainquestion} known to have a positive answer is when $\Gamma=\bb F_k$, a finitely generated free group:

\begin{fact}[Berenstein, Henson, and Ibarluc\'ia \cite{BIH}]\label{BHItheorem}
If $\Gamma=\bb F_k$ is a finitely generated free group, then $\cal K_\Gamma$ admits a model companion.
\end{fact}

In the introduction of \cite{BIH}, the authors point out that their result also holds for the free group $\mathbb{F}_\omega$ on a countably infinite set of generators.

By studying general preservation properties for when $\cal K_\Gamma$ admits a model companion, we significantly enlarge the class of groups for which Question \ref{mainquestion} has a positive answer:

\begin{thm*}
If $\Gamma$ is a universally free group, then $\cal K_\Gamma$ admits a model companion.
\end{thm*}

Here, a \textbf{universally free group} (sometimes called a \textbf{$\omega$-residually free} or \textbf{fully residually free} group) is a group $\Gamma$ which is a model of the universal theory of a free group, or, in logic-free terms, a group that embeds in an ultrapower of a free group.  The finitely generated universally free groups are known as \textbf{limit groups} and are of great interest in geometric group theory.

With further effort, we can generalize the previous result even further.  To explain this generalization, we recall that a group is called \textbf{treeable} if it admits a free {\pmp} action whose orbit equivalence relation can be equipped with a treeing (a precise definition is given in Subsection \ref{treeabletypes} below); 
an alternate description of treeability can be given using a theorem of Hjorth \cite{HjorthME}:  
a group is treeable if it is measure equivalent to a free group.  The class of treeable groups is quite rich and contains many interesting groups, such as surface groups and, more generally, groups admitting a planar Cayley graph.
By a result from \cite{BTW}, the class of treeable groups does have some intersection with the aforementioned class of universally free groups:  every \textbf{elementarily free group}, that is, every finitely generated group with the same first order theory as a nonabelian free group, is treeable.  

One can generalize further by considering the class of \textbf{approximately treeable} groups (a precise definition is given in Subsection \ref{treeabletypes}); approximately treeable groups were studied in \cite{Gab02} under the guise of groups which have approximate ergodic dimension at most $1$. 
In Section \ref{approximatelytreeable}, it is remarked that all universally free groups are approximately treeable.  
(It is a well-known open problem whether or not all limit groups are in fact treeable.)  
The class of approximately treeable groups is a proper extension of the class of all treeable groups as witnessed by the group $\mathbb{F}_2\times \mathbb{Z}$ (which is not treeable by \cite{Gab00} and by \cite{PP00}).  
In Subsection \ref{openmappingsub} below, we develop an ``open mapping'' characterization of when the model companion of $\cal K_\Gamma$ exists and use this characterization in Section \ref{approximatelytreeable} to prove the following substantial generalization of Fact \ref{BHItheorem}:

\begin{thm*}
If $\Gamma$ is an approximately treeable group, then the model companion of $\cal K_\Gamma$ exists.
\end{thm*}

While proving Fact \ref{BHItheorem}, the authors give a concrete axiomatization of the class of e.c.\ members of $\cal K_\Gamma$.  However, their proof adapts known model-theoretic techniques from the area of model theory known as \textbf{stability theory} and extracting an ``ergodic-theoretic'' axiomatization of this class from their axioms is not entirely straightforward.

The open mapping characterization of the existence of the model companion proved in Subsection \ref{openmappingsub} actually yields purely ergodic-theoretic axioms for the model companion whenever it exists.  That being said, these axioms are not entirely illuminating and it seems desirable to find groups for which the axioms for the model companion of $\cal K_\Gamma$ can be made simpler. In Section \ref{concrete} below, we prove the following result along these lines:

\begin{thm*}
    If $\Gamma$ has the \textbf{extension-MD property} and the \textbf{definable cocycle property}, then the model companion of $\cal K_\Gamma$ exists.  Moreover, one can list simple axioms for the model companion of an ergodic-theoretic nature. 
\end{thm*}

The first condition in the theorem is related to the \textbf{MD property}, first introduced by Kechris in \cite{kechrisweak}; roughly speaking, a residually finite group has the MD property if the set of profinite actions of the group is dense in the space of all actions.  If the a priori stronger condition that the set of ergodic profinite actions is dense in the space of all actions, then group is said to have the \textbf{EMD property}.  We say more about these properties in Section \ref{special} below.  The second condition in the theorem roughly states that every cochain that is close to satisfying the cocycle identity is near an actual cocycle. 

In Subsections \ref{subs:definablecocycle} and \ref{subs:weakMD} respectively, we show, using fairly elementary means, that free groups satisfy both properties, whence one obtains a simple, purely ergodic-theoretic axiomatization of the model companions of $\cal K_{\bb F_k}$ and $\cal K_{\bb F_\omega}$.    In Subsections \ref{subs:groupcoh} through \ref{subs:cocytree}, we generalize the result that free groups have the extension-MD property and the definable cocycle property to the wider class of \textbf{strongly treeable} groups:

\begin{thm*}
Strongly treeable groups have the extension-MD property and the definable cocycle property.
\end{thm*}  

Here, a group is strongly treeable if \emph{all} of its free {\pmp} actions admit a treeing.  
It is a well-known open problem whether there exists a group that is treeable but not strongly treeable.  In any event, we now have that if $\Gamma$ is a strongly treeable group, then not only does the model companion of $\mathcal{K}_\Gamma$ exist, but there is a very simple set of axioms for this model companion.  We also show how to tweak these results to cover the case of a treeable (but not necessarily strongly treeable) group at the cost of slightly complicating the axioms. In the process, we obtain the following simple characterization of e.c.\ actions of treeable groups.

\begin{thm*}
	Let $\Gamma$ be a treeable group. Then a {\pmp} action $\Gamma \acts^a (X, \mu)$ is e.c.\ if and only if all of the following hold:
	\begin{enumerate}
		\item $\Gamma \acts^a (X, \mu)$ weakly contains a free treeable action of $\Gamma$;
		\item the trivial extension $\Gamma \acts^{a \times \id} (X \times [0,1], \mu \times \lambda)$ where $\Gamma$ acts on $[0,1]$ by fixing every point, is an e.c.\ extension of $\Gamma \acts^a (X, \mu)$;
		\item $B^1(a, \operatorname{Sym}(k))$ is dense in $Z^1(a, \operatorname{Sym}(k))$ for every $k \in \bb N$.
	\end{enumerate}
\end{thm*}

In \cite{BIH}, Berenstein, Henson, and Ibarluc\'ia show that the profinite completion action of the free group is a concrete example of an e.c.\ action.  In Section \ref{special}, we extend their result in an optimal way by proving the following:

\begin{thm*}
The profinite completion of any group with the EMD property is an e.c.\ action. 
\end{thm*}

  We note that being an EMD group is a necessary condition for the profinite completion to be an e.c.\ action, whence this result is indeed optimal.  Along the way, we prove the following result, which is purely ergodic-theoretic and of independent interest:

\begin{thm*}
Limit groups have property MD.
\end{thm*}

Berenstein, Henson, and Ibarluc\'ia also prove the existence of a weakly mixing e.c.\ action of the free group.  Using a result of Kerr and Pichot, we extend their result to any group without property (T):

\begin{thm*}
If $\Gamma$ is a group without property (T), then the model-theoretically ``generic'' e.c.\ action of $\Gamma$ is weakly mixing.
\end{thm*}

A precise formulation of ``generic'' in the previous theorem will be given in Section \ref{special}.  We note that, for the conclusion of the theorem to hold, the group cannot have property (T), whence the result is once again optimal.  The main technical idea behind this proof is to use various descriptions of weakly mixing actions for nonamenable groups developed by Bergelson to show that they have a very particular infinitary first-order description.

Throughout this paper, we assume familiarity with basic continuous model theory as well as basic ergodic theory.  We refer the reader to \cite{bbhu} for the former and \cite{Walters} for the latter.

\section{Preliminaries}

In this section, we gather all of the necessary preliminaries that are used in the rest of the paper.

\subsection{Probability measure-preserving actions as metric structures}\label{pmpstructures}

In this subsection, we explain how we view probability measure-preserving actions of a countable group $\Gamma$ as structures in continuous logic.  We follow closely the excellent presentation given in \cite{IT}.

Throughout this paper, $\Gamma$ denotes a countable, discrete group.  Given a probability space $(X,\mathcal{B},\mu)$ (sometimes abbreviated $(X,\mu)$ for the sake of brevity), a \textbf{probability measure-preserving (\pmp) action} of $\Gamma$ on $(X,\mathcal{B},\mu)$ is a homomorphism $a:\Gamma\to \operatorname{Aut}(X,\mathcal{B},\mu)$.  We often denote such actions by $\Gamma\acts^a (X,\mathcal{B},\mu)$ or sometimes by the simpler notation $\Gamma\acts^a (X,\mu)$ or even $\Gamma\acts^a X$.  We also write $\gamma^a \cdot x$, $\gamma^a x$, or even $\gamma x$ (when $a$ can be inferred from the context) instead of $a(\gamma)(x)$ for $\gamma\in \Gamma$ and $x\in X$.  Two actions $\Gamma\acts^a X$ and $\Gamma \acts^b X$ are \textbf{isomorphic} or \textbf{conjugate} if there is a measure space isomorphism $\Phi:X\to Y$ such that $\Phi(\gamma^a x)=\gamma^b \Phi(x)$ for all $\gamma\in \Gamma$ and a.e. $x\in X$.  

Recall that one associates to a probability space $(X,\cal B,\mu)$ its \textbf{measure algebra} $\Malg(X,\mathcal{B},\mu)$, or $\Malg(X)$ for the sake of brevity, which is defined to be the Boolean algebra of equivalence classes of elements of $\cal B$ with respect to the pseudometric $d_\mu(A,B):=\mu(A\triangle B)$.  The pseudometric $d_\mu$ naturally induces a metric on $\Malg(X)$, still denoted by $d_\mu$, and the set-theoretic operations of union, intersection, and complement on $\cal B$ induce operations on $\Malg(X)$ rendering it a measured Boolean algebra for which the Boolean algebra operations are uniformly continuous with respect to the metric $d_\mu$.  Moreover, the countable additivity of the measure $\mu$ implies that $d_\mu$ is a complete metric.  Given $A\in \mathcal{B}$, we denote its equivalence class in $\Malg(X)$ by $[A]_\mu$, or sometimes by the corresponding lowercase letter $a$.  The associated measure algebra is separable with respect to the metric $d_\mu$ if and only if $X$ is a \textbf{standard} probability space.This process yields an equivalence of categories between the categories of probability spaces and the associated measure algebras.  Moreover, if $\Gamma$ is a countable group, then the {\pmp} actions of $\Gamma$ on a measure space $X$ correspond in this duality to isometric actions of $\Gamma$ on $\Malg(X)$ that preserve the Boolean operations.

We let $L_\Gamma$ be the language consisting of function symbols for the Boolean operations, a predicate symbol for the measure, and a unary function symbol $u_\gamma$ for each $\gamma\in \Gamma$.  The moduli of uniform continuity for these symbols will be assumed to be the natural ones.  We let $T_\Gamma$ denote the $L_\Gamma$-theory which axiomatizes isometric actions of $\Gamma$ on measure algebras.  More precisely, the axioms include the familiar axioms for measure algebras, axioms which state that each $u_\gamma$ is an automorphism, and axioms stating that $u_{\gamma_1}\circ u_{\gamma_2}=u_{\gamma_1\gamma_2}$ for all $\gamma_1,\gamma_2\in \Gamma$.  These axioms are clearly universal.  If $\Gamma\acts^a X$ is a {\pmp} action of $\Gamma$, we let $\cal M_a$ denote the corresponding model of $T_\Gamma$.

Given {\pmp} actions $\Gamma\acts^a (X,\cal B,\mu)$ and $\Gamma\acts^b(Y,\cal C,\nu)$ of $\Gamma$, recall that a \textbf{factor map} $\pi:(X,\cal B,\mu)\to (Y,\cal C,\nu)$ between these actions is a measurable map that commutes with the actions of $\Gamma$ and for which $\pi_*\mu=\nu$; we then refer to $b$ as a factor of $a$.  Sometimes we abuse notation and simply write $\pi:X\to Y$ for a factor map.  If $\Lambda$ is a subgroup of $\Gamma$, then we call a measurable map $\pi:X\to Y$ a \textbf{$\Lambda$-equivariant factor map} if it a factor map between the restricted actions of $\Lambda$ on $X$ and $Y$.

Given such a factor map $\pi$, it is clear that $\pi^{-1}(\cal C)$ is a $\Gamma$-invariant $\sigma$-subalgebra of $\cal B$.  Conversely, every $\Gamma$-invariant $\sigma$-subalgebra of $\cal B$ arises from a factor map in this way.  Consequently, there is a duality between substructures of models of $T_\Gamma$ and factor maps.  More precisely, if $\cal M\models T_\Gamma$ and $\cal N$ is a substructure of $\cal M$, then we can find {\pmp} actions $\Gamma\acts^a X$ and $\Gamma\acts^b Y$ for which $\cal M\cong \cal M_a$, $\cal N\cong \cal M_b$, and for which $b$ is a factor of $a$.

In the sequel, it will behoove us to know that we can also axiomatize the \textbf{free} {\pmp} actions of $\Gamma$, that is, those actions $\Gamma \acts^a X$ for which, for every $\gamma\in \Gamma\setminus\{e\}$ and $x \in \Malg(X)$ with $\mu(x) > 0$ there is $y \subseteq x$ with $\mu(y) > 0$ and $\mu(y \cap \gamma y) = 0$ (when $X$ is standard this is equivalent to $\mu(\{x\in X \ : \ \gamma^a x=x\})=0$).  Indeed, as observed in \cite[Def. 2.4]{IT}, we can form the theory $T_{\Gamma,free}$ consisting of the above axioms $T_\Gamma$ together with axioms
$$\inf_x \max(|\mu(x)-1/n|, \ \max_{i<j<n} \mu(\gamma^i x \cap \gamma^j x)) = 0$$
for each $(\gamma, n) \in \Gamma \times \bb N$ for which $n < |\langle \gamma \rangle| = \infty$ or $n = |\langle \gamma \rangle| < \infty$.

\subsection{Space of actions}\label{subsectionspace}

Let $(X,\mu)$ be a standard probability space; unless otherwise stated, we allow for the possibility that these probability spaces have atoms. We equip the automorphism group $\Aut(X,\mu)$ of $(X,\mu)$ with the weak topology, namely the weakest topology for which all maps $$T\mapsto [T(A)]_\mu:\Aut(X,\mu)\to \Malg(X)$$ are continuous, as $A$ ranges over measurable subsets of $X$.  With this topology, $\Aut(X,\mu)$ is a Polish group.  We also set $A(\Gamma ,X, \mu)$ to be the set of {\pmp} actions of $\Gamma$ on $X$.  By viewing $A(\Gamma,X,\mu)$ as a closed subspace of $\Aut(X,\mu)^\Gamma$, it is also naturally a Polish space. 

For our purposes, we will also need a relative version of this topology.  Towards this end, consider another standard probability space $(Y,\nu)$ and let $F(Y,X)$ denote the space of all measure preserving maps from $Y$ to $X$, where we identify two such maps if they agree almost everywhere. This is naturally a Polish space since it can be identified with the space of all isometric embeddings from the measure algebra of $X$ into the measure algebra $Y$ sending $[\emptyset ]_{\mu}$ to $[\emptyset ]_{\nu}$, equipped with the pointwise convergence topology.

Let $F(\Gamma ,Y,X )$ denote the set of all triples $(\phi ,b,a)$, where $b\in A(\Gamma ,Y,\mu)$, $a\in A(\Gamma ,X,\mu)$, and $\phi$ is a factor map from $b$ to $a$. Then $F(\Gamma ,Y,X)$ is a closed subspace of $F(Y,X)\times A(\Gamma ,Y,\nu)\times A(\Gamma ,X,\mu)$ hence is a Polish space. For a fixed action $a\in A(\Gamma ,X,\mu )$, let $F_a(\Gamma ,Y,X)$ denote the slice above $a$ in $F(\Gamma ,Y,X)$, which we identify with the set of all pairs $(\phi , b )\in F(Y,X)\times A(\Gamma,Y,\nu)$ where $\phi$ factors $b$ onto $a$. Elements of $F_a(\Gamma ,Y,X)$ are called {\bf extensions of $a$} or {\bf $a$-extensions}. Two $a$-extensions, $(\phi , b )\in F_a(Y,X)$ and $(\psi ,c )\in F_a(X,Z)$ are {\bf isomorphic} if there is a measure space isomorphism $S:Y\rightarrow Z$ such that $S\cdot (\phi , b) = (\psi , c )$, where $S\cdot (\phi , b ) \coloneqq (\phi S^{-1}, S\cdot b )$ and $\gamma ^{S\cdot b}\coloneqq S\gamma ^{b}S^{-1}$.

Given finite collections $\mathcal{A}$ and $\mathcal{B}$ of measurable subsets of $Y$ and $X$ respectively, a finite subset $F$ of $\Gamma$, and $\epsilon >0$, we obtain a basic open neighborhood $U^{\mathcal{A},\mathcal{B},F,\epsilon}_a(\phi ,b )$ of $(\phi , b )\in F_a(\Gamma,Y,X)$ consisting of all pairs $(\psi , c)\in F_a(\Gamma ,Y,X )$ satisfying, for all $A\in \mathcal{A}$, $B\in \mathcal{B}$, and $\gamma \in F$:
$$
\nu(\phi ^{-1}A\triangle \psi ^{-1}A)<\epsilon \text{ and }
\nu(\gamma ^b B\triangle \gamma ^cB)<\epsilon .
$$

\subsection{Ultraproducts of \texorpdfstring{\pmp}{p.m.p.} actions}\label{ultrasub}

Suppose that $(\Gamma\acts^{a_i}(X_i,\mu_i))_{i\in I}$ is a family of {\pmp} actions of $\Gamma$ and $\u$ is an ultrafilter on $I$.  Let $\cal M_{a_i}\models T_\Gamma$ denote the $L_\Gamma$-structure associated to the action $a_i$.  Then we can take the model-theoretic ultraproduct $\prod_\u \cal M_{a_i}$, which is necessarily an isometric action of $\Gamma$ on the measure algebra associated to some {\pmp} action of $\Gamma$.  This action is none other than the ultraproduct action $\Gamma\acts^{\prod_\u a_i}\prod_\u (X_i,\mu_i)$ (see, for example, \cite{CKT}).  More precisely, let $\prod_\u X_i$ denote the set-theoretic ultraproduct of the sets $X_i$ and let $\mu_0$ denote the finitely additive measure on the Boolean algebra $\cal B_0$ of subsets of $\prod_\u X_i$ of the form $\prod_\u A_i$, where each $A_i\subseteq X_i$ is measurable.  Define $N\subseteq \prod_\u X_i$ to be $\mu_0$-null if, for every $\epsilon>0$, there is $A\in \cal B_0$ such that $N\subseteq A$ and $\mu_0(A)<\epsilon$.  Then set $\cal B$ to be the $\sigma$-algebra on $\prod_\u X_i$ generated by $\cal B_0$ and the $\sigma$-ideal of null sets and let $\mu$ be the unique extension of $\mu_0$ to a probability measure on $\cal B$.  Then the diagonal action of $\Gamma$ on $\prod_{i\in I}X_i$ induces a {\pmp} action $\Gamma\acts^{\prod_\u a_i} (\prod_\u X_i,\mu)$ and the corresponding action on measure algebras is precisely the model $\prod_\u \cal M_{a_i}$.

In regards to ultraproducts, we adopt the following notation.  Given a sequence $x_i \in \prod_{i\in I}X_i$, we let $[x_i]_\u$ denote the corresponding element of $\prod_\u X_i$.  Similarly, given a sequence of measurable sets $A_i\subseteq X_i$, we write $[A_i]_\u$ instead of $\prod_\u A_i$. Note that every element in the measure algebra of $\prod_\u (X_i,\mu _i)$ can be represented by a set of the form $[A_i]_\u$.

In general, we note that an ultraproduct action is almost never standard (unless the ultrafilter is somewhat trivial, for example countably complete), but one can always take a separable elementary substructure of the model-theoretic ultraproduct and then the corresponding action will be standard.

When each $\Gamma\acts^{a_i}(X_i,\mu_i)=\Gamma\acts^a(X,\mu)$ for some common action $\Gamma\acts ^a(X,\mu)$, we speak of the \textbf{ultrapower action} $\Gamma\acts^{a_\u}(X,\mu)_\u$ (corresponding to the model-theoretic ultrapower $\cal M_a^\u$) and the diagonal embedding $\cal M_a\hookrightarrow \cal M_a^\u$ corresponds to the diagonal factor map $X_\u\to X$.

Recall that if $\Gamma\acts^a (X,\mu)$ and $\Gamma\acts^b (Y,\nu)$ are {\pmp} actions of $\Gamma$, then $a$ is \textbf{weakly contained} in $b$, denoted $a\preceq b$, if for any finitely many measurable sets $A_1,\ldots,A_n\subseteq X$, finite $F\subseteq \Gamma$, and $\epsilon>0$, there are measurable sets $B_1,\ldots,B_n\subseteq Y$ such that $|\mu(\gamma^aA_i\cap A_j)-\nu(\gamma^bB_i\cap B_j)|<\epsilon$ for all $1\leq i,j\leq n$ and $\gamma\in F$.

\begin{fact}
If $(X,\mu)$ and $(Y,\nu)$ are standard, non-atomic probability spaces, the following are equivalent:
\begin{enumerate}
    \item $a\preceq b$.
    \item There is a sequence of actions $(c_n)_{n\in \bb N}\in A(\Gamma,X,\mu)$ such that $c_n\cong b$ (conjugacy of actions) and $a=\lim_n c_n$
    \item $a$ is contained in some ultrapower $b_\u$ of $b$, that is, if $a$ is a factor of $b_\u$.
    \item $\cal M_a$ is embeddable in an ultrapower $\cal M_b^\u$ of $\cal M_b$.
\end{enumerate}
\end{fact}

For a discussion of the preceding fact, see \cite[Subection 2.2]{kechrisburton}.

\begin{conv}\label{ultraconvention}
    In the remainder of this article, unless explicitly stated, $\u$ always denotes a nonprincipal ultrafilter on a countable index set $I$.  (Many of the facts to follow work for nonprincipal ultrafilters on arbitrary index sets, but some of them might require further mild hypotheses on the ultrafilter such as countable incompleteness.)
\end{conv}

\subsection{Cocycles}

We will discuss cocycles using three different frameworks: ergodic theory,  group cohomology, and continuous model theory. We start with the ergodic theory perspective.

Fix an action $\Gamma\acts^a X$ and a Polish group $G$. A \textbf{cochain} for $a$ is a measurable map $\sigma : \Gamma \times X \rightarrow G$. A \textbf{cocycle} for $a$ is a cochain $\sigma:\Gamma\times X\to G$ satisfying the cocycle identity $\sigma(\gamma_1\gamma_2,x)=\sigma(\gamma_1,\gamma_2x)\sigma(\gamma_2,x)$ for all $\gamma_1,\gamma_2\in \Gamma$ and almost all $x\in X$. We let $Z^1(a,G)$ denote the collection of cocycles for $a$ with values in $G$.

We say that two cocycles $\sigma_1,\sigma_2:\Gamma\times X\to G$ of the action $a$ are \textbf{cohomologous} if there is a measurable map $s:X\to G$ for which $\sigma_2(\gamma,x)=s(\gamma x)^{-1}\sigma_1(\gamma,x)s(x)$ for all $\gamma\in \Gamma$ and for almost all $x\in X$. The set of cohomology classes of cocycles with values in $G$ is called the \textbf{first cohomology space of $a$ relative to $G$}, denoted $H^1(a,G)$.

The \textbf{trivial cocycle} for $a$ sends each pair $(\gamma,x)\in\Gamma\times X$ to the identity $e_G \in G$.  A cocycle for $a$ cohomologous to the trivial cocycle is called a \textbf{coboundary} for $a$.  The set of coboundaries for $a$ with values in $G$ is denoted by $B^1(a,G)$.

We consider $Z^1(a,G)$ as topologized so that $\sigma_n\to \sigma$ if and only if, for each $\gamma\in \Gamma$, we have that the functions $\sigma_n(\gamma, \cdot) : X \to G$ converge to $\sigma(\gamma, \cdot)$ in measure. We say that cocycles $\sigma_1$ and $\sigma_2$ are \textbf{approximately equivalent} if $\sigma_1$ belongs to the closure of the cohomology class of $\sigma_2$.  In other words, $\sigma_1$ and $\sigma_2$ are approximately equivalent if and only if there are measurable functions $f_n:X\to G$ such that, for all $\gamma\in \Gamma$, we have that the functions $x \mapsto f_n(\gamma x)\sigma_1(\gamma,x)f_n(x)^{-1}$ converge in measure to $x \mapsto \sigma_2(\gamma,x)$.

In the case that $(U,\rho)$ is a standard probability space and $G = \Aut(U, \rho)$, from a cocycle $\sigma$ for $a$ we can form the \textbf{skew-product extension} $X\times_\sigma (U,\rho)$, which is the {\pmp} action of $\Gamma$ on $(X\times U, \mu \times \rho)$ given by $\gamma(x,u):=(\gamma x,\sigma(\gamma,x)u)$. (Our primary interest will be in the cases where $(U,\rho)$ is either a finite set, equipped with its counting measure, in which case $\Aut(U,\rho)$ is simply $S_n$ for some $n\in \bb N$, or where $U = [0,1]$ and $\rho$ is Lebesgue measure.)  Note then that the projection map $X\times U\to X$ is a factor map.  \textbf{Rohklin's skew-product theorem} asserts that every ergodic action that factors onto $X$ is isomorphic to a skew-product extension of $X$ (see, for example, \cite{glasner}). We remark that when $\sigma_1$ and $\sigma_2$ are cohomologous, say with $s:X\to \Aut(U, \rho)$ measurable and satisfying $\sigma_2(\gamma,x)=s(\gamma x)^{-1}\sigma_1(\gamma,x)s(x)$ for all $\gamma\in \Gamma$ and for almost all $x\in X$, then the map $$(x,u)\mapsto (x,s(x)^{-1}u):X\times_{\sigma_1}U\to X\times_{\sigma_2}U$$ is an isomorphism of the associated skew-product extensions. In particular, when $\sigma$ is a coboundary the corresponding skew-product extension is isomorphic to the product of $a$ with the trivial action on $U$.

We remark that $\Aut([0,1], \lambda)$, where $\lambda$ is the Lebesgue measure, is a Polish group when equipped with the \textbf{weak topology}, which is the weakest topology for which all maps $$T\mapsto [T(A)]_\mu:\Aut([0,1],\lambda)\to \Malg([0,1])$$ are continuous, as $A$ ranges over Borel subsets of $[0,1]$.

We now turn to the group cohomology perspective, which we use to access topological arguments that will streamline some of our proofs. Let $G$ be any Polish group. In addition to the usual left-shift action $\Gamma \acts^s G^\Gamma$ given by the rule
\[(\gamma^s \cdot y)(\delta) = y(\gamma^{-1} \delta),\]
we will also make use of the right-shift action $\Gamma \acts^r G^\Gamma$ given by the formula
\[(\gamma^r \cdot y)(\delta) = y(\delta \gamma).\]
For $y \in G^\Gamma$ and $\gamma \in \Gamma$ we will write $y^\gamma$ in place of $\gamma^r \cdot y$ (this notation may require some care, since $(y^\gamma)^\delta = y^{\delta \gamma}$). Notice that $G^\Gamma$ is a group with respect to coordinate-wise composition and that for every $\gamma \in \Gamma$ the map $y \in G^\Gamma \mapsto y^\gamma \in G^\Gamma$ is a group automorphism.

Define the space of $1$-cochains $C^1(\Gamma, G^\Gamma)$ to be the set of all functions from $\Gamma$ to $G^\Gamma$. Similarly, we define the space of $2$-cochains $C^2(\Gamma, G^\Gamma)$ to be the set of all functions from $\Gamma \times \Gamma$ to $G^\Gamma$. We endow $C^1(\Gamma, G^\Gamma)$ and $C^2(\Gamma, G^\Gamma)$ with the product topologies obtained by identifying them with the sets $G^{\Gamma \times \Gamma}$ and $G^{\Gamma \times \Gamma \times \Gamma}$, respectively.

The coboundary map $\partial : C^1(\Gamma, G^\Gamma) \rightarrow C^2(\Gamma, G^\Gamma)$ is given by the formula
\[\partial c(\alpha, \beta) = c(\alpha \beta)^{-1} c(\beta)^\alpha c(\alpha) .\]
Note that in the above equation $\partial c(\alpha, \beta)$ is an element of $G^\Gamma$ and that the right-hand side is a product of elements of $G^\Gamma$. We call $c \in C^1(\Gamma, G^\Gamma)$ a cocycle if $\partial c(\alpha, \beta)(\gamma) = e_G$ for all $\alpha, \beta, \gamma \in \Gamma$, and we denote by $Z^1(\Gamma, G^\Gamma)$ the set of all cocycles.

We define actions of $\Gamma$ on $C^1(\Gamma, G^\Gamma)$ and $C^2(\Gamma, G^\Gamma)$ via the formulas
\begin{align*}
	(\alpha^t \cdot c_1)(\beta)(\gamma) & = c_1(\beta)(\alpha^{-1} \gamma)\\
	(\alpha^t \cdot c_2)(\beta, \gamma)(\delta) & = c_2(\beta, \gamma)(\alpha^{-1} \delta)
\end{align*}
for $c_1 \in C^1(\Gamma, G^\Gamma)$, $c_2 \in C^2(\Gamma, G^\Gamma)$ and $\alpha, \beta, \gamma, \delta \in \Gamma$. Equivalently, these actions are described in terms of the left-shift action $\Gamma \acts^s G^\Gamma$ via the formulas
\begin{align*}
	(\alpha^t \cdot c_1)(\beta) & = \alpha^s \cdot (c_1(\beta))\\
	(\alpha^t \cdot c_2)(\beta, \gamma) & = \alpha^s \cdot (c_2(\beta, \gamma)).
\end{align*}
Using the fact that the left and right shift actions of $\Gamma$ on $G^\Gamma$ commute we have that for $c \in C^1(\Gamma, G^\Gamma)$ and $\alpha, \beta, \gamma \in \Gamma$
\[(\gamma^t \cdot c)(\beta)^\alpha = (\gamma^s \cdot (c(\beta)))^\alpha = \gamma^s \cdot (c(\beta)^\alpha).\]
Therefore
\begin{align*}
	\partial(\gamma^t \cdot c)(\alpha, \beta) & = (\gamma^t \cdot c)(\alpha \beta)^{-1} \cdot (\gamma^t \cdot c)(\beta)^\alpha \cdot (\gamma^t \cdot c)(\alpha)\\
	& = (\gamma^s \cdot (c(\alpha \beta)))^{-1} \cdot (\gamma^s \cdot (c(\beta)^\alpha)) \cdot (\gamma^s \cdot (c(\alpha)))\\
	& = \gamma^s \cdot (c(\alpha \beta)^{-1} c(\beta)^\alpha c(\alpha))\\
	& = \gamma^s \cdot (\partial c(\alpha, \beta))\\
	& = (\gamma^t \cdot \partial c)(\alpha, \beta).
\end{align*}
As a result, we have that the coboundary map $\partial$ is $\Gamma$-equivariant. This additionally implies that $Z^1(\Gamma, G^\Gamma)$ is $\Gamma$-invariant.

\begin{lem} \label{lem:CocycleCorr}
	Let $\Gamma \acts^a (X, \mu)$ be a {\pmp} action and let $G$ be a Polish group. Then there is a one-to-one correspondence between measurable cochains $\theta : \Gamma \times X \rightarrow G$ and measurable equivariant maps $c : X \rightarrow C^1(\Gamma, G^\Gamma)$ given by the relations
	\begin{equation} \label{eqn:CocycleCorr}
		c(x)(\beta)(\alpha) = \theta(\beta^{-1}, (\alpha^{-1})^a \cdot x) \quad \text{and} \quad \theta(\gamma, x) = c(x)(\gamma^{-1})(e).
	\end{equation}
	Moreover, this correspondence produces a bijection between the measurable cocycles $\sigma : \Gamma \times X \rightarrow G$ and measurable equivariant maps $z : X \rightarrow Z^1(\Gamma, G^\Gamma)$.
\end{lem}

\begin{proof}
	It is immediately apparent that if $\theta : \Gamma \times X \rightarrow G$ is any measurable map (that is, a cochain) then the function $c : X \rightarrow C^1(\Gamma, G^\Gamma)$ given by the left-hand side of (\ref{eqn:CocycleCorr}) is measurable and that $\theta$ is recovered from $c$ by the formula in the right-hand side of (\ref{eqn:CocycleCorr}). Additionally, when $c$ is obtained from $\theta$ using the left side of (\ref{eqn:CocycleCorr}), $c$ is automatically equivariant since
	\[c(\gamma^a \cdot x)(\beta)(\alpha) = \theta(\beta^{-1}, (\alpha^{-1} \gamma)^a \cdot x) = c(x)(\beta)(\gamma^{-1} \alpha) = (\gamma^t \cdot c(x))(\beta)(\alpha).\]
	Conversely, given any measurable equivariant map $c : X \rightarrow C^1(\Gamma, G^\Gamma)$, if $\theta$ is defined using the right side of (\ref{eqn:CocycleCorr}) then the equivariance of $c$ implies that
	\[\theta(\beta^{-1}, (\alpha^{-1})^a \cdot x) = c((\alpha^{-1})^a \cdot x)(\beta)(e) = ((\alpha^{-1})^t \cdot c(x))(\beta)(e) = c(x)(\beta)(\alpha),\]
	and thus $c$ is recovered from $\theta$ using the left side of (\ref{eqn:CocycleCorr}). Finally, for any $x \in X$ we have that $c(x) \in Z^1(\Gamma, G^\Gamma)$ if and only if $c(x)(\alpha \beta)(\gamma)$ is equal to
	\[c(x)(\beta)^\alpha(\gamma) c(x)(\alpha)(\gamma) = c(x)(\beta)(\gamma \alpha) c(x)(\alpha)(\gamma)\]
	for all $\alpha, \beta, \gamma \in \Gamma$. When $c$ and $\theta$ are related according to the formulas in (\ref{eqn:CocycleCorr}), this is equivalent to the requirement that $\theta(\beta^{-1} \alpha^{-1}, (\gamma^{-1})^a \cdot x)$ is equal to $\theta(\beta^{-1}, (\alpha^{-1} \gamma^{-1})^a \cdot x) \theta(\alpha^{-1}, (\gamma^{-1})^a \cdot x)$ for all $\alpha, \beta, \gamma \in \Gamma$. Thus $c(x) \in Z^1(\Gamma, G^\Gamma)$ for $\mu$-almost-every $x$ if and only if $\theta$ is a cocycle.
\end{proof}

Lastly, we discuss the model-theoretic perspective on cocycles, exclusively looking at the case where $G = K$ is a finite group. In what follows, we use the terminology around definability presented in the first author's article \cite{spgap}.  For any finite group $K$, we may assume that we have a sort $S_K$ in our langugage, whose intended interpretation in $\cal M_a$ is simply the Cartesian product $\cal M_a^{\Gamma\times K}$ (equipped with some fixed compatible complete metric), and that we have function symbols $\pi_{\gamma,k}$ for each $\gamma\in \Gamma$ and $k\in K$ whose intended interpretations are the projections maps $\pi_{\gamma,k}^{\cal M_a}:\cal M_a^{\Gamma\times K}\to \cal M_a$.  We identify an element $B\in \cal M_a^{\Gamma\times K}$ with the cochain $\sigma_B:\Gamma\times X\to K$ for which $\pi_{\gamma,k}(B)=\{x\in X \ : \ \sigma_B(\gamma,x)=k\}$.  Conversely, given a cochain $\sigma:\Gamma\times X\to K$, we let $B_\sigma\in \cal M_a^{\Gamma\times K}$ denote the corresponding tuple.

Fix bijections $e:\Gamma\to \bb N$ and $f:\Gamma\times \Gamma\times K\to \mathbb N$.  We define the $T_\Gamma$-formula $\Cocy_K$ whose interpretations are given by $\Cocy_K^{\cal M_a}: \cal M_a^{\Gamma\times K}\to [0,1]$, where 
$$\Cocy_K^{\cal M_a}(B)=\max(\Phi_1(B),\Phi_2(B)),$$ with $\Phi_1(B)$ the $T_\Gamma$-formula
$$\sum_{\gamma\in \Gamma}2^{-e(\gamma)}d((B(\gamma,k))_{k\in K},\operatorname{Part}_K)$$

and $\Phi_2(B)$ the $T_\Gamma$-formula $$\sum_{(\gamma,\delta,k)\in \Gamma\times \Gamma\times K}2^{-f(\gamma,\delta,k)}d\left( \bigcup_{h \in K} (\delta^{-1}B(\gamma,h)\cap B(\delta,h^{-1}k)),B(\gamma\delta,k)\right).$$
In the definition of $\Phi_1$, $\operatorname{Part}_K$ denotes the $T_\Gamma$-definable set given by partitions of the measure algebra indexed by $K$.  We note that $\Cocy_K^{\cal M_a}(B)=0$ if and only if $\sigma_B$ is a cocycle of the action $a$.  In other words, the $T_\Gamma$-functor corresponding to $K$-valued cocycles is a $T_\Gamma$-zeroset that we denote $Z(\Cocy_K^{\cal M_a})$.

\subsection{Treeability notions for groups}\label{treeabletypes}
Consider a countable group $\Gamma$ and a {\pmp} action $\Gamma \acts^a (X, \mu)$. Let $\mathcal R_a$ be the equivalence relation on $X$ given by the orbits. We call a set $\cal G \subseteq \cal R_a$ a {\bf directed measurable graph} if $\cal G$ is a measurable, anti-symmetric, irreflexive subset of $\cal R_a$. We view such a $\cal G$ as a collection of edges on the vertices $X$, and we write $\cal R_\cal G$ for the equivalence relation given by the connected components of $\cal G$ (ignoring edge direction). A {\bf graphing} of $\cal R_a$ is a measurable directed graph $\cal G \subseteq \cal R_a$ having the property that $\cal R_\cal G = \cal R_a$. A {\bf treeing} of $\cal R_a$ is a graphing of $R_a$ having no cycles. The action $\Gamma \acts^a (X, \mu)$ is called {\bf treeable} if $\cal R_a$ admits some treeing. Similarly, if for every finite set $H \subseteq \Gamma$ and every $\epsilon > 0$ there is a measurable directed graph $\cal G \subseteq \cal R_a$ that has no cycles and satisfies
$$\mu(\{x \in X : (h_1^a \cdot x, h_2^a \cdot x) \in \cal R_{\cal G}\ \text{ for all }h_1,h_2\in H\}) > 1 - \epsilon,$$
then the action $\Gamma \acts^a (X, \mu)$ is called {\bf approximately treeable}. Note that treeable actions are approximately treeable.

The group $\Gamma$ is called {\bf strongly treeable} if every free {\pmp} action of $\Gamma$ is treeable, and it is called {\bf treeable} if at least one free {\pmp} action of $\Gamma$ is treeable. More generally, $\Gamma$ is called {\bf approximately treeable} if at least one free {\pmp} action of $\Gamma$ is approximately treeable.

Amenable groups, free groups, finitely generated groups admitting planar Cayley graphs, elementarily free groups (that is, groups with the same first-order theory as a nonabelian free group), and the group of isometries of the hyperbolic plane and all its closed subgroups are examples of strongly treeable groups \cite{CGMTD,Gab}. It is a prominent open question whether there is a group that is treeable but not strongly treeable. As we will see in Section \ref{approximatelytreeable}, the class of approximately treeable groups includes all treeable groups and all universally free groups and is closed under increasing unions and extensions by amenable groups.

While the above are the conventional definitions for these properties, we find it more convenient to work with alternative characterizations that we now introduce.

Let $d$ be the diagonal translation action of $\Gamma$ on $\Gamma \times \Gamma$, specifically $\gamma^d \cdot (\alpha, \beta) = (\gamma \alpha, \gamma \beta)$. We similarly denote by $d$ the action of $\Gamma$ on the space $\Pow{\Gamma \times \Gamma}$ of all subsets of $\Gamma \times \Gamma$, so that $(\alpha, \beta) \in A \Leftrightarrow (\gamma \alpha, \gamma \beta) \in \gamma^d \cdot A$. We endow $\Pow{\Gamma \times \Gamma}$ with the topology given by pointwise convergence of indicator functions. For $S \in \Pow{\Gamma \times \Gamma}$ write $\bar S$ for the symmetric reflection of $S$, that is, $$\bar S = \{(\beta, \alpha) \in \Gamma \times \Gamma : (\alpha, \beta) \in S\}.$$

Let $\cal F(\Gamma) \subseteq \Pow{\Gamma \times \Gamma}$ denote the space of all directed forests on $\Gamma$, specifically, $F \subseteq \Gamma \times \Gamma$ belongs to $\cal F(\Gamma)$ precisely when $F \cap \bar F = \varnothing$ and $F \cup \bar F$ is an acyclic graph on $\Gamma$. Notice in this case that $\bar F$ is the same graph but with the direction of all edges reversed.  We also let $\cal T(\Gamma) \subseteq \cal F(\Gamma)$ denote the space of all directed trees with vertex set $\Gamma$.

Let $\cal E(\Gamma) \subseteq \Pow{\Gamma \times \Gamma}$ be the space of all equivalence relations on $\Gamma$. For $F \in \cal F(\Gamma)$ we write $E_F \in \cal E(\Gamma)$ for the equivalence relation given by the connected components of $F$ (ignoring the direction of the edges), and for $E \in \cal E(\Gamma)$ we write $\Gamma / E$ for the set of $E$-equivalence classes. Notice that $d$ provides an action of $\Gamma$ on both $\cal F(\Gamma)$ and $\cal E(\Gamma)$, and that $\gamma^d \cdot E_F = E_{\gamma^d \cdot F}$ for $\gamma \in \Gamma$ and $F \in \cal F(\Gamma)$.

Given a {\pmp} action $\Gamma \acts^a (X, \mu)$, there is a one-to-one correspondence between measurable sets $\cal Y \subseteq \cal R_a$ and equivariant measurable maps $\phi : X \rightarrow \Pow{\Gamma \times \Gamma}$ given by the rule
\[(\alpha, \beta) \in \phi(x) \Leftrightarrow ((\alpha^{-1})^a \cdot x, (\beta^{-1})^a \cdot x) \in \cal Y.\]
This correspondence identifies measurable equivalence subrelations $\cal R$ of $\cal R_a$ with equivariant measurable maps to $\cal E(\Gamma)$, identifies treeings $\cal G \subseteq \cal R_a$ with measurable equivariant maps to $\cal T(\Gamma)$, and identifies directed measurable graphs $\cal G \subseteq \cal R_a$ having no cycles with equivariant measurable maps to $\cal F(\Gamma)$. Additionally, when $\cal G \subseteq \cal R_a$ is a measurable directed graph having no cycles, if we let $\phi_\cal G$ and $\phi_{\cal R_\cal G}$ denote the maps associated with $\cal G$ and $\cal R_\cal G$, respectively, we have that $E_{\phi_\cal G(x)} = \phi_{\cal R_\cal G}(x)$ for all $x \in X$.

From these observations we obtain the following characterizations:

\begin{prop}
A group $\Gamma$ is:
\begin{enumerate}
    \item treeable if and only if there exists a $\Gamma$-invariant Borel probability measure on $\cal T(\Gamma)$;
    \item strongly treeable if and only if: for every free {\pmp} action $\Gamma \acts^a (X, \mu)$, there is a $\Gamma$-invariant Borel probability measure $\nu$ on $\cal T(\Gamma)$ such that $\Gamma \acts^a (X, \mu)$ factors onto $\Gamma \acts^d (\cal T(\Gamma), \nu)$;
    \item approximately treeable if and only if: for every finite set $H \subseteq \Gamma$ and every $\epsilon > 0$, there is a $\Gamma$-invariant measure $\mu$ on $\cal F(\Gamma)$ satisfying
\[\mu(\{F \in \cal F(\Gamma) : H \times H \not\subseteq E_F\}) < \epsilon.\]
\end{enumerate}
\end{prop}

\begin{proof}
The only implication not obvious from the above discussion is the ``if'' direction of (1); to see this, take the direct product of $\Gamma \acts^d (\cal T(\Gamma), \nu)$ with any free {\pmp} action of $\Gamma$, let $\phi$ be the projection map to $\cal T(\Gamma)$, and apply the above correspondence.
\end{proof}

We write $\Prob(\Pow{\Gamma \times \Gamma})$ for the space of all Borel probability measures on $\Pow{\Gamma \times \Gamma}$, equipped with the weak$^*$ topology, and write $\Prob(\cal F(\Gamma))$ and $\Prob(\cal T(\Gamma))$ for the subspaces of probability measures on $\cal F(\Gamma)$ and $\cal T(\Gamma)$, respectively (equipped with the subspace topoloogy). 

\begin{cor}
A group $\Gamma$ is approximately treeable if and only if: for every weak$^*$ open neighborhood $U$ of the point-mass $\delta_{\Gamma \times \Gamma} \in \Prob(\cal E(\Gamma))$, there is an invariant Borel probability measure $\mu$ on $\cal F(\Gamma)$ so that the pushforward of $\mu$ under the map $F \in \cal F(\Gamma) \mapsto E_F \in \cal E(\Gamma)$ belongs to $U$.
\end{cor}

\section{Existentially closed actions}

In this section, we gather a number of basic facts about e.c. \pmp actions of $\Gamma$.  We also consider study the Rokhlin entropy of e.c. actions and cocycles on e.c. actions.

\subsection{Definitions, first properties, and a useful reformulation}

\begin{defn}
Given actions $\Gamma \acts^a X$ and $\Gamma\acts^b Y$ with $\cal M_a\subseteq \cal M_b$, we say that $\cal M_a$ is \textbf{existentially closed} (or \textbf{e.c.} for short) \textbf{in} $\cal M_b$ if:  for any quantifier-free $L_\Gamma$-formula $\varphi(x,y)$, where $x$ and $y$ are finite tuples of variables, and any $c\in \cal M_a$ of the same length as $x$, we have $$\left(\inf_y \varphi(c,y)\right)^{\cal M_a}=\left(\inf_y \varphi(c,y)\right)^{\cal M_b}.$$  We say that $\cal M_a$ is \textbf{existentially closed} (or e.c.) if it is existentially closed in $\cal M_b$ whenever $\cal M_a\subseteq \cal M_b$.
\end{defn}

In the context of the above definition, considering the dual situation of a factor map $Y\to X$, we may also say that the factor map $Y\to X$ is existentially closed if the corresponding inclusion of models of $T_\Gamma$ is existentially closed.

The following is a useful, well-known  ``logic-free'' characterization of e.c.\ inclusions:

\begin{fact}\label{soft}
The inclusion $\cal M_a\subseteq \cal M_b$ is e.c.\ if and only if there is an ultrafilter $\u$ and an embedding $\iota:\cal M_b\to \cal M_a^\u$ such that the restriction $\iota|\cal M_a$ of $\iota$ to $\cal M_a$ is the usual diagonal embedding $\cal M_a\hookrightarrow\cal M_a^\u$.  Dually, the factor map $Y\to X$ is e.c.\ if and only if there is a factor map $X_\u\to Y$ such that the composite factor map $X_\u\to Y\to X$ is the usual diagonal factor map.
\end{fact}

The following is a well-known fact about the existence of e.c.\ actions, relativized to our current setting; see, for example, \cite[Fact 2.8]{usvy}.

\begin{fact}
E.c.\ actions exist.  In fact, given any action $\Gamma\acts^a X$, there is an e.c.\ action $\Gamma\acts^b Y$ with $\cal M_a\subseteq \cal M_b$.  Moreover, if $X$ is standard, then $Y$ can also be taken to be standard.
\end{fact}

We record the following immediate consequence of the definition of e.c.\ actions:

\begin{lem}\label{ecfree}
Suppose that $\Gamma\acts^a X$ is an e.c.\ action.  Then $X$ is an atomless probability space and $a$ is a free action.
\end{lem}

\begin{proof}
Consider the factor map $X\times [0,1]\to X$, where $\Gamma$ acts on $[0,1]$ in the trivial manner.  Since this latter action is on an atomless space and being atomless can be axiomatized using existential axioms, the first statement follows from the definition of e.c.\ actions.  (Alternatively, one can use the characterization in Fact \ref{soft} to find a factor map $X_\u\to X\times [0,1]$ such that the composite map $X_\u \to X\times [0,1]\to X$ is the diagonal map; this immediately implies that $X$ is atomless.)  Since there is also a factor map $Y\to X$ with $Y$ a free action, the same reasoning, together with the fact that being a free action is axiomatizable by existential conditions (see Subsection \ref{pmpstructures}), implies that $a$ is a free action.
\end{proof}

We end this subsection with a useful, ergodic-theoretic reformulation for a factor map to be e.c.\ that will be used throughout the paper.  First, we introduce the following notation.  Throughout this paper, we view natural numbers as ordinals, that is, $p=\{0,1,\ldots,p-1\}$ for every natural number $p$.

\begin{notation}
If $\Gamma$ acts on a set $X$, $\alpha : X \rightarrow p$ is a function, and $F\subseteq \Gamma$ is a subset, then we define the function $\alpha_F : X \rightarrow p^F$ by $\alpha_F(x)(f) = \alpha(f^{-1} \cdot x)$. 
\end{notation}

\begin{prop}\label{criterion}
The factor map $\phi:Y\to X$ is e.c.\ if and only if:  for all $p,q\in \bb N$, measurable maps $\alpha:Y\to p$, $\beta:X\to q$, finite $S\subseteq \Gamma$, and $\epsilon>0$, there is a measurable map $\tilde{\alpha}:X\to p$ such that
$$\left|\nu(\alpha_S^{-1}(\pi)\cap \phi^{-1}(\beta^{-1}(j))-\mu(\tilde{\alpha}_S^{-1}(\pi)\cap \beta^{-1}(j))\right|<\epsilon$$ for all $\pi\in p^S$ and $j\in q$. Furthermore, if $\cal A_n$ is an increasing sequence of algebras whose union is dense in $\cal B_X$, then it suffices to verify for each $n$ that the above condition holds for every $\cal A_n$-measurable map $\beta : X \rightarrow q$.
\end{prop}

\begin{proof}
This is simply translating definitions from one framework to another. Specifically, in place of discussing tuples of sets as done in the formal definition above, one can instead discuss the finitely many atoms of the Boolean algebra that those sets generate. For the final statement, each measurable map $\beta$ can be approximated by maps $\beta_n$ where $\beta_n$ is $\cal A_n$-measurable. We leave the details to the reader.
\end{proof}

\subsection{E.c. actions are maximal with respect to weak containment}

Recall from Subsection \ref{ultrasub} above that the action $\Gamma\acts^a X$ is weakly contained in the action $\Gamma\acts^b Y$ if $\cal M_a$ embeds into an ultrapower $\cal M_b^\u$ of $\cal M_b^\u$.  Recalling the statement of Fact \ref{soft}, we immediately see:

\begin{lem}\label{ecweak}
If $\cal M_a$ is e.c.\ in $\cal M_b$, then $b$ is weakly contained in $a$.
\end{lem}

In the study of weak containment of {\pmp} actions of countable groups, special attention has been paid to actions $a$ that are \textbf{maximal for weak containment}, that is, every other {\pmp} action is weakly contained in $a$; for example, see the wonderful survey article \cite[Section 5]{kechrisburton}.  Following nomenclature recently used in the model theory of operator algebras (see \cite{mtoa3}), we might also sometimes refer to such actions as \textbf{locally universal}.  As mentioned in \cite{kechrisburton}, locally universal actions always exist.  (This is a special case of a much more general fact, as argued in \cite{mtoa3}.)  The proof of Lemma \ref{ecfree} shows also that locally universal actions are free.  We note the following well-known result, relativized to our current setting:

\begin{lem} \label{lem:locuniv}
If $\cal M_a$ is an e.c.\ model of $T_\Gamma$, then $a$ is a locally universal action.
\end{lem}

\begin{proof}
Let $\Gamma\acts^b Y$ be any action of $\Gamma$.  Then since $\cal M_a\subseteq \cal M_{a\times b}$, by Lemma \ref{ecweak}, we see that $a\times b$ is weakly contained in $a$, whence $b$ is also weakly contained in $a$.  It follows that $a$ is a locally universal action.
\end{proof}

For certain groups, ``concrete'' examples of locally universal actions are known.  In particular, for so-called \textbf{EMD} groups (see Subsection \ref{profinitesub} for the definition), the profinite completion action is locally universal.  In Subsection \ref{profinitesub} below, we will generalize this fact by showing that the profinite completion action is an e.c.\ action when the group has property EMD.

\begin{lem}
For each group $\Gamma$, there is a theory $T_{\Gamma,max}$ whose models are precisely the locally universal actions of $\Gamma$.
\end{lem}  

\begin{proof}
Fix a locally universal action $\Gamma \curvearrowright^a (X,\mu)$ of $\Gamma$.  Let $T_\Gamma$ denote the set of existential sentences true in $\cal M_a$.  Since any two locally universal actions satisfy the same existential sentences, we have that all locally universal actions model $T_\Gamma$.  Conversely, if $\cal M_b$ is a model of $T_{\Gamma,max}$, then $\cal M_a$ is a model of the universal theory of $\cal M_b$, whence $\cal M_a$ embeds in an ultrapower of $\cal M_b$; since $\cal M_a$ is locally universal, so is $\cal M_b$.
\end{proof}

\subsection{Rokhlin entropy of e.c. actions}

In this subsection, we prove an interesting result about the Rohklin entropy of e.c.\ actions of non-amenable groups on standard spaces.  This result will not be used in the remainder of the paper.

Suppose that $\Gamma \acts (X, \mu)$ is an aperiodic {\pmp} action on a standard Borel probability space. For a countable Borel partition $\alpha$, we denote by $\sH(\alpha)$ the \textbf{Shannon entropy of $\alpha$}:
$$\sH(\alpha) = \sum_{A \in \alpha} - \mu(A) \log \mu(A),$$
 where we use the convention that $0 \log 0 = 0$. Similarly, when $\beta$ is a countable Borel partition of $X$ satisfying $\sH(\beta) < \infty$ and $\cF$ is a $\sigma$-algebra of Borel sets, the \textbf{relative Shannon entropies} are defined as
$$\sH(\alpha \given \beta) = \sH(\alpha \vee \beta) - 
\sH(\beta)$$
$$\sH(\alpha \given \cF) = \inf_{\beta \subseteq \cF} \sH(\alpha \given \beta),$$
where the infimum in the second line is over all finite partitions $\beta \subseteq \cF$. 

Write $\salg_\Gamma(\alpha)$ for the smallest $\Gamma$-invariant $\sigma$-algebra containing $\alpha$.  For any collection $\mathcal{C}$ of Borel susets of $X$ and any $\Gamma$-invariant $\sigma$-algebra $\cF$ of Borel sets, the \textbf{Rokhlin entropy $\rh_\Gamma(\mathcal{C} \given \cF)$ of $\mathcal{C}$ relative to $\cF$} is defined to be the infimum of $\sH(\alpha \given \cF)$, as $\alpha$ ranges over all countable Borel partitions of $X$ satisfying $\salg_\Gamma(\alpha) \vee \cF \supseteq \mathcal{C}$. When $\cF = \{\emptyset, X\}$ is trivial, we write $\rh_\Gamma(\mathcal{C})$ in place of $\rh_\Gamma(\mathcal{C} \given \cF)$. The 
\textbf{Rokhlin entropy of $\Gamma \acts (X, \mu)$}, denoted $\rh_\Gamma(X, \mu)$, is defined to be $\rh_\Gamma(\Borel(X))$, where $\Borel(X)$ is the Borel $\sigma$-algebra of $X$.

Rokhlin entropy was introduced by the second author in 2019 \cite{Sew19} and is one of two extensions of the classical Kolmogorov--Sinai entropy theory for actions of countable amenable groups. Specifically, in the case that $\Gamma$ is amenable and the action is free, Rokhlin entropy coincides with Kolmogorov--Sinai entropy \cite[Cor. 1.9]{AlpSew}. The other extension of Kolmogorov--Sinai entropy is sofic entropy, which was introduced by Lewis Bowen in 2010 \cite{Bowen10}. Although sofic entropy is more practical to compute and has been studied in greater depth, we work with Rokhlin entropy here because it has the advantage of being defined for actions of all countable groups (sofic entropy is defined only for actions of sofic groups). Additionally, Rokhlin entropy is an upper bound to sofic entropy whenever the latter is defined \cite[Prop. 1.10]{AlpSew} (see also \cite[Prop. 5.3]{Bowen10}), which means that the proposition below automatically provides the optimal result for sofic entropy as well.

\begin{prop}
Let $\Gamma \acts (X, \mu)$ be a {\pmp} action on a standard Borel probability space. If the action is e.c.\ and $\Gamma$ is non-amenable, then $\rh_\Gamma(X, \mu) = 0$.
\end{prop}

\begin{proof}
Let $(\alpha_n)_{n \in \bN}$ be an increasing sequence of finite Borel partitions whose union generates the entire Borel $\sigma$-algebra on $X$. Since Rokhlin entropy is countably subadditive \cite[Corollary 1.5]{AlpSew}, we have
$$\rh_\Gamma(X, \mu) \leq \sum_{n \in \bN} \rh_\Gamma(\alpha_n).$$
Thus it suffices to show that $\rh_\Gamma(\alpha) = 0$ for every finite Borel partition $\alpha$ of $X$.

Fix a finite Borel partition $\alpha$ of $X$. Since $\Gamma$ is non-amenable, by a result of Bowen \cite{bowen} there exists a free {\pmp} action $\Gamma \acts (Y, \nu)$ with $\rh_\Gamma(Y, \nu) = 0$ for which there is a factor map $\phi : Y \rightarrow X$. Fix $\epsilon > 0$ and pick a countable Borel partition $\beta$ of $Y$ satisfying $\sH(\beta) < \epsilon$ and $\salg_\Gamma(\beta) = \Borel(Y)$. Since $\phi^{-1}(\alpha) \subseteq \salg_\Gamma(\beta)$ we have
$$0 = \sH(\phi^{-1}(\alpha) \given \salg_\Gamma(\beta)) = \inf_{\substack{F \subseteq \Gamma\\F \text{ finite}}} \sH\left(\phi^{-1}(\alpha) \Given \bigvee_{g \in F} g \cdot \beta\right),$$
where the second equality is a basic property of Shannon entropy (see \cite[Lemma 1.7.11]{Dow11}). So there is a finite set $F \subseteq \Gamma$ with $\sH(\phi^{-1}(\alpha) \given \bigvee_{g \in F} g \cdot \beta) < \epsilon$. Since $\Gamma \acts (X, \mu)$ is existentially closed, there must exist a finite Borel partition $\beta'$ of $X$ satisfying $\sH(\beta') < \epsilon$ and $\sH(\alpha \given \bigvee_{g \in F} g \cdot \beta') < \epsilon$. Consequently, by subadditivity of Rokhlin entropy,
$$\rh_\Gamma(\alpha) \leq \rh_\Gamma(\beta') + \rh_\Gamma(\alpha \given \salg_\Gamma(\beta')) \leq \sH_\mu(\beta') + \sH_\mu \left( \alpha \Given \bigvee_{g \in F} g \cdot \beta' \right) < 2 \epsilon.$$
As $\epsilon$ was arbitrary, we conclude that $\rh_\Gamma(\alpha) = 0$.
\end{proof}

If $\Gamma$ is an amenable group and $r \in (0, +\infty]$, then there exists a free {\pmp} action of $\Gamma$ on a standard Borel probability space having Kolmogorov--Sinai entropy equal to $r$ (for instance, any Bernoulli shift over $\Gamma$ whose base space has Shannon entropy equal to $r$). Such an action would be e.c.\ since all free actions of amenable groups are e.c.\, and it would have Rokhlin entropy $r$ since Rokhlin entropy and Kolmogorov--Sinai entropy coincide for free actions of amenable groups \cite[Corollary 1.9]{AlpSew}. The assumption in the above proposition that $\Gamma$ be non-amenable is therefore necessary.

\subsection{Cocycles on e.c. actions} \label{subs:ec_cocycles}

\begin{lem}\label{eccocycle}
Suppose that $\Gamma\acts^a(X,\mu)$ is e.c.\  Then, for any finite group $K$, $B^1(a,K)$ is dense in $Z^1(a,K)$.
\end{lem}

\begin{proof}
Let $Y$ be the skew-product extension $X\times_\sigma K$.  If $p_0:Y\to K$ is the projection map, then we have $p_0(\gamma\cdot (x,k))p_0(x,k)^{-1}=\sigma(\gamma,x)$ for all $\gamma$, $x$, and $k$.  Thus, since $\Gamma\acts^a(X,\mu)$ is e.c.\, for any finite $F\subseteq \Gamma$ and $\epsilon>0$, there is a map $p:X\to K$ such that $$\mu(\{x\in X \ : \ p(\gamma x)p(x)^{-1}=\sigma(\gamma,x) \text{ for all }\gamma\in F\})>1-\epsilon,$$ proving the lemma.
\end{proof}

The above fact can be strengthened in the case where $\Gamma$ has property (T).

\begin{cor}
If $\Gamma$ has property (T), then for any e.c.\ action $\Gamma\acts^a(X,\mu)$ and any finite group $K$, we have $H^1(a,K)=0$.
\end{cor}

\begin{proof}
    For ergodic actions of property (T) groups, its known that $B^1(a,K)$ is clopen in $Z^1(a,K)$ \cite[Lemma 4.2]{Po07} (alternatively see \cite[Theorem 4.2]{Fu07}). If our e.c.\ action were ergodic we could combine this with Lemma \ref{eccocycle} and be done. However, e.c.\ actions of property (T) groups are never ergodic, so we instead adapt the proof of \cite[Theorem 4.2]{Fu07} to the non-ergodic setting.
    
	Since $\Gamma$ has property (T), there is a finite set $F \subseteq \Gamma$ and $\epsilon > 0$ such that for every unitary representation $\pi : \Gamma \rightarrow \mathcal{U}(\mathcal{H})$ and every unit vector $\xi_0$, if $|\langle \pi(\gamma)(\xi_0), \xi_0 \rangle| \geq 1 - \epsilon$ for every $\gamma \in F$ then there is a $\Gamma$-invariant unit vector $\xi \in \mathcal{H}$ with $\|\xi - \xi_0\| < 1/16$.
	
	We claim that for every measurable $\Gamma$-invariant set $Y \subseteq X$ of positive measure and cocycles $\alpha, \beta : \Gamma \times X \rightarrow K$ satisfying, for every $\gamma\in F$:
	$$\mu(\{x \in X : \alpha(\gamma^{-1},x) \neq \beta(\gamma^{-1}, x)\}) \leq \epsilon \mu(Y)$$
	there is a measurable $\Gamma$-invariant set $Z \subseteq Y$ and a measurable function $f : Z \rightarrow K$ satisfying $\mu(Z) \geq \mu(Y) / 2$ and $\alpha(\gamma, x) = f(\gamma x) \beta(\gamma, x) f(x)^{-1}$ for a.e. $x \in Z$ and every $\gamma \in \Gamma$.
	
	Assume for now that the claim holds. Set $X_0 = \emptyset$ and inductively assume that $X_m$ has been defined for all $m \leq n$. If $\mu(\bigcup_{m \leq n} X_m) = 1$ then the induction can stop, but otherwise it proceeds as follows. By Lemma \ref{eccocycle} $B^1(a,K)$ is dense in $Z^1(a,K)$, so we can pick a cocycle $\sigma_{n+1} \in B^1(a,K)$ satisfying, for every $\gamma\in F$:
	$$\mu(\{x \in X : \sigma(\gamma^{-1},x) \neq \sigma_{n+1}(\gamma^{-1}, x)\}) \leq \epsilon \cdot \mu \left(X \setminus \bigcup_{m \leq n} X_m \right).$$
	Since $\sigma_{n+1}$ is a coboundary, we can pick a measurable function $h_{n+1} : X \rightarrow K$ satisfying $\sigma_{n+1}(\gamma, x) = h_{n+1}(\gamma x) h_{n+1}(x)^{-1}$ for a.e. $x$ and every $\gamma$. Next apply the claim of the previous paragraph to $Y = X \setminus \bigcup_{m \leq n} X_m$, $\alpha = \sigma$, and $\beta = \sigma_{n+1}$ to obtain a $\Gamma$-invariant measurable set $X_{n+1} \subseteq X \setminus \bigcup_{m \leq n} X_m$ with
	$$\mu(X_{n+1}) \geq \frac{1}{2} \mu \left( X \setminus \bigcup_{m \leq n} X_m \right)$$
	and a measurable function $f_{n+1} : X_{n+1} \rightarrow K$ such that
	$$\sigma(\gamma, x) = f_{n+1}(\gamma x) \sigma_{n+1}(\gamma, x) f_{n+1}(x)^{-1}$$
	for a.e. $x \in X_{n+1}$ and every $\gamma \in \Gamma$. Then the sets $X_n$ will be pairwise disjoint and their union will be conull and the function $f : \bigcup_n X_n \rightarrow K$ defined by $f(x) = f_n(x) h_n(x)$ for $x \in X_n$ will satisfy $\sigma(\gamma, x) = f(\gamma x) f(x)^{-1}$ for a.e. $x$ and every $\gamma$, implying that $\sigma$ is a coboundary as desired.
	
	We now prove the claim. Let $Y$, $\alpha$ and $\beta$ be as described. Define an action of $\Gamma$ on $Y \times K$ by $\gamma \cdot (x, k) = (\gamma x, \alpha(\gamma, x) k \beta(\gamma, x)^{-1})$. Set $\mathcal{H} = L^2(Y \times K, \frac{1}{\mu(Y)} \mu \times c)$ where $c$ is the counting measure on $K$, and let $\pi : \Gamma \rightarrow \mathcal{U}(\mathcal{H})$ be the unitary representation $\pi(\gamma)(\eta)(x,k) = \eta(\gamma^{-1} \cdot (x,k))$. Let $\xi_0$ be the unit vector $1_{Y \times \{e_K\}}$ and observe that for $\gamma \in F$
	$$|\langle \pi(\gamma)(\xi_0), \xi_0 \rangle| = \frac{1}{\mu(Y)} \mu(\{x \in Y : \alpha(\gamma^{-1}, x) = \beta(\gamma^{-1}, x)\}) \geq 1-\epsilon.$$
	It follows from our choice of $F$ and $\epsilon$ that there is a $\Gamma$-invariant unit vector $\xi$ satisfying $\|\xi - \xi_0\| < 1/16$.
	
	Define $Z$ to be the set of $x \in Y$ for which there is a unique $k \in K$ maximizing the value of $|\xi(x,k)|$ and, in this case, define $f(x)$ to be that unique element of $K$. Note that since $f^{-1}(k_0) = \bigcap_{k \in K \setminus \{k_0\}} \{x \in Y : |\xi(x,k)| < |\xi(x,k_0)|\}$ and $Z = \bigcup_{k_0 \in K} f^{-1}(k_0)$, both $f$ and $Z$ are measurable. The invariance of $\xi$ tells us that
	$$\xi(\gamma x, k) = \pi(\gamma^{-1})(\xi)(x, \alpha(\gamma,x)^{-1} k \beta(\gamma,x)) = \xi(x, \alpha(\gamma,x)^{-1} k \beta(\gamma, x)),$$
	and since the map $k \in K \mapsto \alpha(\gamma,x)^{-1} k \beta(\gamma, x)$ is a permutation of $K$, we see that $Z$ is $\Gamma$ invariant and that $\alpha(\gamma,x)^{-1} f(\gamma x) \beta(\gamma, x) = f(x)$ for all $x \in Z$.
	
	Finally, it only remains to check that $\mu(Z) \geq \mu(Y) / 2$. Consider the sets
	$$D_1 = \left\{x \in Y : |1 - \xi(x,e_K)|^2 \geq \frac{1}{4} \right\}$$
	$$D_2 = \left\{x \in Y : \sum_{k \in K \setminus \{e_K\}} |\xi(x,k)|^2 \geq \frac{1}{4} \right\}.$$
	If $x \in Y \setminus (D_1 \cup D_2)$ then $|\xi(x,e_K)|^2 \geq \frac{1}{4}$ while $\sum_{k \in K \setminus \{e_K\}} |\xi(x,k)|^2 < \frac{1}{4}$. So $Y \setminus (D_1 \cup D_2) \subseteq f^{-1}(e_K) \subseteq Z$. Since
	$$\frac{1}{\mu(Y)} \int_Y \left(|1-\xi(x,e_K)|^2 + \sum_{k \in K \setminus \{e_K\}} |\xi(x,k)|^2\right) \ d \mu = \|\xi - \xi_0\|^2 < \frac{1}{16},$$
	we have that $\mu(D_1) < \mu(Y) / 4$ and $\mu(D_2) < \mu(Y) / 4$ and therefore
	\begin{equation*}
		\mu(Z) \geq \mu(Y \setminus (D_1 \cup D_2)) \geq \mu(Y) / 2.\qedhere
	\end{equation*}
\end{proof}

The previous corollary can be extended to all groups at the expense of restricting to e.c.\ actions which are ultraproducts.

\begin{cor}
	If $\Gamma\acts^{a_i}(X_i,\mu)$ is a family of actions and the nonprincipal ultraproduct action $\Gamma\acts^a (X,\mu)$ is e.c.\, then $H^1(a,K)=0$ for every finite group $K$.
\end{cor}

This immediately follows from Lemma \ref{eccocycle} and the following general observation below.

\begin{lem}\label{ultraclosed}
	Suppose that $\Gamma \acts^a(X,\mu)$ is a nonprincipal ultraproduct action.  Then $B^1(a,K)$ is closed in $Z^1(a,K)$ for any finite group $K$.
\end{lem}

\begin{proof}
	Suppose that $\u$ is a nonprincipal ultrafilter on a set $I$, $\Gamma \acts^{a_i} (X_i, \mu_i)$ is a {\pmp} action for every $i \in I$, $(X, \mu) = \prod_\u (X_i, \mu_i)$ and $a = \prod_\u a_i$. Let $\sigma : \Gamma \times X \rightarrow K$ belong to the closure of $B^1(a, K)$ in $Z^1(a, K)$. Choose an increasing sequence of finite sets $W_n \subseteq \Gamma$ with $\bigcup_{n \in \bN} W_n = \Gamma$ and for each $n \in \bN$ pick a cocycle $\sigma_n \in B^1(a,K)$ satisfying
	$$ \mu(S_{\gamma,k} \symd S_{n,\gamma,k}\}) < 2^{-n} \text{ for all }\gamma \in W_n \text{ and } k \in K, $$
	where $S_{\gamma, k} = \{x \in X : \sigma(\gamma, x) = k\}$ and $S_{n,\gamma,k} = \{x \in X : \sigma_n(\gamma, x) = k\}$.
	
	For each $n \in \bN$, pick a measurable function $f_n : X \rightarrow K$ satisfying $f_n(\gamma x) f_n(x)^{-1} = \sigma_n(\gamma, x)$ for a.e. $x \in X$ and every $\gamma \in \Gamma$. Choose measurable functions $f_n^i : X_i \rightarrow K$ with $f_n^{-1}(k) = [(f_n^i)^{-1}(k)]_\u$ for every $k \in K$ and define the cocycle $\sigma_n^i(\gamma, x) = f_n^i(\gamma x) f_n^i(x)^{-1}$. Set $S_{n,\gamma,k}^i = \{x \in X_i : \sigma_n^i(\gamma, x) = k\}$.
	
	Choose measurable sets $S_{\gamma,k}^i \subseteq X_i$ satisfying $S_{\gamma,k} = [S_{\gamma,k}^i]_\u$ for all $\gamma \in \Gamma$ and $k \in K$. Since $\u$ is nonprincipal, we can fix a function $M : I \rightarrow \bN$ satisfying $\lim_\u M(i) = \infty$. For each $i \in I$, define
	$$m(i) = \max\{n \leq M(i) :  \mu_i(S_{n,\gamma,k}^i \symd S_{\gamma,k}^i) \leq 2^{-n+1} \text{ for all }\gamma \in W_n \text{ and } k \in K\}.$$
	Since for every $n \in \bN$ we have $\{i : M(i) \geq n\} \in \u$ and
	$$\lim_\u \mu_i(S_{n,\gamma,k}^i \symd S_{\gamma,k}^i) = \mu(S_{n,\gamma,k} \symd S_{\gamma,k}) < 2^{-n} \text{ for all }\gamma \in W_n \text{ and } k \in K,$$
	it is immediately seen that $\{i : m(i) \geq n\} \in \u$ and thus $\lim_\u m(i) = \infty$.
	
	Since $\lim_\u m(i) = \infty$ and $\bigcup_n W_n = \Gamma$, we have that $\lim_\u \mu_i(S_{m(i),\gamma,k}^i \symd S_{\gamma,k}^i) \leq \lim_\u 2^{-m(i)+1} = 0$ for every $\gamma \in \Gamma$ and $k \in K$. Therefore $S_{\gamma,k} = [S_{\gamma,k}^i]_\u = [S_{m(i),\gamma,k}^i]_\u$ for all $\gamma \in \Gamma$ and $k \in K$. Consequently,
	$$\sigma(\gamma, [x_i]_\u) = \lim_\u \sigma_{m(i)}^i(\gamma,x_i) = \lim_\u f_{m(i)}^i(\gamma x_i) f_{m(i)}^i(x_i)^{-1},$$
	and defining $h([x_i]_\u) = \lim_\u f_{m(i)}^i(x_i)$ we have $h : X \rightarrow K$ is measurable and $\sigma(\gamma, x) = h(\gamma x) h(x)^{-1}$ for a.e. $x$ and all $\gamma \in \Gamma$. We conclude that $\sigma \in B^1(a,K)$.\qedhere
\end{proof}	

\section{Special e.c. actions}\label{special}

In this section, we study when the profinite completion action is e.c. and when there is a weakly mixing e.c. action.  In the process, we prove a fact of independent interest, namely that limit groups have Kechris' property MD.

\subsection{Profinite completions}\label{profinitesub}
In \cite{BIH}, the authors show that the natural action of a finitely generated free group $\bb F$ on its profinite completion $\hat {\bb F}$ is an e.c.\ action.  In this section, we generalize this result to the largest class of groups for which it could possibly hold.

First, recall that the profinite completion $\hat{\Gamma}$ of a countable residually finite group $\Gamma$ is the inverse limit of the finite groups $\Gamma / \Lambda$ as $\Lambda$ varies over the normal finite-index subgroups of $\Gamma$. The profinite completion $\hat{\Gamma}$ is a compact group and thus admits a unique Haar probability measure $\mu_{\hat \Gamma}$. $\Gamma$ naturally embeds into $\hat{\Gamma}$ and thus acting by left-translation yields an ergocdic p.m.p. action of $\Gamma$ on $(\hat{\Gamma}, \mu_{\hat \Gamma})$. For each normal finite-index subgroup $\Lambda \lhd \Gamma$, the closure $\bar{\Lambda}$ of $\Lambda$ in $\hat{\Gamma}$ is a finite-index clopen subgroup of $\hat{\Gamma}$, and the partition $\cal C_{\bar{\Lambda}}$ of $\hat{\Gamma}$ into its left $\bar{\Lambda}$-cosets is $\Gamma$-invariant (that is, $\Gamma$ permutes the cosets). Moreover, there is a decreasing sequence $\Lambda_n$ of finite-index normal subgroups of $\Gamma$ such that the sequence $\cal C_{\bar{\Lambda_n}}$ separates points.

An action $\Gamma\acts^a (X,\mu)$ is called \textbf{profinite} if there is a decreasing sequence of finite $\Gamma$-invariant measurable partitions of $X$ which separate points.  Following Kechris \cite{kechrisweak}, a residually finite group $\Gamma$ is said to be \textbf{MD} if the set of profinite actions of $\Gamma$ on $(X,\mu)$ is dense in the space $A(\Gamma,X,\mu)$ of all actions of $\Gamma$.  The group $\Gamma$ is said to have the a priori stronger property \textbf{EMD} if the set of ergodic profinite actions of $\Gamma$ on $(X,\mu)$ is dense in $A(\Gamma,X,\mu)$.  It is an open question whether or not the two notions coincide for all groups, but by work of the third author (\cite[Corollary 4.7 and Theorem 4.10]{TD}), we have that they coincide for all groups without property (T) and that they coincide for all groups if and only if property MD implies the negation of property (T).  For examples and closure properties of these (somewhat mysterious) classes of groups, see \cite[Section 5]{kechrisburton}.

Kechris \cite[Propositions 4.2, 4.5, and 4.8]{kechrisweak} showed that a group $\Gamma$ has property EMD precisely when its action on its profinite completion $\hat\Gamma$ is locally universal while it has property MD precisely when its action on $\hat \Gamma\times [0,1]$ is locally universal (where the action on the second coordinate is trivial).  In this section, we show that for these classes of groups, the associated actions are in fact existentially closed.  (Since being existentially closed implies being locally universal, these results are optimal.)

\begin{lem} \label{lem:MDsubgroup}
	Let $\Gamma$ be a countable group and let $\Lambda \leq \Gamma$ be a subgroup.
	\begin{enumerate}
		\item If $\Gamma$ has property MD, then $\Lambda$ has MD as well.
		\item If $\Gamma$ has property EMD and $\Lambda$ has finite index in $\Gamma$, then $\Lambda$ has EMD as well.
	\end{enumerate}
\end{lem}

\begin{proof}
	(1) is observed in \cite[Section 4]{kechrisweak}. Finally, (2) follows from (1) together with the following facts: no group with property (T) can have property EMD \cite[Proposition 6]{kechrisweak}; for groups without property (T), EMD and MD are equivalent \cite[Corollary 4.7]{TD}; and, since $\Lambda$ is a finite index subgroup of $\Gamma$, $\Lambda$ has property (T) if and only if $\Gamma$ does.
\end{proof}

The following is the main result of this section:
	
\begin{thm}\label{thm:MD_EMD_ec_profinite}
	Let $\Gamma$ be a countable residually finite group and let $(\hat{\Gamma}, \mu_{\hat{\Gamma}})$ denote the profinite completion of $\Gamma$ equipped with its normalized Haar probability measure. Finally, let $\lambda$ denote Lebesgue measure on $[0,1]$.
	\begin{enumerate}
		\item If $\Gamma$ has property EMD, then the action $\Gamma \acts (\hat{\Gamma}, \mu_{\hat{\Gamma}})$ is existentially closed.
		\item If $\Gamma$ has property MD, then the action $\Gamma \acts (\hat{\Gamma} \times [0, 1], \mu_{\hat{\Gamma}} \times \lambda)$ is existentially closed.
	\end{enumerate}
\end{thm}

\begin{proof}
	In case (1), set $(X, \mu) = (\hat{\Gamma}, \mu_{\hat{\Gamma}})$ and in case (2), set $(X, \mu) = (\hat{\Gamma} \times [0, 1], \mu_{\hat{\Gamma}} \times \lambda)$.

    Let $\Gamma \acts (Y, \nu)$ be a {\pmp} action and $\phi : Y \rightarrow X$ a $\Gamma$-equivariant factor map. Also let $p \in \bN$ and $\alpha : Y \rightarrow p$ be measurable. Let $S \subseteq \Gamma$ be finite with $e \in S$ and let $\epsilon > 0$.

    Let $H_n$ be the intersection of all subgroups of $\Gamma$ having index at most $n$. Write $\bar{H}_n$ for the closure of $H_n$ in $\hat{\Gamma}$. Let $\cal B_{[0,1]}$ be the Borel $\sigma$-algebra on $[0,1]$. In case (1) set $\cal A_n = \{\gamma \bar{H}_n : \gamma \in \Gamma\}$ and in case (2) set $\cal A_n = \{\gamma \bar{H}_n \times B : \gamma \in \Gamma, \ B \in \cal B\}$. Then the $\cal A_n$'s are an increasing sequence of algebras whose union is dense in $X$. So the stronger form of Proposition \ref{criterion} says that its enough to consider partitions of $X$ which are $\cal A_n$-measurable for some $n$.

    Fix $n$, set $H = H_n$ and $\bar{H} = \bar{H}_n$. Pick a choice $r : \Gamma / H \rightarrow \Gamma$ of representatives for the cosets of $H$ in $\Gamma$ with $r(H) = e$, and let $\rho : (\Gamma / H) \times \Gamma \rightarrow H$ be the cocycle $\rho(a H, \gamma) = r(a H) \gamma r(a \gamma H)^{-1}$. Define the finite set $T = \{r(a H) : a \in \Gamma\}$. In case (1) we let $q = 1 = \{0\}$ and let $\beta : X \rightarrow q$ be the constant function, and in case (2) we let $q \in \bN$ and let $\beta : X \rightarrow q$ be a $\{\hat{\Gamma} \times B : B \in \cal B\}$-measurable map satisfying $\mu(\beta^{-1}(j)) > 0$ for every $j \in q$. Note that $\beta$ is $\Gamma$-invariant.

    For $t \in T$ and $j \in q$, set $X_t^j = \beta^{-1}(j) \cap t^{-1} \bar{H} \in \cal A_n$, and notice that these sets partition $X$. Also set $X_t = t^{-1} \bar{H} = \bigcup_{j \in q} X_t^j$. By Proposition \ref{criterion} we will be done if we can find a measurable function $\tilde{\alpha} : X \rightarrow p$ with the property that for every $\pi \in p^S$, $j \in q$, and $t \in T$
	$$\left| \mu(\tilde{\alpha}_S^{-1}(\pi) \cap X_t^j) - \nu(\alpha_S^{-1}(\pi) \cap \phi^{-1}(X_t^j) \right| < \epsilon.$$
	
	Consider the functions $\alpha_{t S}$, $t \in T$. We always have $t \cdot \alpha_S(y) = \alpha_{t S}(t \cdot y)$ since for $s \in S$
	$$(t \cdot \alpha_S(y))(t s) = \alpha_S(y)(s) = \alpha(s^{-1} \cdot y) = \alpha(s^{-1} t^{-1} t \cdot y) = \alpha_{t S}(t \cdot y)(t s).$$
	Therefore for all $t \in T$ and $\pi \in p^S$ we have $t \cdot \alpha_S^{-1}(\pi) = \alpha_{t S}^{-1}(t \cdot \pi)$ . 
    In particular, for every $j \in q$, we have
	\begin{equation} \label{eqn:md_ec1}
		t \cdot \Big( \alpha_S^{-1}(\pi) \cap \phi^{-1}(X_t^j) \Big) = \alpha_{t S}^{-1}(t \cdot \pi) \cap \phi^{-1}(X_{e}^j).
	\end{equation}
	Similarly, since $e \in S$, for $t \in T$, $s \in S$, and $y \in Y$, we have
	\begin{align*}
		\alpha_{t S}(y)(t s) = \alpha(s^{-1} t^{-1} \cdot y) & = \alpha(r(t s H)^{-1} \rho(t H, s)^{-1} \cdot y)\\
		& = \alpha_{r(t s H) S}(\rho(t H, s)^{-1} \cdot y)(r(t s H)).
	\end{align*}
	So for every $t \in T$ and $j \in q$, we have
	\begin{equation} \label{eqn:md_ec2}
		\bigcup_{s \in S} \{y \in \phi^{-1}(X_{e}^j): \alpha_{r(t s H) S}(\rho(t H, s)^{-1} \cdot y)(r(t s H)) \neq \alpha_{t S}(y)(t s)\} = \emptyset.
	\end{equation}
	
	For $j \in q$, let $\mu_j$ denote the normalized restriction of $\mu$ to $X_{e}^j$, and similarly define $\nu_j$ to be the normalized restriction of $\nu$ to $\phi^{-1}(X_{e}^j)$. The profinite completion of $H$ is isomorphic to $\bar{H}$ and its normalized Haar probability measure $\mu_{\bar{H}}$ coincides with the normalized restriction of $\mu_{\hat{\Gamma}}$ to $\bar{H}$. Notice that in case (1) $j \in q$ can only have value $0$ and $H \acts (X_{e}^0, \mu_0)$ is isomorphic to $H \acts (\bar{H}, \mu_{\bar{H}})$, and in case (2) $H \acts (X_{e}^j, \mu_j)$ is isomorphic to $H \acts (\bar{H} \times [0, 1], \mu_{\bar{H}} \times \lambda)$ for every $j \in q$. It follows from the assumptions of cases (1) and (2) and Lemma \ref{lem:MDsubgroup} that the action $H \acts (X_{e}^j, \mu_j)$ weakly contains all $H$-actions for every $j \in q$. Consequently, we can find measurable functions $\gamma_t : X_{e} \rightarrow p^{t S}$ for $t \in T$ satisfying the following two conditions. First, relative to each of the sets $X_{e}^j$, the $\gamma_t$'s will have distribution in measure close to the $\alpha_{t S}$'s, meaning that for all $t \in T$, $j \in q$, and $\pi \in p^S$, we have
	\begin{equation} \label{eqn:md_ec3}
		\left|\mu_j \Big( \gamma_t^{-1}(t \cdot \pi) \cap X_{e}^j \Big) - \nu_j \Big( \alpha_{t S}^{-1}(t \cdot \pi) \cap \phi^{-1}(X_{e}^j) \Big)\right| < |\Gamma : H| \epsilon / 2.
	\end{equation}
	Second, we control how the functions $\gamma_t$ relate to the action of $H$ and demand, in view of (\ref{eqn:md_ec2}), that $\mu_j(D_t^j) < |\Gamma : H| \epsilon / 2$ for all $t \in T$ and $j \in q$, where
	\begin{equation*}
		D_t^j = \bigcup_{s \in S} \{x \in X_{e}^j:  \gamma_{r(t s H)}(\rho(t H, s)^{-1} \cdot x)(r(t s H)) \neq \gamma_t(x)(t s)\}.
	\end{equation*}
	
	Define $\tilde{\alpha} : X \rightarrow p$ by setting $\tilde{\alpha}(x) = \gamma_t(t \cdot x)(t)$ when $t \in T$ and $x \in X_t$. Notice that when $x \in X_t^j \setminus t^{-1} \cdot D_t^j$ we have $t \cdot \tilde{\alpha}_S(x) = \gamma_t(t \cdot x)$, since for any $s \in S$ we have $s^{-1} \cdot x \in X_{r(t s H)}$ and
	\begin{align*}
		(t \cdot \tilde{\alpha}_S(x))(t s) = \tilde{\alpha}_S(x)(s) = \tilde{\alpha}(s^{-1} \cdot x) & = \gamma_{r(t s H)}(r(t s H) s^{-1} \cdot x)(r(t s H))\\
		& = \gamma_{r(t s H)}(\rho(t H, s)^{-1} t \cdot x)(r(t s H))
	\end{align*}
	and the final term above is equal to $\gamma_t(t \cdot x)(t s)$ since $t \cdot x \not\in D_t^j$. It follows that
	\begin{equation} \label{eqn:md_ec4}
		\Big( t \cdot ( \tilde{\alpha}_S^{-1}(\pi) \cap X_t^j ) \Big) \symd \Big( \gamma_t^{-1}(t \cdot \pi) \cap X_{e}^j \Big) \subseteq D_t^j.
	\end{equation}
	
	For $\pi \in p^S$ and $j \in q$, equation (\ref{eqn:md_ec1}) implies that
	\begin{equation*}
		\nu(\alpha_S^{-1}(\pi) \cap \phi^{-1}(X_t^j)) = \nu(\alpha_{t S}^{-1}(t \cdot \pi) \cap \phi^{-1}(X_{e}^j))
	\end{equation*}
	and equation (\ref{eqn:md_ec4}) implies
	\begin{equation*}
		\left| \mu(\tilde{\alpha}_S^{-1}(\pi) \cap X_t^j) - \mu(\gamma_t^{-1}(t \cdot \pi) \cap X_{e}^j) \right| \leq \mu(D_t^j) = |\Gamma : H|^{-1} \mu_j(D_t^j) < \epsilon.
	\end{equation*}
	Since $\mu_j$ and $\nu_j$ are the normalized restrictions of $\mu$ and $\nu$ to $X_{e}^j$ and $\phi^{-1}(X_{e}^j)$, respectively, and since $\mu(X_{e}^j) = \nu(\phi^{-1}(X_{e}^j))$, it follows from (\ref{eqn:md_ec3}) and the above two equations that
	$$\left| \nu(\alpha_S^{-1}(\pi) \cap \phi^{-1}(X_t^j)) - \mu(\tilde{\alpha}_S^{-1}(\pi) \cap X_t^j) \right| < \epsilon.$$
	We conclude that the action $\Gamma \acts (X, \mu)$ is existentially closed.
\end{proof}

Since the finitely generated free group $\bb F$ has property EMD (see \cite[Theorem 1]{kechrisweak} and, using different terminology, Bowen \cite{Bowen03}), the previous theorem generalizes \cite[Theorems 6.7 and 6.18]{BIH}.

The following is a nice application of Theorem \ref{thm:MD_EMD_ec_profinite}:

\begin{prop}\label{MDcriterion}
Suppose that $\Gamma$ has a coamenable normal subgroup $\Lambda$ such that $\Gamma/\Lambda$ is residually finite.  Further suppose that $\Lambda$ can be written as the increasing union of a sequence $(\Lambda_n)_{n\in \bb N}$ of subgroups such that:
\begin{itemize}
    \item Each $\Lambda_n$ has property MD;
    \item Each finite index subgroup of each $\Lambda_n$ is closed in the profinite topology on $\Gamma$.
\end{itemize}
Then $\Gamma$ has property MD.
\end{prop}

\begin{proof}
Since each finite index subgroup of each $\Lambda _n$ is closed in the profinite topology of $\Gamma$, the action of $\Lambda _n$ on $\hat{\Gamma}$ is isomorphic to the product of the action of $\Lambda _n$ on $\widehat{\Lambda}_n$ with an identity action of $\Lambda _n$, so this action of $\Lambda _n$ is existentially closed by Theorem \ref{thm:MD_EMD_ec_profinite}. Thus, the action of $\Lambda$ on $\hat{\Gamma}$ is existentially closed as well. The action of $\Gamma$ on $\hat{\Gamma}$ factors onto $\hat{\Gamma}/\bar{\Lambda}$, where $\bar{\Lambda}$ denotes the closure of $\Lambda$ in $\hat{\Gamma}$. Since $\Gamma /\Lambda$ is residually finite, the action of $\Gamma /\Lambda$ on $\hat{\Gamma}/\bar{\Lambda}$ is free. The action of $\Gamma$ on $\hat{\Gamma}$ therefore satisfies the hypotheses of Theorem \ref{thm:lift}, hence it is existentially closed. In particular, $\Gamma$ has property MD.
\end{proof}

A consequence of the previous proposition is the following, expanding the collection of examples of groups known to have property MD:

\begin{thm}\label{thm:Limit_MD}
Limit groups have property MD.
\end{thm}

\begin{proof}
Let $\Gamma$ be a limit group. A result of Kochloukova\cite[Corollary B]{Kochloukova} is that $\Gamma$ has a free normal subgroup $\Lambda$ such that $\Gamma/\Lambda$ is torsion-free nilpotent (and, in particuar, residually finite).  Write $\Lambda$ as a union of an increasing sequence of finitely generated free subgroups $\Lambda_n$, whence each $\Lambda_n$ has property MD.  By a result of Wilton \cite{wilton}, limit groups are subgroup separable, meaning that each finitely generated subgroup of $\Gamma$ is closed in the profinite topology of $\Gamma$.  In particular, each finite index subgroup of each $\Lambda_n$ is closed in the profinite topology of $\Gamma$.  Hence $\Gamma$ has property MD by Proposition \ref{MDcriterion}.
\end{proof}

\begin{remark}
Theorem \ref{thm:Limit_MD} implies that limit groups have property FD, the representation theoretic analogue of MD introduced by Lubotzky and Shalom in \cite{Lubotzky-Shalom} (note that FD was introduced prior to MD). Property MD implies property FD by \cite{kechrisweak}, although the converse is open. Property FD for limit groups could also be deduced from \cite{Kochloukova} and \cite{wilton} by appealing to \cite[Corollary 2.5]{Lubotzky-Shalom}. While our Proposition \ref{MDcriterion} is an analogue of \cite[Corollary 2.5]{Lubotzky-Shalom}, its proof is conceptually a bit different, since it makes critical use of existentially closed actions. One may also give a somewhat ad hoc proof of Proposition \ref{MDcriterion}, avoiding the use of existentially closed actions, that more closely parallels the proof of \cite[Corollary 2.5]{Lubotzky-Shalom}, by using an approach similar to \cite{BTD2}.
\end{remark}

\subsection{Weakly mixing e.c. actions} \label{sub:weakmixing}

Recall that the {\pmp} action $\Gamma\acts^a X$ is \textbf{weakly mixing} if the product action $\Gamma\acts^{a\times a} X\times X$ is ergodic.  It follows from \cite[Theorems 4.3 and 6.6]{BIH} and \cite{KP} that there is a weakly mixing e.c.\ action of the free group.  In this subsection, we generalize this result to the case of any group without property (T) and in fact show that the ``model-theoretically generic'' e.c.\ action (in a sense we make precise below) is weakly mixing.  First, we need the following result, which is somewhat implicit in Bergelson's \cite{bergelson}: 

\begin{prop}
The action $\Gamma\acts X$ is weakly mixing if and only if:  for any measurable sets $A_1,\ldots,A_n,B_1,\ldots,B_n\subseteq X$ and any $\epsilon>0$, we have $$\bigcap_{i=1}^n\{\gamma\in \Gamma \ : \ |\mu(A_i\cap \gamma B_i)-\mu(A_i)\mu(B_i)|<\epsilon\}\not=\emptyset.$$
\end{prop}

\begin{proof}
First suppose that the action is weakly mixing.  By \cite[Theorem 4.7]{bergelson}, each set appearing in the above intersection is a central* subset of $\Gamma$.  Since the family of central* subsets of $\Gamma$ has the finite intersection property, we see that the above intersection is nonempty.

Now suppose that the above condition holds.  In order to show that the action is weakly mixing, by \cite[Exercise 21]{bergelson}, it suffices to show:  for any $f_1,\ldots,f_n\in L^2_0(X)$ and any $\epsilon>0$, we have
$$\bigcap_{i=1}^n\{\gamma\in \Gamma \ : |\langle U_\gamma f_i,f_i\rangle|<\epsilon\}\not=\emptyset.$$ Here, $U_\gamma$ is the Koopman representation associated to the action and $L^2_0(X)$ is the orthogonal complement of the subspace of $L^2(X)$ consisting of vectors of integral $0$.  For each $i=1,\ldots,n$, take simple functions $h_1,\ldots,h_n\in L^2(X)$ such that $\|f_i-h_i\|_2<\delta$ for some sufficiently small $\delta<\sqrt{\frac{\epsilon}{3}}$ so that $$\left|\langle U_\gamma f_i,f_i\rangle|-|\langle U_\gamma h_i,h_i\rangle\right|<\frac{\epsilon}{3}$$ for all $i=1,\ldots,n$ and all $\gamma \in \Gamma$.  Write $h_i=\sum_{j} c_{ij}1_{A_{ij}}$.  Then
$$\langle U_\gamma h_i,h_i\rangle=\sum_{j,k}c_{ij}\overline{c_{ik}}\mu(\gamma A_{ij}\cap A_{ik}).$$ Fix $\eta>0$ so that $\eta\sum_{j,k}|c_{ij}c_{ik}|<\frac{\epsilon}{3}$ for all $i=1,\ldots,n$.  By assumption, there is $\gamma\in \Gamma$ such that $|\mu(\gamma A_{ij}\cap A_{ik})-\mu(A_{ij})\mu(A_{ik})|<\eta$ for each $i,j,k$.  It follows that $\left||\langle U_\gamma h_i,h_i\rangle|-|\sum_{jk}c_{ij}\overline{c_{ik}}\mu(A_{ij})\mu(A_{ik})|\right|\leq\sum_{jk}|c_{ij}\overline{c_{ik}}|\eta<\frac{\epsilon}{3}$. But $$\left|\sum_{jk}c_{ij}\overline{c_{ik}}\mu(A_{ij})\mu(A_{ik})\right|=\left|\int h_i\right|^2=\left|\int (f_i-h_i)\right|^2\leq \|f_i-h_i\|_2^2<\delta^2<\frac{\epsilon}{3}.$$  Consequently, for this $\gamma \in \Gamma$ and all $i=1,\ldots,n$, we have $|\langle U_\gamma f_i,f_i\rangle|<\epsilon$, as desired.
\end{proof}

The previous proposition shows that weak mixing can be written in an infinitary first-order way of a particularly simple form.  Recall that a \textbf{$\forall\bigvee\exists$-sentence} is one of the form $\sup_x\bigvee_{n\in \bb N}\varphi_n(\vec x)$, where each $\varphi(x)$ is an existential formula.  We call a class $\mathcal{C}$ of models of $T_\Gamma$ \textbf{$\forall\bigvee\exists$-axiomatizable} if there is a collection $\sigma_n$ of $\forall\bigvee\exists$-sentences such that, for all $\mathcal{M}\models T_\Gamma$, we have $\mathcal{M}\in \mathcal{C}$ if and only if $\mathcal{M}\models\sigma_n$ for each $n\in \bb N$.

\begin{cor}\label{weakAVE}
The class of weakly mixing models of $T_\Gamma$ is $\forall\bigvee\exists$-axiomatizable.
\end{cor}

\begin{proof}
For each $n$, let $\sigma_n$ be the sentence
$$\sup_{x_1,\ldots,x_n,y_1,\ldots,y_n}\bigvee_{\gamma \in \Gamma}\max_{i=1,\ldots,n}|\mu(x_i\cap \gamma y_i)-\mu(x_i)\mu(y_i)|.$$
Then an action $a$ is weakly mixing if and only if $\sigma_n^{\cal M_a}=0$ for all $n\in \bb N$.
\end{proof}

The following corollary uses the notion of an \textbf{enforceable property}, which is a model-theoretic notion of genericity; a precise definition can be found in the first author's article \cite{enforceable}.

\begin{cor}
Suppose that $\Gamma$ is a group without property (T).  Then being a weakly mixing action is an enforceable property of actions of $\Gamma$.
\end{cor}

\begin{proof}
By a result of Kerr and Pichot \cite{KP}, since $\Gamma$ does not have property (T), there is a locally universal weakly mixing action of $\Gamma$.  Using this fact, Corollary \ref{weakAVE}, and \cite[Proposition 2.6]{enforceable}, the result follows.
\end{proof}

Since being e.c.\ is also an enforceable property (\cite[Proposition 2.10]{enforceable}) and since the conjunction of two enforceable properties is also enforceable, we get:

\begin{cor}
Suppose that $\Gamma$ is a group without property (T).  Then there is an e.c.\ weakly mixing action of $\Gamma$.
\end{cor}

Note that if a group has property (T), then no locally universal action of it can be weakly mixing (or even ergodic), whence the previous corollary is optimal.

\section{Generalities on model companions of meaure-preserving actions}

In this section, we establish a number of interesting general results on the existence of the model companion for $T_\Gamma$.  A consequence of these results is that $T_\Gamma$ exists whenever $\Gamma$ is a universally free group, a fact we generalize later in Section \ref{approximatelytreeable}.  We conclude this section with an open mapping criterion for the existence of the model companion for $T_\Gamma$.

\subsection{Definition of model companions}\label{modcompsubsection}

\begin{defn}
We say that \textbf{the model companion of $T_\Gamma$ exists} if there is a set $T$ of $L_\Gamma$-sentences such that, for all $\cal M_a\models T_\Gamma$, we have $\cal M_a\models T$ if and only if $a$ is an e.c.\ action.  In this case, there is a unique such theory $T$ (up to logical equivalence), which we denote by $T_\Gamma^*$.
\end{defn}

\begin{remark}
The model theorist will recognize that this is not the official definition of the model companion but is rather an equivalent reformulation (which holds in our context since the theory $T_\Gamma$ is $\forall\exists$-axiomatizable).
\end{remark}

There is a useful test for when the model companion exists:  

\begin{fact}
$T_\Gamma^*$ exists if and only if:  whenever $(\cal M_{a_i})_{i\in I}$ is a family of e.c.\ actions of $\Gamma$ and $\u$ is an ultrafilter on $I$, we have that $\prod_\u \cal M_{a_i}$ is also an e.c.\ action of $\Gamma$.
\end{fact}

The proof of the previous fact hings on an abstract characterization of when a class of structures in some language is the set of models of some theory, namely when the class is closed under isomorphism, ultraproduct, and ultraroot (see, for example, \cite[Proposition 5.14]{bbhu}).  In general, the class of e.c.\ structures is closed under elementary substructures, whence the nontrivial closure condition in the previous fact is that of being closed under ultraproducts.

As mentioned in the introduction, when $\Gamma$ is amenable, $T_\Gamma^*$ exists and $T_\Gamma^*=T_{\Gamma,free}$.  By a recent result of Berenstein, Ibarlucia, and Henson \cite{BIH}, for any finitely generated free group $\bb F$, we have that $T_{\bb F}^*$ exists.  We will generalize this fact in a number of ways throughout this paper.

\subsection{Model-theoretic shenanigans}

In this subsection, we establish some purely model-theoretic results; these results will be applied in the next subsection to the case of actions.

\begin{defn}
Suppose that $L_1\subseteq L_2$ are languages, $T_2$ is an $L_2$-theory, and $T_1:=T_2|L_1$, that is, the set of sentences in $T_2$ that are actually $L_1$-sentences.  We say that the pair $(T_1,T_2)$ has the:
\begin{itemize}
    \item \textbf{expansion property} if, given any $\cal M\models T_1$, there is $F(\cal M)\models T_2$ such that $\cal M\subseteq F(\cal M)|L_1$;
    \item \textbf{relative expansion property} if, given any $\cal N\models T_2$ and $\cal M\models T_1$ with $\cal N|L_1\subseteq \cal M$, then there is $G(\cal M,\cal N)\models T_2$ with $\cal M\subseteq G(\cal M,\cal N)|L_1$ and $\cal N\subseteq G(\cal M,\cal N)$.
\end{itemize}
\end{defn}

\begin{remark}
In the notation of the previous definition, if $L_2$ is countable, then a simple compactness argument implies that, to show that the pair $(T_1,T_2)$ has the (relative) expansion property, it suffices to consider only countable models.
\end{remark}

\begin{remark}
In the next section, we will suppose that $\Lambda$ is a subgroup of $\Gamma$ and work in the setting of the previous definition with $L_1:=L_\Lambda$, $L_2:=L_\Gamma$, and $T_2:=T_\Gamma$, so that $T_1=T_\Lambda$.  We will then show that $(T_\Lambda,T_\Gamma)$ has both the expansion property and the relative expansion property.
\end{remark}

\begin{lem}\label{mcdown}
Suppose that the pair $(T_1,T_2)$ has both the expansion property and the relative expansion property.  Further suppose that both $T_1$ and $T_2$ have the amalgamation property and that $T_2$ is $\forall\exists$-axiomatizable.  Then if $T_2$ has a model companion, then so does $T_1$.
\end{lem}

\begin{remark}
For the sake of simplicity, we carry out the proof of this lemma as well as the two results that follow in the setting of classical (discrete) logic.  The case of continuous logic is no more difficult, just slightly more annoying to write down.
\end{remark}

\begin{proof}[Proof of Lemma \ref{mcdown}]
Fix a family $(\cal M_i)_{i\in I}$ of e.c.\ models of $T_1$ and an ultrafilter $\u$ on $I$.  Set $\cal M:=\prod_\u \cal M_i$.  We wish to show that $\cal M$ is also an e.c.\ model of $T_1$.  Fix $\cal M'\models T_1$ with $\cal M\subseteq \cal M'$, a quantifier-free $L_1$-formula $\varphi(x,y)$, and elements $a_i\in \cal M_i$ such that $\cal M'\models \exists x\varphi(x,a)$, where $a=(a_i)_\u$.  We wish to show that $\cal M\models \exists x\varphi(x,a)$.

Since $T_1$ has the amalgamation property and $\cal M\subseteq \cal M'$ and $\cal M\subseteq \prod_\u F(\cal M_i)|L_1$, we have $\cal Q\models T_1$ and embeddings $i:\cal M'\hookrightarrow \cal Q$ and $j:\prod_\u F(\cal M_i)|L\hookrightarrow Q$ such that $i|\cal M=j|\cal M$.  Without loss of generality, we may assume that $j$ is an inclusion mapping and so $\prod_\u F(\cal M_i)|L\subseteq \cal Q$.  We then have $\cal R:=G(\cal Q,\prod_\u F(\cal M_i))\models T_2$ such that $\cal Q\subseteq G(\cal Q,\cal R)|L_1$ and $\prod_\u F(\cal M_i)\subseteq G(\cal Q,\cal R)$.  For each $i\in I$, let $\cal P_i\models T_2$ be e.c.\ with $F(\cal M_i)\subseteq \cal P_i$ (which is possible since $T_2$ is $\forall\exists$-axiomatizable), so $\prod_\u F(\cal M_i)\subseteq \cal P:=\prod_\u \cal P_i$.  By assumption, $\cal P$ is also e.c.\  Since $T_2$ has the amalgamation property, we can find $\cal S\models T_2$ and embeddings $b:G(\cal Q,\cal R)\hookrightarrow \cal S$ and $c:\cal P\hookrightarrow \cal S$ so that $b|\prod_\u F(\cal M_i)=c|\prod_\u F(\cal M_i)$.

We are now ready to conclude:  Since $\cal M'\models \exists x \varphi(x,a)$, we have that $\cal Q\models \exists x \varphi(x,a)$ and so $G(\cal Q,\cal R)\models \exists x \varphi(x,a)$ and so $\cal S\models \exists x \varphi(b(a))$.  Since $\cal P$ is e.c.\ and since $b(a)=c(a)$, we have that $\cal P\models \exists x \varphi(x,a)$.  Consequently, for $\u$-almost all $i\in I$, we have $\cal P_i\models \exists x \varphi(x,a_i)$.  Since $\cal M_i\subseteq \cal P_i|L_1$, we have that $\cal M_i\models \exists x \varphi(x,a_i)$ for these $i\in I$, and thus $\cal M\models \exists x \varphi(x,a)$, as desired.
\end{proof}

\begin{lem}\label{ecrestrict}
Suppose that $(T_1,T_2)$ has the relative expansion property.  Then for any e.c.\ model $\cal M$ of $T_2$, we have that $\cal M|L_1$ is an e.c.\ model of $T_1$.
\end{lem}

\begin{proof}
Take $\cal N\models T_1$ with $\cal M|L_1\subseteq \cal N$ and an existential $L_1$-sentence $\sigma$ with parameters from $\cal M$ such that $\cal N\models \sigma$.  We then have $G(\cal N,\cal M)\models \sigma$ and $\cal M\subseteq G(\cal N,\cal M)$; since $\cal M$ is e.c.\, then $\cal M\models \sigma$, as desired.
\end{proof}

\begin{cor}\label{mclocal}
Suppose that $T$ is an $L$-theory.  Further suppose that there is an increasing sequence $(L_n)_{n\in \bb N}$ of sublanguages of $L$ with $L=\bigcup_{n\in \bb N}L_n$ such that $(T_n,T)$ has the relative expansion property for all $n$, where $T_n:=T|L_n$.  Further suppose that $T_n$ has a model companion for each $n$.  Then $T$ also has a model companion.
\end{cor}

\begin{proof}
Fix a family $(\cal M_i)_{i\in I}$ of e.c.\ models of $T$ and an ultrafilter $\u$ on $I$.  Set $\cal M:=\prod_\u \cal M_i$.  Suppose that there is  $\cal M'\models T$ with $\cal M\subseteq \cal M'$ and an existential $L$-sentence with parameters from $\cal M$ such that $\cal M'\models \sigma$.  We want to show $\cal M\models \sigma$.  Take $n\in \bb N$ such that $\sigma$ is an $L_n$-sentence.  By Lemma \ref{ecrestrict}, $\cal M_i|L_n$ is an e.c.\ model of $T_n$ for each $i\in I$; since $T_n$ has a model companion, we have that $\prod_\u (\cal M_i|L_n)$ is an e.c.\ model of $T_n$.  Since $\prod_\u (\cal M_i|L_n)\subseteq \cal M'|L_n$, we have that $\cal M|L_n=\prod_\u (\cal M_i|L_n)\models \sigma$, as desired.
\end{proof}

\subsection{Preservation properties for the existence of \texorpdfstring{$T_\Gamma^*$}{the model companion}}

The next two lemmas shows that, for any subgroup $\Lambda$ of $\Gamma$, the pair $(T_\Lambda,T_\Gamma)$ has both the expansion property and the relative expansion property.

\begin{lem}\label{expgroups}
Let $\Gamma$ be a group with subgroup $\Lambda$.  Then the pair $(T_\Lambda,T_\Gamma)$ has the expansion property. 
\end{lem}

\begin{proof}
This is an immediate consequence of the existence of the \emph{coinduction} procedure (see, for example, \cite[Section 10(G)]{kechrisglobal}).
\end{proof}

The following is a special case of Epstein's construction of coinducing an action from a subequivalence relation, see \cite[Section 3]{IKT}.

\begin{lem}\label{relexpgroups}
Let $\Gamma$ be a group with subgroup $\Lambda$.  Then the pair $(T_\Lambda,T_\Gamma)$ has the relative expansion property.  More precisely, let $\Gamma \acts (X, \mu)$ and $\Lambda \acts (Y, \nu)$ be {\pmp} actions on standard probability spaces and $\phi : Y \rightarrow X$ a $\Lambda$-equivariant factor map. Then there exists a {\pmp} action $\Gamma \acts (\bar{Y}, \bar{\nu})$ and a $\Lambda$-equivariant map $\pi : \bar{Y} \rightarrow Y$ such that $\phi \circ \pi$ is $\Gamma$-equivariant.
\end{lem}

\begin{proof}
	Fix a choice of representatives $r : \Gamma / \Lambda \rightarrow \Gamma$ for the cosets of $\Lambda$ in $\Gamma$ with $r(\Lambda) = e$, and define the cocycle $\rho : \Gamma \times (\Gamma / \Lambda) \rightarrow \Lambda$ by $\rho(\gamma, a \Lambda) = r(\gamma a \Lambda)^{-1} \gamma r(a \Lambda)$. Define an action of $\Gamma$ on $\bar{Y} := Y^{\Gamma/ \Lambda}$ by
	\[(\gamma \cdot \bar{y})(a \Lambda) = \rho(\gamma^{-1}, a \Lambda)^{-1} \bar{y}(\gamma^{-1} a \Lambda).\]
	
	Let $\nu = \int \nu_x \ d \mu(x)$ be the \textbf{disintegration} of $\nu$ with respect to $\phi$ (that is, $x \mapsto \nu_x$ is the unique, up to agreement almost-everywhere, measurable function satisfying $\nu_x(\phi^{-1}(x)) = 1$ for $\mu$-almost-every $x$ and $\nu = \int \nu_x \ d \mu(x)$; see \cite[Theorem A.7]{glasner}). Define
	\[\bar{\nu} = \int_X \prod_{a \Lambda \in \Gamma / \Lambda} \nu_{r(a \Lambda)^{-1} \cdot x} \ d \mu(x),\]
	so a $\bar{\nu}$-random point $\bar{y}$ of $Y^{\Gamma / \Lambda}$ is obtained by choosing a $\mu$-random point $x \in X$ and independently for each $a \Lambda \in \Gamma / \Lambda$ letting $\bar{y}(a \Lambda)$ be a $\nu_{r(a \Lambda)^{-1} \cdot x}$-random point in $Y$. Clearly if $\pi : Y^{\Gamma / \Lambda} \rightarrow Y$ is the evaluation map $\pi(\bar{y}) = \bar{y}(\Lambda)$, then $\pi$ is measure-preserving, and it is also $\Lambda$-equivariant since
	\[\pi(\lambda \cdot \bar{y}) = (\lambda \cdot \bar{y})(\Lambda) = \rho(\lambda^{-1}, \Lambda)^{-1} \bar{y}(\Lambda) = \lambda \bar{y}(\Lambda) = \lambda \pi(\bar{y}).\]
	Also, for $\bar{\nu}$-almost-every $\bar{y} \in Y^{\Gamma / \Lambda}$ there is $x \in X$ satisfying $\phi(\bar{y}(a \Lambda)) = r(a \Lambda)^{-1} \cdot x$ for every $a \Lambda \in \Gamma / \Lambda$, and in this case
	\begin{align*}
	\phi \circ \pi(\gamma \cdot \bar{y}) & = \phi((\gamma \cdot \bar{y})(\Lambda))\\
	& = \phi(\rho(\gamma^{-1}, \Lambda)^{-1} \bar{y}(\gamma^{-1} \Lambda))\\
	& = \rho(\gamma^{-1}, \Lambda)^{-1} \phi(\bar{y}(\gamma^{-1} \Lambda))\\
	& = \rho(\gamma^{-1}, \Lambda)^{-1} r(\gamma^{-1} \Lambda)^{-1} \cdot x\\
	& = r(\Lambda)^{-1} \gamma \cdot x\\
	& = \gamma \cdot x\\
	& = \gamma \cdot \phi \circ \pi(\bar{y}).
	\end{align*}
	Thus $\phi \circ \pi$ is $\Gamma$-equivariant on a $\bar{\nu}$-conull subset of $Y^{\Gamma / \Lambda}$.
	
	To finish the proof, it only remains to check that the measure $\bar{\nu}$ is $\Gamma$-invariant. First notice that uniqueness of measure disintegration and the $\Lambda$-invariance of $\mu$ and $\nu$ imply that $\nu_x(\lambda \cdot E) = \nu_{\lambda^{-1} \cdot x}(E)$ for every $\lambda \in \Lambda$, every measurable set $E \subseteq Y$, and $\mu$-almost-every $x \in X$. Now consider any collection of measurable sets $E_{a \Lambda} \subseteq Y$ for $a \Lambda \in \Gamma / \Lambda$, set $E = \prod_{a \Lambda \in \Gamma / \Lambda} E_{a \Lambda} \subseteq Y^{\Gamma / \Lambda}$, and let $\gamma \in \Gamma$. Notice that
	\[\gamma \cdot E = \prod_{a \Lambda \in \Gamma / \Lambda} \rho(\gamma^{-1}, a \Lambda)^{-1} E_{\gamma^{-1} a \Lambda}.\]
	We then have
	\begin{align*}
	\bar{\nu}(\gamma \cdot E) & = \int_X \prod_{a \Lambda \in \Gamma / \Lambda} \nu_{r(a \Lambda)^{-1} \cdot x}(\rho(\gamma^{-1}, a \Lambda)^{-1} E_{\gamma^{-1} a \Lambda}) \ d \mu(x)\\
	& = \int_X \prod_{a \Lambda \in \Gamma / \Lambda} \nu_{\rho(\gamma^{-1}, a \Lambda) r(a \Lambda)^{-1} \cdot x}(E_{\gamma^{-1} a \Lambda}) \ d \mu(x)\\
	& = \int_X \prod_{a \Lambda \in \Gamma / \Lambda} \nu_{r(\gamma^{-1} a \Lambda)^{-1} \gamma^{-1} \cdot x}(E_{\gamma^{-1} a \Lambda}) \ d \mu(x)\\
	& = \int_X \prod_{a \Lambda \in \Gamma / \Lambda} \nu_{r(\gamma^{-1} a \Lambda)^{-1} \cdot x}(E_{\gamma^{-1} a \Lambda}) \ d \mu(x)\\
	& = \int_X \prod_{a \Lambda \in \Gamma / \Lambda} \nu_{r(a \Lambda)^{-1} \cdot x}(E_{a \Lambda}) \ d \mu(x)\\
	& = \bar{\nu}(E),
	\end{align*}
	where for the third-to-last equality we use the $\Gamma$-invariance of $\mu$, and for the second-to-last equality we use the fact that left-multiplication by $\gamma^{-1}$ permutes the set $\Gamma / \Lambda$. Since every measurable set can be approximated in $\bar{\nu}$-measure by sets that are finite disjoint unions of sets of the above form, it follows that $\bar{\nu}$ is indeed $\Gamma$-invariant.
\end{proof}

Since both $T_\Gamma$ and $T_\Lambda$ have the amalgamation property (as witnessed by the existence of relatively independent joinings-see Lemma \ref{lem:amalgam} below) and are $\forall\exists$-axiomatizable, Lemmas \ref{mcdown}, \ref{expgroups}, and \ref{relexpgroups} allow us to conclude:

\begin{cor} \label{cor:mc_sub}
The property of $T_\Gamma$ having a model companion is inherited by subgroups:  if $\Lambda$ is a subgroup of $\Gamma$ and $T_\Gamma^*$ exists, then so does $T_\Lambda^*$.
\end{cor}

Lemmas \ref{ecrestrict} and \ref{relexpgroups} imply:

\begin{cor}\label{restrictedec}
If $\Gamma \acts (X, \mu)$ is an e.c.\ action and $\Lambda \leq \Gamma$ is a subgroup, then the restricted action $\Lambda \acts (X, \mu)$ is e.c.\ as well.
\end{cor}

Corollary \ref{mclocal} and Lemma \ref{relexpgroups} imply:

\begin{cor}\label{mclocgroup}
The property of $T_\Gamma$ having a model companion is a local property of groups:  if $\Gamma$ is a group such that $T_\Lambda^*$ exists for every finitely generated subgroup $\Lambda$ of $\Gamma$, then $T_\Gamma^*$ also exists.
\end{cor}

The following corollary is an immediate consequence of Corollary \ref{mclocal}; it was stated in the introduction of \cite{BIH}:

\begin{cor}\label{freeinfgen}
If $\bb F_\infty$ is the free group on a countably infinite set of generators, then $T_{\bb F_{\infty}}^*$ exists.
\end{cor}

Note that it is unclear if having a model companion is closed under quotients.  Indeed, by the previous corollary, if it is closed under quotients, then $T_\Gamma^*$ exists for all countable groups $\Gamma$.

\subsection{Coamenable subgroups}

It is currently unknown if the class of groups for which $T_\Gamma^*$ exists is closed under extensions, that is to say: if $\Gamma$ is a countable group with normal subgroup $\Lambda$ for which both $T_\Lambda^*$ and $T_{\Gamma/\Lambda}^*$ exists, must $T_\Gamma^*$ necessarily exist.  In this subsection, we show that this is the case if we assume that $\Lambda$ is a co-amenable normal subgroup of $\Gamma$, that is to say $\Gamma/\Lambda$ is amenable.

Below, when $\Lambda$ is a normal subgroup of $\Gamma$, we say that $\Gamma\acts (X,\mu)$ weakly contains a free {\pmp} action of $\Gamma/\Lambda$ to mean that this action weakly contains some action $\Gamma\acts (Y,\nu)$ such that every point in $Y$ has stabilizer $\Lambda$. In this case we can always choose $(Y, \nu)$ to be a standard Borel probability space since weak containment is transitive and every free {\pmp} action of $\Gamma / \Lambda$ weakly contains a Bernoulli shift action of $\Gamma / \Lambda$ \cite{AW13}.

The following is a consequence of the Ornstein-Weiss quasitiling lemma:

\begin{lem}\label{OWlemma}
Suppose that $\Lambda$ is a normal co-amenable subgroup of $\Gamma$ and that $\Gamma\acts^a (X,\mu)$ is a {\pmp} action that weakly contains a free {\pmp} action of $\Gamma/\Lambda$.  Then for every finite $S\subseteq \Gamma$ with $e\in S$ and every $\epsilon>0$, there is a measurable map $c:X\to \Gamma/\Lambda$ having finite image such that $\mu(X')>1-\epsilon$, where 
$$X'=\{x\in X \ : \ c(s^{-1}\cdot x)=c(x)s \text{ for all }s\in S\}.$$
\end{lem}

\begin{proof}
	Let $(Y, \nu)$ be a standard Borel probability space and $\Gamma \acts^b (Y, \nu)$ an action that is weakly contained in the action $\Gamma \acts (X, \mu)$ and has the property that every $y \in Y$ has stabilizer $\Lambda$. Since $\Gamma / \Lambda$ is amenable, the orbit equivalence relation
	$$E_\Gamma^Y = \{(y, \gamma \cdot y) : y \in Y, \ \gamma \in \Gamma\}$$
	must be $\nu$-hyperfinite \cite{OW80}, meaning there is a sequence of Borel equivalence relations $E_n \subseteq E_\Gamma^Y$ such that each class of each $E_n$ is finite, $E_n \subseteq E_{n+1}$ for all $n$, and $\bigcup_n E_n$ coincides with $E_\Gamma^Y$ on a $\Gamma$-invariant conull set. Set $F = E_n$ for a sufficiently large value of $n$ so that the set
	$$Y' = \{y \in Y : (y, s^{-1} \cdot y) \in F \text{ for all }s\in S\}$$
	has measure greater than $1-\epsilon$. Since $F$ is a Borel equivalence relation whose classes are all finite, there exists a Borel set $D \subseteq Y$ that contains precisely one point from every $F$-class \cite[Thm 12.16]{K95}. Let $d : Y \rightarrow \Gamma / \Lambda$ be the function that sends $y \in Y$ to the $\Lambda$-coset $\gamma \Lambda$ for any (equivalently every) $\gamma \in \Gamma$ satisfying $\gamma y \in D$ and $(y, \gamma y) \in F$. We observe that $d$ is measurable since for every $\gamma \in \Gamma$ the set of $y \in Y$ satisfying $y \ F \ \gamma \cdot y$ is equal to the Borel set $(\mathrm{id} \times b(\gamma))^{-1}(F)$. It is immediate from these definitions that the set of $y \in Y$ satisfying $d(s^{-1} \cdot y) = d(y) s$ for all $s\in S$ is precisely $Y'$ and thus has measure larger than $1-\epsilon$. Finally, since the map $d$ is described by the countable measurable partition $\{d^{-1}(\gamma \Lambda) : \gamma \Lambda \in \Gamma / \Lambda\}$ and the translated maps $d(s^{-1} \cdot)$, $s \in S$, are described by the $S$-translates of that partition, the fact that $\Gamma \acts (X, \mu)$ weakly contains $\Gamma \acts (Y, \nu)$ immediately implies that a function $c : X \rightarrow \Gamma / \Lambda$ with the desired property exists.
\end{proof}

\begin{thm} \label{thm:lift}
Let $\Gamma$ be a countable group and $\Lambda$ a normal co-amenable subgroup of $\Gamma$. If $\Gamma \acts (X, \mu)$ is a {\pmp} action that weakly contains a free {\pmp} action of $\Gamma / \Lambda$ and if the restricted action $\Lambda \acts (X, \mu)$ is existentially closed, then the action $\Gamma \acts (X, \mu)$ is existentially closed.
\end{thm}

\begin{proof}
We use the criterion for being e.c.\ established in Proposition \ref{criterion} (and use the notation established right before the statement of that proposition).  Let $\Gamma \acts (Y, \nu)$ be a {\pmp} action and let $\phi : Y \rightarrow X$ be a $\Gamma$-equivariant factor map. Fix $p, q \in \bN$, $\alpha : Y \rightarrow p$ and $\beta : X \rightarrow q$ measurable maps, $S \subseteq \Gamma$ be finite with $e \in S$, and $\epsilon > 0$. Fix a choice $r : \Gamma / \Lambda \rightarrow \Gamma$ of representatives for the cosets of $\Lambda$ in $\Gamma$ and let $\rho : (\Gamma / \Lambda) \times \Gamma \rightarrow \Lambda$ be the cocycle $\rho(a \Lambda, \gamma) = r(a \Lambda) \gamma r(a \gamma \Lambda)^{-1}$.

By applying Lemma \ref{OWlemma} and composing with the function $r$, we obtain a measurable map $c : X \rightarrow r(\Gamma / \Lambda)$ having finite image and satisfying $\mu(X') > 1 - \frac{\epsilon}{2}$, where
$$X' = \{x \in X : \ c(s^{-1} \cdot x) = r(c(x) s \Lambda) \text{ for all }s\in S\}.$$
Define the finite set $W = c(X) \subseteq r(\Gamma / \Lambda) \subseteq \Gamma$ and for $w \in W$ set $X_w = c^{-1}(w)$. Also set $W' = \{w \in W : X_w \cap X' \neq \varnothing\}$ and notice that whenever $w \in W'$ and $s \in S$, we have $r(w s \Lambda) \in W$ since $r(w s \Lambda) = c(s^{-1} \cdot x)$ whenever $x \in X_w \cap X'$.

Now consider the functions $\alpha_{w S}$ for $w \in W$. We observe that we always have $w \cdot \alpha_S(y) = \alpha_{w S}(w \cdot y)$ since for $s \in S$, we have
$$(w \cdot \alpha_S(y))(w s) = \alpha_S(y)(s) = \alpha(s^{-1} \cdot y) = \alpha(s^{-1} w^{-1} w \cdot y) = \alpha_{w S}(w \cdot y)(w s).$$
Therefore for all $w \in W$ and $\pi \in p^S$, we have $w \cdot \alpha_S^{-1}(\pi) = \alpha_{w S}^{-1}(w \cdot \pi)$.  In particular, we have
\begin{equation} \label{eqn:lift1}
w \cdot \Big( \alpha_S^{-1}(\pi) \cap \phi^{-1}(X_w) \Big) = \alpha_{w S}^{-1}(w \cdot \pi) \cap w \cdot \phi^{-1}(X_w).
\end{equation}
Similarly, since $e \in S$, for $w \in W'$, $s \in S$, and $y \in Y$ we have
\begin{align*}
\alpha_{w S}(y)(w s) = \alpha(s^{-1} w^{-1} \cdot y) & = \alpha(r(w s \Lambda)^{-1} \rho(w \Lambda, s)^{-1} \cdot y)\\
 & = \alpha_{r(w s \Lambda) S}(\rho(w \Lambda, s)^{-1} \cdot y)(r(w s \Lambda)).
\end{align*}
So for every $w \in W'$, we have
\begin{equation} \label{eqn:lift2}
\bigcup_{s \in S} \left\{y \in Y: \alpha_{r(w s \Lambda) S}(\rho(w \Lambda, s)^{-1} \cdot y)(r(w s \Lambda)) \neq \alpha_{w S}(y)(w s)\right\} = \emptyset.
\end{equation}

Since $\Lambda \acts (X, \mu)$ is existentially closed, we can find measurable functions $\gamma_w : w \cdot X_w \rightarrow p^{w S}$ for $w \in W$ satisfying the following two conditions. First, relative to each of the sets $w \cdot (\beta^{-1}(j) \cap X_w)$, the $\gamma_w$'s will have identical distribution in measure to the functions $\alpha_{w S}$, meaning that for all $w \in W$, $j \in q$, and $\pi \in p^S$, we have
\begin{equation} \label{eqn:lift3}
\mu \Big( \gamma_w^{-1}(w \cdot \pi) \cap w \cdot \beta^{-1}(j) \Big) = \nu \Big( \alpha_{w S}^{-1}(w \cdot \pi) \cap \phi^{-1}(w \cdot \beta^{-1}(j) \cap w \cdot X_w) \Big).
\end{equation}
(We point out that taking an intersection with $w \cdot X_w$ on the left would be redundant since the domain of $(\gamma_w)$ is $w \cdot X_w$). Second, we control how the functions $\gamma_w$ relate to the action of $\Lambda$ and demand, in view of (\ref{eqn:lift2}), that $\mu(D_w) < \frac{\epsilon}{2|W|}$ for all $w \in W'$, where
\begin{equation*}
D_w = \bigcup_{s \in S} \left\{x \in w \cdot X_w:  \gamma_{r(w s \Lambda)}(\rho(w \Lambda, s)^{-1} \cdot x)(r(w s \Lambda)) \neq \gamma_w(x)(w s)\right\}.
\end{equation*}

Define $\tilde{\alpha} : X \rightarrow p$ by setting $\tilde{\alpha}(x) = \gamma_w(w \cdot x)(w)$ when $w \in W$ and $x \in X_w$. Notice that when $x \in (X_w \setminus w^{-1} \cdot D_w) \cap X'$, we have $w \cdot \tilde{\alpha}_S(x) = \gamma_w(w \cdot x)$, since for any $s \in S$, we have $s^{-1} \cdot x \in X_{r(w s \Lambda)}$ since $x \in X'$ and therefore
\begin{align*}
(w \cdot \tilde{\alpha}_S(x))(w s) = \tilde{\alpha}_S(x)(s) = \tilde{\alpha}(s^{-1} \cdot x) & = \gamma_{r(w s \Lambda)}(r(w s \Lambda) s^{-1} \cdot x)(r(w s \Lambda))\\
 & = \gamma_{r(w s \Lambda)}(\rho(w \Lambda, s)^{-1} w \cdot x)(r(w s \Lambda))
\end{align*}
and the final term above is equal to $\gamma_w(w \cdot x)(w s)$ since $w \cdot x \not\in D_w$. As we additionally have that $X_w \cap X' = \emptyset$ when $w \in W \setminus W'$, we conclude that for all $w \in W$, we have
\begin{equation} \label{eqn:lift4}
\Big( w \cdot ( \tilde{\alpha}_S^{-1}(\pi) \cap X_w ) \Big) \symd \Big( \gamma_w^{-1}(w \cdot \pi) \cap w \cdot X_w \Big) \subseteq D_w \cup w \cdot (X_w \setminus X').
\end{equation}

For $\pi \in p^S$ and $j \in q$, equation (\ref{eqn:lift1}) implies that
\begin{equation} \label{eqn:lift5}
\nu(\alpha_S^{-1}(\pi) \cap \phi^{-1}(\beta^{-1}(j))) = \sum_{w \in W} \nu(\alpha_{w S}^{-1}(w \cdot \pi) \cap w \cdot \phi^{-1}(\beta^{-1}(j) \cap X_w))
\end{equation}
and equation (\ref{eqn:lift4}) implies
\begin{equation} \label{eqn:lift6}
\left| \mu(\tilde{\alpha}_S^{-1}(\pi) \cap \beta^{-1}(j)) - \sum_{w \in W} \mu(\gamma_w^{-1}(w \cdot \pi) \cap w \cdot (\beta^{-1}(j) \cap X_w)) \right| < \epsilon.
\end{equation}
Since the sums over $w \in W$ in (\ref{eqn:lift5}) and (\ref{eqn:lift6}) are equal by (\ref{eqn:lift3}), it follows that
$$\left| \nu(\alpha_S^{-1}(\pi) \cap \phi^{-1}(\beta^{-1}(j))) - \mu(\tilde{\alpha}_S^{-1}(\pi) \cap \beta^{-1}(j)) \right| < \epsilon.$$
We conclude that the action $\Gamma \acts (X, \mu)$ is existentially closed.
\end{proof}

\begin{cor} \label{cor:coam_mc}
Suppose that $\Gamma$ is a group containing a normal coamenable subgroup $\Lambda$ for which $T_\Lambda^*$ exists.  Then $T_\Gamma^*$ exists as well.
\end{cor}

\begin{proof}
Let $(\cal M_{a_i})_{i\in I}$ be a family of e.c.\ models of $T_\Gamma$ and fix an ultrafilter $\u$ on $I$; it suffices to show that $\cal M_a:=\prod_\u \cal M_{a_i}$ is also an e.c.\ model of $T_\Gamma$.  For each $i\in I$, let $\cal M_{b_i}:=\cal M_{a_i}|L_{\Lambda}\models T_{\Lambda}$ denote the restricted action and set $\cal M_b:=\prod_\u \cal M_{b_i}=\cal M_a|L_\Lambda$.  By Corollary \ref{restrictedec}, each $\cal M_{b_i}$ is an e.c.\ model of $T_\Lambda$.  Since $T_\Lambda^*$ exists, $\cal M_b$ is also an e.c.\ model of $T_\Lambda$.  Since each $\cal M_{a_i}$ is e.c.\, it is also locally universal, whence so is $\cal M_a$. By Theorem \ref{thm:lift}, $\cal M_a$ is e.c.\ as desired.
\end{proof}

\begin{remark}
    Under the assumptions of the previous corollary, the axioms for $T_\Gamma^*$ are the axioms for $T_\Lambda^*$ together with the axioms for $T_{\Gamma,max}$.
\end{remark}

Recall from the introduction that a group $\Gamma$ is universally free if it embeds into an ultrapower of a free group.

\begin{cor}
Suppose that $\Gamma$ is a universally free group.  Then $T_\Gamma^*$ exists.
\end{cor}

\begin{proof}
By Corollary \ref{mclocgroup}, we may assume that $\Gamma$ is finitely generated, that is, that $\Gamma$ is a limit group.  By the aforemtnioned result of Kochloukova \cite{Kochloukova}, $\Gamma$ has a normal coamenable (not necessarily finitely generated) free subgroup $\Lambda$.  By Corollary \ref{freeinfgen}, $T_\Lambda^*$ exists.  Thus, by Corollary \ref{cor:coam_mc}, we have that $T_\Gamma^*$ exists.
\end{proof}

\subsection{An open mapping characterization for the existence of \texorpdfstring{$T_\Gamma^*$}{the model companion}}\label{openmappingsub}

In this subsection, we give an ergodic-theoretic characterization of the existence of $T_\Gamma^*$, which also yields axioms for the model companion when it exists.  First, we need a few preparatory lemmas.

\begin{lem}\label{eccover}
	For every $n \in \bN$, let $\Gamma \acts (X_n, \mu_n)$ be a {\pmp} action. Then there exists a {\pmp} action $\Gamma \acts (X, \mu)$ which is e.c.\ and which factors onto $\Gamma \acts (X_n, \mu_n)$ for every $n$.
\end{lem}

\begin{proof}
	Let $\Gamma\acts (Y,\nu)$ be any pmp action factoring onto each $\Gamma\acts (X_n,\mu_n)$ and then let $\Gamma\acts (X,\mu)$ be an e.c.\ action factoring onto $\Gamma \acts (Y,\nu)$.
\end{proof}

\begin{lem} \label{lem:amalgam}
	Let $\Gamma \acts (X, \mu)$ and $\Gamma \acts (Y, \nu)$ be {\pmp} actions that both factor onto the {\pmp} action $\Gamma \acts (Z, \eta)$ via the factor maps $\phi$ and $\psi$, respectively. If $(Z, \eta)$ is a standard probability space then there is a $\Gamma$-invariant measure $\lambda$ on $X \times Y$ having marginals $\mu$ on $X$ and $\nu$ on $Y$, respectively, and satisfying $\phi(x) = \psi(y)$ for $\lambda$-almost-every $(x, y) \in X \times Y$.
\end{lem}

\begin{proof}
	For each set $A \in \mathcal{B}_X$ denote by $f_A$ the Radon--Nikodym derivative of the measure $C \in \mathcal{B}_Z \mapsto \mu(A \cap \phi^{-1}(C))$ with respect to $\eta$, and similarly for $B \in \mathcal{B}_Y$ define $g_B$ to be the Radon--Nikodym derivative of $C \mapsto \nu(B \cap \psi^{-1}(C))$ with respect to $\eta$.
	
	Define a function $\lambda$ on the set of all measurable rectangles $A \times B$ ($A \in \mathcal{B}_X$, $B \in \mathcal{B}_Y$) by
	$$\lambda(A \times B) = \int f_A \cdot g_B \ d \eta.$$
	Suppose that $A \times B$ is the disjoint union of $A_n \times B_n$, $n \in \bN$. Since $(Z, \eta)$ is standard, we can pick countably generated $\sigma$-algebras $\Sigma_X \subseteq \mathcal{B}_X$ and $\Sigma_Y \subseteq \mathcal{B}_Y$ containing $\phi^{-1}(\mathcal{B}_Z)$ and $\psi^{-1}(\mathcal{B}_Z)$, respectively, with $A_n \in \Sigma_X$ and $B_n \in \Sigma_Y$ for all $n$. Disintegrate $\mu \restriction \Sigma_X$ with respect to $\phi$ and $\nu \restriction \Sigma_Y$ with respect to $\psi$ to obtain almost-everywhere unique measurable maps $z \mapsto \mu_z$ and $z \mapsto \nu_z$, where $\mu_z$ and $\nu_z$ are probability measures on $\Sigma_X$ and $\Sigma_Y$ respectively, satisfying $\mu_z(\phi^{-1}(z)) = 1 = \nu_z(\psi^{-1}(z))$ for $\eta$-almost every $z \in Z$, $\mu \restriction \Sigma_X = \int_Z \mu_z \ d \eta$ and $\nu \restriction \Sigma_Y = \int_Z \nu_z \ d \eta$. Then for $A' \in \Sigma_X$ we have
	$$\int \mu_z(A') \cdot 1_C \ d \eta = \mu(A' \cap \phi^{-1}(C)) = \int f_{A'} \cdot 1_C \ d \eta$$
	for all $C \in \mathcal{B}_Z$ and therefore $f_{A'}(z) = \mu_z(A')$ for a.e. $z$. Similarly $g_{B'}(z) = \nu_z(B')$ for every $B' \in \Sigma_Y$ and a.e. $z$. It follows that the function $\lambda$ and the measure $\int \mu_z \times \nu_z \ d \eta$ coincide when restricted to measurable rectangles in $\Sigma_X \times \Sigma_Y$ and therefore $\sum_{n \in \bN} \lambda(A_n \times B_n) = \lambda(A \times B)$.
	
	The previous paragraph shows that $\lambda$ is a probability premeasure on the algebra of finite unions of measurable rectangles, so by Caratheodory's theorem $\lambda$ has a unique extension to a probability measure on $\mathcal{B}_X \times \mathcal{B}_Y$. Since $\lambda$ is $\Gamma$-invariant it follows from the uniqueness of the extension that it must be $\Gamma$-invariant as well. Moreover, it is immediately seen that this extension, which we denote by $\lambda$ as well, has marginals $\mu$ and $\nu$ on $X$ and $Y$ respectively. 
	
	Lastly, since $Z$ is standard we have that $Z \times Z \setminus \{(z,z) : z \in Z\}$ is a countable union of measurable rectangles. Since every measurable rectangle $C \times D \subseteq Z \times Z$ that is disjoint with the diagonal satisfies
	$$\lambda((\phi \times \psi)^{-1}(C \times D)) = \int_Z 1_C \cdot 1_D \ d \eta = \eta(C \cap D) = 0,$$
	we conclude that $\phi(x) = \psi(y)$ for $\lambda$-almost-every $(x,y)$.	
\end{proof}

\begin{lem}
	Let $(\Gamma \acts (X_i, \mu_i))_{i \in I}$ be a collection of {\pmp} actions, let $\u$ be an ultrafilter on $I$, let $\Gamma \acts \prod_\u (X_i, \mu_i)$ be the ultraproduct action, and let $q \in \bN$.
	\begin{enumerate}
		\item Given any measurable map $\alpha : \prod_\u X_i \rightarrow q$, there exist measurable maps $\alpha^i : X_i \rightarrow q$ such that $\alpha([x_i]_\u) = \lim_{i \rightarrow \u} \alpha^i(x_i)$ for almost-every $[x_i]_\u \in \prod_\u X_i$.
		\item If $\alpha_i : X_i \rightarrow q$ is measurable for each $i\in I$ and $\alpha : \prod_\u X_i \rightarrow q$ is defined by $\alpha([x_i]_\u) = \lim_{i \rightarrow \u} \alpha^i(x_i)$, then
		\begin{enumerate}
			\item $\alpha$ is measurable,
			\item $\alpha_\Gamma([x_i]_\u) = \lim_{i \rightarrow \u} \alpha^i_\Gamma([x_i]_\u)$ for every $[x_i]_\u \in \prod_\u X_i$, and
			\item $(\alpha_\Gamma)_*(\prod_\u \mu_i) = \lim_{i \rightarrow \u} (\alpha^i_\Gamma)_*(\mu_i)$.
		\end{enumerate} 
	\end{enumerate}
\end{lem}

\begin{proof}
	(1). Since $\alpha$ is measurable, for each $k \in q-1$ there is a collection of measurable sets $A_k^i \subseteq X_i$ such that $\alpha^{-1}(k) = [A_k^i]_\u$ up to a $\prod_\u \mu_i$-null set. Define $\alpha^i : X_i \rightarrow q$ by
	$$\alpha^i(x) = \begin{cases}
		k & \text{if } k \in q-1 \text{ is least with } x \in A_k^i\\
		q-1 & \text{if } x \not\in A_0^i \cup \cdots \cup A_{q-2}^i.
	\end{cases}$$
	Then $\alpha^i$ is measurable and it is easy to check by induction on $k \in q$ that for almost-every $[x_i]_\u \in \alpha^{-1}(k)$ we have $\alpha([x_i]_\u) = \lim_{i \rightarrow \u} \alpha^i(x_i)$.
	
	(2). It is immediate that $\alpha$ is measurable since $\alpha^{-1}(k) = [(\alpha^i)^{-1}(k)]_\u$. Also, for every $\gamma \in \Gamma$ we have
	$$\alpha_\Gamma([x_i]_\u)(\gamma) = \alpha(\gamma^{-1} \cdot [x_i]_\u) = \alpha([\gamma^{-1} x_i]_\u) = \lim_{i \rightarrow \u} \alpha^i(\gamma^{-1} x_i) = \lim_{i \rightarrow \u} \alpha^i_\Gamma(x_i)(\gamma).$$
	This establishes (a) and (b). Next consider any clopen set $C \subseteq q^\Gamma$. For every $[x_i]_\u \in \prod_\u X_i$ we have that $\alpha_\Gamma([x_i]_\u) = \lim_{i \rightarrow \u} \alpha^i_\Gamma(x_i)$ belongs to $C$ if and only if (since $C$ is both open and closed) $\{i \in I : \alpha^i_\Gamma(x_i) \in C\} \in \u$, or equivalently $[x_i]_\u \in [(\alpha^i_\Gamma)^{-1}(C)]_\u$. Therefore $(\alpha_\Gamma)^{-1}(C) = [(\alpha^i_\Gamma)^{-1}(C)]_\u$ and
	$$(\alpha_\Gamma)_*(\prod_\u \mu_i)(C) = \lim_\u (\alpha_i^\Gamma)_*(\mu_i)(C).$$
	Since $\lim_{i \rightarrow \u} (\alpha_\Gamma^i)_*(\mu_i)$ is a probability measure on $q^\Gamma$ and all Borel probability measures on $q^\Gamma$ are uniquely determined by their values on clopen sets, (c) follows.
\end{proof}

We now prove a lemma providing an ergodic-theoretic characterization of e.c.\ factor maps.  To state it, given any $q\in \bb N$, we let $\operatorname{Prob}_\Gamma(q^\Gamma)$ denote the set of probability measures on $q^\Gamma$ preserved by the natural action of $\Gamma$ on $q^\Gamma$.  Given another integer $p$, we let $\operatorname{Prob}_\Gamma(q^\Gamma\times p^\Gamma)$ have the analogous meaning.  We view each of these spaces as equipped with their weak$^*$-topologies.  We also let $\pi:q^\Gamma\times p^\Gamma\to q^\Gamma$ denote the canonical projection map, which induces a push-forward map $\pi_*:\Prob_\Gamma(q^\Gamma \times p^\Gamma)\to \Prob_\Gamma(q^\Gamma)$. 

\begin{lem}\label{eclemma}
	The action $\Gamma \curvearrowright^a (X,\mu)$ is e.c.\ if and only if:  for any $p,q\in \bb N$, any weak$^*$-open subset $U$ of $\operatorname{Prob}_\Gamma(q^\Gamma\times p^\Gamma)$, and any $\beta:X\to q$ for which $(\beta_\Gamma)_*(\mu)\in \pi_*(U)$, there is $\gamma:X\to p$ such that $((\beta\times \gamma)_\Gamma)_*(\mu)\in U$.
\end{lem}

\begin{proof}
	First suppose that $\Gamma \curvearrowright^a (X,\mu)$ is e.c.\  Take a weak$^*$-open subset $U$ of $\operatorname{Prob}_\Gamma(q^\Gamma\times p^\Gamma)$ and a function $\beta:X\to q$ for which $(\beta_\Gamma)_*(\mu)\in \pi_*(U)$.  By assumption, there is $\lambda\in U$ such that $\pi_*(\lambda)=(\beta_\Gamma)_*(\mu)$. By Lemma \ref{lem:amalgam}, there is a $\Gamma$-invariant probability measure $\omega$ on $X\times q^\Gamma\times p^\Gamma$ that has marginal $\mu$ on $X$ and $\lambda$ on $q^\Gamma\times p^\Gamma$ and satisfies $\beta_\Gamma(x) = y$ for $\omega$-almost-every $(x,y,z) \in X \times q^\Gamma \times p^\Gamma$. By assumption, the factor map given by the projection $\phi:X\times q^\Gamma\times p^\Gamma\to X$ is e.c.\ and we have that $(((\beta \circ \phi) \times \alpha)_\Gamma)_* \omega = \lambda \in U$, where $\alpha : X \times q^\Gamma \times p^\Gamma \rightarrow p$ is the map $\alpha(x,y,z) = z(e)$. Therefore the desired $\gamma$ is obtained by applying Proposition \ref{criterion} with this $\alpha$ together with a finite set $S \subseteq \Gamma$ and an $\epsilon > 0$ that are suitable for ensuring membership in the open set $U$ containing $\lambda$.
    
	We now prove the converse.  Towards this end, fix a factor map $\phi:Y\to X$; we wish to show that this map is e.c.\ using the criterion of Proposition \ref{criterion}.  We thus take measurable maps $\beta:X\to q$ and $\gamma:Y\to p$, finite $S\subseteq \Gamma$, and $\epsilon>0$.  Note then that $((\beta\circ \phi)_\Gamma\times \gamma_\Gamma)_*(\nu)\in \operatorname{Prob}_\Gamma(q^\Gamma\times p^\Gamma)$ and that $S$ and $\epsilon$ determine an open neighborhood $U$ of this measure.  Moreover, 
	$$(\beta_\Gamma)_*(\mu)=(\pi_\Gamma)_*((\beta\circ \phi)_\Gamma\times \gamma_\Gamma)_*(\nu)\in \pi_*(U).$$  By hypothesis, there is $\tilde{\gamma}:X\to p$ such that $((\beta\times \tilde{\gamma})_\Gamma)_*(\mu)\in U$, which verifies the criterion of Proposition \ref{criterion}.
\end{proof}

The following theorem is the main result of this subsection and offers an ergodic-theoretic characterization of the existence of $T_\Gamma^*$.

\begin{thm} \label{prop:openmap}
	$T_\Gamma^*$ exists if and only if for every $p,q \in \bN$, the push-forward map $\pi_*:\Prob_\Gamma(q^\Gamma \times p^\Gamma)\to \Prob_\Gamma(q^\Gamma)$ is an open map with respect to the weak$^*$ topologies.
\end{thm}

\begin{proof}
	($\Rightarrow$) Fix $q, p \in \bN$ and let $\pi : q^\Gamma \times p^\Gamma \rightarrow q^\Gamma$ be the projection map. Consider any measure $\lambda \in \Prob_\Gamma(q^\Gamma \times p^\Gamma)$ and any weak$^*$ neighborhood $W \subseteq \Prob_\Gamma(q^\Gamma \times p^\Gamma)$ of $\lambda$. Set $\nu = \pi_*(\lambda)$ and consider any sequence $\nu_n \in \Prob_\Gamma(q^\Gamma)$ that converges weak$^*$ to $\nu$ and a nonprincipal ultrafilter $\u$ on $\bb N$. It will suffice to show that
	$$\{n \in \bN: \nu_n \in \pi_*(W)\} \in \u.$$

	By Lemma \ref{eccover}, we may take an e.c.\ action $\Gamma \acts (X, \mu)$ having the property that it factors onto $\Gamma \acts (q^\Gamma, \nu_n)$ for every $n \in \bN$. For each $n\in \bb N$, pick a measurable map $\alpha^n : X \rightarrow q$ satisfying $(\alpha^n_\Gamma)_*(\mu) = \nu_n$. Consider the ultrapower action $\Gamma \acts (X, \mu)_\u$ and define $\alpha : X_\u \rightarrow q$ by setting $\alpha([x_n]_\u) = \lim_\u \alpha^n(x_n)$. Then $\alpha$ is measurable and $(\alpha_\Gamma)_*(\mu_\u) = \nu$.
	
	Apply Lemma \ref{lem:amalgam} to get a $\Gamma$-invariant probability measure $\omega$ on $X_\u \times q^\Gamma \times p^\Gamma$ that has marginal $\mu_\u$ on $X_\u$, marginal $\lambda$ on $q^\Gamma \times p^\Gamma$, and satisfies $\alpha_\Gamma(x) = \pi(y,z)$ for $\omega$-a.e. $(x,y,z) \in X_\u \times q^\Gamma \times p^\Gamma$. Define $\bar{\alpha}$ and $\bar{\beta}$ on $X \times q^\Gamma \times p^\Gamma$ by setting $\bar{\alpha}(x,y,z) = \alpha(x) \in q$ and $\bar{\beta}(x,y,z) = z(e) \in p$. By our choice of $\omega$, we have that $\bar{\alpha}_\Gamma(x,y,z) = \alpha_\Gamma(x)$ is equal to $y$ almost-everywhere. Therefore $(\bar{\alpha}_\Gamma \times \bar{\beta}_\Gamma)(x,y,z) = (y,z)$ almost-everywhere and hence $(\bar{\alpha}_\Gamma \times \bar{\beta}_\Gamma)_*(\omega) = \lambda \in W$.
	
	Since $T_\Gamma^*$ exists, the action $\Gamma \acts (X, \mu)_\u$ is e.c.\ Consequently, based on the last sentence of the previous paragraph, there must exist a measurable map $\beta : X_\u \rightarrow p$ satisfying $(\alpha_\Gamma \times \beta_\Gamma)_*(\mu_\u) \in W$. Let $\beta^n : X \rightarrow p$ be a sequence of measurable maps satisfying $\beta([x_n]_\u) = \lim_\u \beta^n(x_n)$. Then we have $(\alpha \times \beta)([x_n]_\u) = \lim_\u (\alpha_n(x_n), \beta_n(x_n))$ and therefore
	$$\lim_{n \rightarrow \u} (\alpha^n_\Gamma \times \beta^n_\Gamma)_*(\mu) = (\alpha_\Gamma \times \beta_\Gamma)_*(\mu_\u) \in W.$$
	Applying $\pi_*$ to both sides we obtain $\{n \in \bN: \nu_n \in \pi_*(W)\} \in \u$ as claimed.
	
	($\Leftarrow$)  Assuming the open mapping condition, we see that the characterization of being e.c.\ in the previous lemma is first-order.  Indeed, first note that it suffices to assume that $U$ is a basic open subset of $\operatorname{Prob}_\Gamma(q^\Gamma\times p^\Gamma)$.  Second, by assumption, $\pi_*(U)$ is an open subset of $\operatorname{Prob}_\Gamma(q^\Gamma)$, and it then suffices to consider basic open subsets of $\operatorname{Prob}_\Gamma(q^\Gamma)$ contained in $\pi_*(U)$.
\end{proof}

\begin{remark}
	As explained in the proof of the previous theorem, the characterization for being e.c.\ given in Lemma \ref{eclemma} does indeed yield an axiomatization of $T_\Gamma^*$ when it exists.
\end{remark}

\section{A concrete axiomatization}\label{concrete}

In this section, we show that, for strongly treeable groups, the model companion $T_\Gamma^*$ exists and has a concrete set of axioms that are ergodic-theoretic in nature. We additionally obtain the same result for treeable groups but with axioms that, while still ergodic-theoretic in nature, are slightly less concrete.

Our axioms for the model companion will rely on two properties we introduce: the definable cocycle property and the extension-MD property. These properties are discussed in the first two subsections below, and in the third subsection we will describe the concrete axiomatization. Lastly, in the final three subsections of this section we will verify that these two properties hold for suitable actions of (strongly) treeable groups.

\subsection{Model companions and the definable cocycle property}\label{subs:definablecocycle}

Building upon our observations about cocycles for e.c.\ actions in Subsection \ref{subs:ec_cocycles} (specifically Lemma \ref{eccocycle}), we first note that the mere existence of the model companion immediately yields a remarkable feature of cochains mapping to finite groups that are close to satisfying the conditions for being a cocycle. If such cochains are said to be ``almost-cocycles,'' then the next result says that if the model companion exists, then every almost-cocycle is near an actual cocycle in some extension (equivalently, it is near a coboundary in some extension).

\begin{prop} \label{prop:mc_cocycle}
Suppose that $T_\Gamma^*$ exists.  Then for every $\epsilon>0$, finite set $S \subseteq \Gamma$, and finite group $K$, there is $\delta>0$ such that, for all actions $\Gamma\acts^a(X,\mu)$ and measurable maps $\sigma:\Gamma\times X\to K$, if $\Cocy_K^{\cal M_a}(B_\sigma)<\delta$, then there is an action $\Gamma\acts (Y,\nu)$, a factor map $\phi:Y\to X$, and a measurable map $\alpha:Y\to K$ such that $$\nu(\{y\in Y \ : \ \alpha(s y) = \sigma(s, \phi(y)) \alpha(y) \text{ for every $s \in S$}\})>1-\epsilon.$$
\end{prop}

\begin{proof}
Suppose, towards a contradiction, that the above condition does not hold for some $\epsilon$, $S$ and $K$.  For each $n\in \bb N$, take a {\pmp} action $\Gamma \acts^{a_n} (X_n, \mu_n)$ and a map $\sigma_n : \Gamma \times X_n \rightarrow K$ such that $\Cocy^{\cal M_{a_n}}_K(B_{\sigma_n})<\frac{1}{n}$ and yet, for every extension $Y$ of $X_n$ and $\alpha : Y \rightarrow K$, the set of $y$ satisfying $\alpha(s y) = \sigma_n(s, \phi(y)) \alpha(y)$ for all $s\in S$ has measure at most $1-\epsilon$. For each $n\in \bb N$, take an e.c.\ action $\Gamma \acts^{b_n} (Z_n, \eta_n)$ that factors onto $\Gamma \acts^{a_n} (X_n, \mu_n)$, say via the map $\phi_n$. Let $\u$ be a nonprincipal ultrafilter on $\bb N$ and write $\Gamma\acts^b(Z, \eta)$ for the ultraproduct of the actions $\Gamma\acts^{b_n}(Z_n,\eta_n)$.  Define $\sigma : \Gamma \times Z \rightarrow K$ by setting $\sigma_\gamma^{-1}(k):=[\phi_n^{-1} \circ (\sigma_n)_\gamma^{-1}(k)]_\u$ for each $k\in K$.  Then $\Cocy^{\cal M_b}_K(B_\sigma) = 0$ and therefore $\sigma$ is a cocycle for the ultraproduct action $\Gamma\acts^b Z$.  Since $T_\Gamma^*$ exists, $\Gamma \acts^b (Z, \eta)$ is e.c.\, whence Lemma \ref{eccocycle} implies that there is a measurable map $\alpha : Z \rightarrow K$ such that the measure of the set of $z$ with $\alpha(s z) = \sigma(s, z) \alpha(z)$ for all $s \in S$ has measure strictly greater than $1-\epsilon$. Pick a sequence of maps $\alpha_n : Z_n \rightarrow K$ such that $\alpha$ is the ultralimit of the $\alpha_n$'s, that is, that $\alpha^{-1}(k)=[\alpha_n^{-1}(k)]_\u$ for all $k\in K$.  It follows that there is $n\in \bb N$ such that the set of $z\in Z_n$ satisfying $\alpha_n(s z) = \sigma_n(s, \phi_n(z)) \alpha_n(z)$ for all $s \in S$ has measure greater than $1-\epsilon$, contradicting the choice of the action $\Gamma\acts^{b_n} (X_n,\mu_n)$.
\end{proof}

The above proposition provides the impetus for a new definition, something we call the definable cocycle property. This property will be used to ensure that statements about cocycles mapping into finite groups are actually first-order, contributing part of our concrete axiomatization of the model companion. We first note some equivalences.

\begin{lem} \label{lem:defcocycle}
    Let $\Gamma$ be a countable group, $K$ a finite group, and fix an $L_\Gamma$-theory $T$ extending $T_\Gamma$ (such as $T_\Gamma$ itself, $T_{\Gamma,free}$, $T_{\Gamma,max}$, or $T_{\Gamma}^*$). Let $\rho_1$ and $\rho_2$ be metrics on $C^1(\Gamma, K^\Gamma)$ and $C^2(\Gamma, K^\Gamma)$ respectively that are compatible with their product topologies. Then the following are equivalent:
    \begin{enumerate}
        \item \label{item:defcocycle1} For any $\epsilon>0$, there is a $\delta>0$ so that, for any $\cal M_a\models T$ and any $B\in \cal M_a^{\Gamma\times K}$, if $\Cocy_K^{\cal M_a}(B)<\delta$, then there is a cocycle $\sigma$ of $a$ such that $d(B,B_\sigma)\leq\epsilon$.
        \item \label{item:defcocycle2} For any $\epsilon > 0$, there is $\delta > 0$ so that, for any action $\Gamma \acts^a (X, \mu)$ with $\cal M_a \models T$ and any equivariant measurable map $c : X \rightarrow C^1(\Gamma, K^\Gamma)$, if
        \[\int \rho_2(\partial c(x), e_{C^2(\Gamma, K^\Gamma)}) \ d \mu < \delta,\]
        then there is a measurable equivariant map $z : X \rightarrow Z^1(\Gamma, K^\Gamma)$ such that
        \[\int \rho_1(c(x),z(x)) \ d \mu < \epsilon.\]
        \item \label{item:defcocycle3} For any family $(\cal M_{a_i})_{i\in I}$ of models of $T$ and any ultrafilter $\u$ on $I$, setting $\cal M:=\prod_\u \cal M_{a_i}$, we have $Z(\Cocy_K^{\cal M})=\prod_\u Z(\Cocy_K^{\cal M_{a_i}})$.
        \item \label{item:defcocycle4} For any $T$-formula $\Phi(x,y)$, with $x$ ranging over sort $\cal M_a^{\Gamma\times K}$, the $T$-functors $\sup_{B\in Z(\Cocy_K)} \Phi(B,y)$ and $\inf_{B\in Z(\Cocy_K)} \Phi(B,y)$ are $T$-formulae again.
    \end{enumerate}
\end{lem}

\begin{proof}
    The equivalence of (\ref{item:defcocycle1}) and (\ref{item:defcocycle2}) is immediate from the definitions. The equivalence of (\ref{item:defcocycle1}), (\ref{item:defcocycle3}) and (\ref{item:defcocycle4}) is a special case of \cite[Theorem 2.13]{spgap}. Note that in \cite[Theorem 2.13]{spgap}, there was no cardinality restriction on the index set (recall Convention \ref{ultraconvention} above); however, for separable theories (such as the $T$ considered here), it is an artifact of the proof that one only needs to assume the preservation under countable ultraproducts.
\end{proof}

Note that condition (\ref{item:defcocycle3}) of the previous lemma is equivalent to the following:  given any family $(\Gamma\acts^{a_i} (X_i,\mu_i))_{i\in I}$ of models of $T$ and any ultrafilter $\u$ on $I$, setting $\Gamma\acts^a (X,\mu)$ to be the ultraproduct action, if $\sigma:\Gamma\times X\to K$ is a cocycle, then there are cocycles $\sigma^i:\Gamma\times X_i\to K$ such that $\sigma_\gamma^{-1}(k)=[(\sigma_\gamma^i)^{-1}(k)]_\u$ for all $\gamma\in \Gamma$ and $k\in K$.  In other words, if we let $\sigma^\u:\Gamma\times X\to K$ denote the corresponding \textbf{ultraproduct cocycle}, that is, the cocycle of the ultraproduct action given by the formula $(\sigma_\u)_\gamma^{-1}(k):=[(\sigma_\gamma^i)^{-1}(k)]_\u$ for all $\gamma\in \Gamma$ and $k\in K$, then condition (2) states that every cocycle of the product action is an ultraproduct cocyle.

When the equivalent conditions of Lemma \ref{lem:defcocycle} hold, we say that $\Cocy_K$ is a \textbf{$T$-definable set}. If $\Cocy_K$ is a $T$-definable set for all finite groups $K$, then we say that $T$ has the \textbf{definable cocycle property}.

\begin{cor}
Suppose that $T_\Gamma^*$ exists.  Then $T_\Gamma^*$ has the definable cocycle property.
\end{cor}

\begin{proof}
This follows immediately from Proposition \ref{prop:mc_cocycle} and the definition of e.c.\ action.
\end{proof}

We will soon establish that $T_{\Gamma, free}$ (resp. $T_{\Gamma, max}$) has the definable cocycle property when $\Gamma$ is strongly treeable (resp. $\Gamma$ is treeable). For now we make the following simple observation:

\begin{lem}\label{freedefcocycle}
If $\bb F$ is any free group, then $T_{\bb F}$ has the definable cocycle property.
\end{lem}

\begin{proof}
We verify condition (\ref{item:defcocycle3}) of Lemma \ref{lem:defcocycle}. Suppose that $\bb F \acts^a X$ is the ultraproduct of the actions $\bb F\acts^{a_i}X_i$ and suppose that $\sigma:\bb F\times X\to K$ is a cocycle for $a$.  If $S$ is a free generating set for $\bb F$, then for each $s \in S$ and $k\in K$, we can write $\pi_{s,k}(B_\sigma)=[X_{i,s,k}]_\u$, where, for each $i\in I$ and $s \in S$, $(X_{i,s,k})_{k \in K}$ is a measurable partition of $X_i$.  If we define $\sigma_{i,0}:S\times X_i\to K$ by setting $\sigma_{i,0}(s,x):=k$ when $x\in X_{i,s,k}$, then $\sigma_{i,0}$ extends uniquely to a cocycle $\sigma_i:\bb F\times X_i\to K$ in such a way that $[\sigma_i]_\u=\sigma$. 
\end{proof}

\subsection{Finite-to-one extensions and the extension-MD property}\label{subs:weakMD}

In the last subsection, we observed that when $T$ has the definable cocycle property, the $K$-valued cocycles form a definable set. We also know that cocycles to finite groups correspond to finite-to-one extensions (via the skew product construction and Rokhlin's skew product theorem). This leads us to consider the nature of finite-to-one extensions and their relationship with e.c.\ actions. We start with the following theorem, which characterizes those actions that are e.c.\ for finite-to-one extensions.

\begin{thm}\label{concretetheorem}
An action $\Gamma\acts(X,\mu)$ is e.c.\ for finite-to-one extensions if and only if: for every integer $k$, every cocycle $\sigma : \Gamma \times X \rightarrow \operatorname{Sym}(k)$, every finite set $F\subseteq \Gamma$, every $\epsilon > 0$, every integer $q$, and every map $\beta : X \to q$, there exists a map $\alpha : X \to k$ such that:
\begin{enumerate}
    \item $\alpha_*\mu$ is the normalized counting measure on $k$.
    \item $\alpha$ is $\epsilon$-independent with $\beta$, that is, for all $i< k$ and all $j<q$, we have
    $$\left|\mu(\alpha^{-1}(i)\cap \beta^{-1}(j))-\mu(\alpha^{-1}(i))\mu(\beta^{-1}(j))\right|\leq \epsilon.$$
    \item $\mu(\{x\in X \ : \alpha(f x) = \sigma(f,x)(\alpha(x)) \text{ for every }f \in F\})\geq 1-\epsilon$.
\end{enumerate}
\end{thm}

\begin{proof}
First assume that $\Gamma\acts (X,\mu)$ is e.c.\ for finite-to-one extensions and fix $k$, $\sigma$, $F$, $\epsilon$, $q$, and $\beta$ as above. Let $u_k$ denote the normalized counting measure on $k$ and let $(Y, \nu) = (X \times_\sigma k, \mu \times u_k)$ be the skew-product extension of $\Gamma\acts (X, \mu)$ with respect to $\sigma$.
Let $\phi : Y \rightarrow X$ and $\hat{\alpha} : Y \rightarrow k$ be the coordinate projection maps. Notice that $\hat{\alpha}_* \nu = u_k$ and that $\hat{\alpha}$ is independent with $\beta \circ \phi$. Also notice that $\hat{\alpha}(f \cdot y) = \sigma(f, \phi(y)) (\hat{\alpha}(y))$ for every $f \in F$ and $y \in Y$, or equivalently, writing $S_{f, \tau} = \{x \in X : \sigma(f, x) = \tau\}$,
\[Y = \bigcup_{i \in k} \bigcap_{f \in F} \bigcup_{\tau \in \operatorname{Sym}(k)} \hat{\alpha}^{-1}(i) \cap \phi^{-1}(S_{f,\tau}) \cap f^{-1} \cdot \hat{\alpha}^{-1}(\tau(i)).\]
Since the action of $\Gamma\acts (X, \mu)$ is e.c.\ for finite-to-one extensions, there exists a map $\alpha : X \rightarrow k$ such that $\alpha_*\mu$ is as close to $u_k$ as desired, $\alpha$ is $(\epsilon/2)$-independent with $\beta$, and the set
\[\bigcup_{i \in k} \bigcap_{f \in F} \bigcup_{\tau \in \operatorname{Sym}(k)} \alpha^{-1}(i) \cap S_{f,\tau} \cap f^{-1} \cdot \alpha^{-1}(\tau(i)),\]
that is, the set of $x \in X$ satisfying $\alpha(fx) = \sigma(f,x)(\alpha(x))$ for all $f\in F$, has measure greater than $1-(\epsilon/2)$. If we choose $\alpha$ so that $\alpha_* \mu$ is sufficiently close to $u_k$, then we can perturb $\alpha$ so as to satisfy items (1)-(3) above.

Now assume that $\Gamma\acts (X,\mu)$ satisfies items (1)-(3) for any choice of $k$, $\sigma$, $F$, $\epsilon$, $q$, and $\beta$ as in the statement of the proposition.  We show that $\Gamma\acts (X,\mu)$ is e.c.\ for finite-to-one extensions.  By the Rohklin skew-product theorem and the ergodic decomposition, it suffices to show that $\Gamma\acts (X,\mu)$ is e.c.\ for skew-product extensions associated with finite groups.  Thus, fix a cocycle $\sigma:\Gamma\times X\to \operatorname{Sym}(k)$, let $u_k$ be normalized counting measure on $k$, and consider the associated skew-product extension $(Y, \nu) = (X\times_\sigma k, \mu \times u_k)$.  We wish to show that the projection map $\phi : Y \rightarrow X$ is e.c.\ Let $p, q \in \bN$, let $\gamma : Y \rightarrow p$ and $\beta : X \rightarrow q$ be measurable maps, let $\epsilon > 0$, and let $F \subseteq \Gamma$ be finite. Our goal is to find a measurable map $\hat{\gamma} : X \rightarrow p$ satisfying
\begin{equation} \label{eqn:finite-to-one0}
 |\mu(\hat{\gamma}_F^{-1}(\pi) \cap \beta^{-1}(j)) - \nu(\gamma_F^{-1}(\pi) \cap \phi^{-1}(\beta^{-1}(j))| < \epsilon
\end{equation}

for all $\pi \in p^F$ and all $j \in q$.

For $\pi \in p^F$ and $i \in k$, let $D_{\pi,i}$ be the set of $x \in X$ satisfying $\gamma_F(x,i) = \pi$. Then
\begin{equation} \label{eqn:finite-to-one1}
\nu(\gamma_F^{-1}(\pi) \cap \phi^{-1}(\beta^{-1}(j))) = \sum_{i \in k} \frac{1}{k} \mu(D_{\pi,i} \cap \beta^{-1}(j)).
\end{equation}
By our assumption on the action $\Gamma \acts (X, \mu)$, we can pick a measurable map $\alpha : X \rightarrow k$ satisfying $\mu(A) > 1 - \frac{\epsilon}{2}$, where
\[A = \{x \in X : \forall f \in F \ \alpha(f^{-1} x) = \sigma(f^{-1}, x)(\alpha(x)),\]
and satisfying
\begin{equation} \label{eqn:finite-to-one2}
\left| \sum_{i \in k} \mu(\alpha^{-1}(i) \cap D_{\pi,i} \cap \beta^{-1}(j)) - \sum_{i \in k} \frac{1}{k} \mu(D_{\pi, i} \cap \beta^{-1}(j) \right| < \frac{\epsilon}{2}
\end{equation}
for every $\pi \in p^F$ and $j \in q$. Define $\hat{\gamma} : X \rightarrow p$ by $\hat{\gamma}(x) = \gamma(x, \alpha(x))$. Then for $x \in A$, we have
\begin{align*}
    \hat{\gamma}_F(x)(f) & = \hat{\gamma}(f^{-1} x)\\
    & = \gamma(f^{-1} x, \alpha(f^{-1} x))\\
    & = \gamma(f^{-1} x, \sigma(f^{-1},x)(\alpha(x)))\\
    & = \gamma(f^{-1} \cdot (x, \alpha(x)))\\
    & = \gamma_F(x,\alpha(x))(f).
\end{align*}
As a result, for $x \in A$ we have $\hat{\gamma}_F(x) = \pi$ if and only if there is $i \in k$ with $x \in \alpha^{-1}(i) \cap D_{\pi, i}$. Therefore
\begin{equation} \label{eqn:finite-to-one3}
\left|\mu(\hat{\gamma}_F^{-1}(\pi) \cap \beta^{-1}(j)) - \sum_{i \in k} \mu(\alpha^{-1}(i) \cap D_{\pi,i} \cap \beta^{-1}(j)) \right| < \mu(X \setminus A) < \frac{\epsilon}{2}.
\end{equation}
Combining equations \eqref{eqn:finite-to-one1}, \eqref{eqn:finite-to-one2}, and \eqref{eqn:finite-to-one3} shows that \eqref{eqn:finite-to-one0} holds.
\end{proof}

Let $\Gamma \acts^a (X, \mu)$ be a {\pmp} action, let $\lambda$ be Lebesgue measure, and let $\Gamma \acts^{\id} [0,1]$ be the trivial action fixing every point. We call $\Gamma \acts^{a \times \id} (X \times [0,1], \mu \times \lambda)$ the trivial extension of $\Gamma \acts^a (X, \mu)$ having atomless fibers.

\begin{cor} \label{cor:simple}
	Let $\Gamma \acts^a (X, \mu)$ be a {\pmp} action, and suppose that $\Gamma \acts^{a \times \id} (X \times [0,1], \mu \times \lambda)$ is an e.c.\ extension of $\Gamma \acts^a (X, \mu)$. Then $\Gamma \acts^a (X, \mu)$ is e.c.\ for finite-to-one extensions if and only if $B^1(a, \operatorname{Sym}(k))$ is dense in $Z^1(a, \operatorname{Sym}(k))$ for every $k \in \bb N$.
\end{cor}

\begin{proof}
	We will apply Theorem \ref{concretetheorem}. So let $k, q \in \bb N$, let $\sigma \in Z^1(a, \operatorname{Sym}(k))$, let $F \subseteq \Gamma$ be finite, let $\epsilon > 0$, and let $\beta : X \rightarrow q$ be measurable.
	
	Let $u_k$ be the normalized counting measure on $k = \{0, \ldots, k-1\}$. Define $(\tilde X, \tilde \mu) = (X \times k, \mu \times u_k)$ and define $\tilde \beta : \tilde X \rightarrow q$ and $\tilde \sigma : \Gamma \times \tilde X \rightarrow \operatorname{Sym}(k)$ by
	\[\tilde \beta(x,i) = \beta(x) \quad \text{and} \quad \tilde \sigma(\gamma, (x,i)) = \sigma(\gamma, x).\]
	We know that $\Gamma \acts^{a \times \id} (\tilde X, \tilde \mu)$ is an e.c.\ extension of $\Gamma \acts^a (X, \mu)$ since this extension is intermediary to the trivial extension having atomless fibers. Therefore it suffices to find a measurable map $\tilde \alpha : \tilde X \rightarrow k$ satisfying conditions (1), (2), and (3) of Theorem
	\ref{concretetheorem} with $X$, $\mu$, $\sigma$, $\alpha$, $\beta$ replaced by $\tilde X$, $\tilde \mu$, $\tilde \sigma$, $\tilde \alpha$, $\tilde \beta$.
	
	Since we are assuming $B^1(a, \operatorname{Sym}(k))$ is dense in $Z^1(a, \operatorname{Sym}(k))$, there is a measurable map $\alpha : X \rightarrow \operatorname{Sym}(k)$ such that the set
	\[Y = \{x \in X : \sigma(f, x) = \alpha(f^a \cdot x) \alpha(x)^{-1} \text{ for all } f \in F\}\]
	satisfies $\mu(Y) > 1 - \epsilon$. Now define $\tilde \alpha : \tilde X \rightarrow k$ by
	\[\tilde \alpha(x,i) = \alpha(x)(i).\]
	Then we have
	\[\tilde \alpha_* \tilde \mu = \int \alpha(x)_* u_k \ d \mu = \int u_k \ d \mu = u_k,\]
	so (1) is satisfied, and
	\begin{align*}
		\tilde \mu(\tilde \alpha^{-1}(i) \cap \tilde \beta^{-1}(j) ) & = \tilde \mu(\{(x, \alpha(x)^{-1}(i)) : x \in \beta^{-1}(j)\})\\
		& = \frac{1}{k} \cdot \mu(\beta^{-1}(j))\\
		& = \tilde \mu( \tilde \alpha^{-1}(i)) \tilde \mu (\tilde \beta^{-1}(j)),
	\end{align*}
	so (2) is satisfied. Finally, $\tilde \mu(Y \times k) = \mu(Y) > 1 - \epsilon$ and for every $(x,i) \in Y \times k$ and $f \in F$ we have
	\[\tilde \alpha(f^{a \times \id} (x,i)) = \alpha(f^a \cdot x)(i) = \sigma(f, x) \circ \alpha(x)(i) = \tilde \sigma(f, (x,i)) (\tilde \alpha(x,i)).\]
	Thus (3) is satisfied.
\end{proof}

The previous results motivate the consideration of actions having the property that all of their extensions can be approximated by their finite-to-one extensions. Drawing a parallel with property MD for groups, we make the following definition.

\begin{defn} \label{def:weakMD}
We say that a {\pmp} action $\Gamma \acts^a (X, \mu)$ on a standard probability space is \textbf{extension-MD} if for any (equivalently, every) non-atomic standard probability space $(Y, \nu)$, the set of finite-to-one $a$-extensions are dense in $F_a(\Gamma, Y,X)$ (as defined in Subsection \ref{subsectionspace}). If every free {\pmp} action of $\Gamma$ on a standard probability space is extension-MD, then we say that $\Gamma$ has the \textbf{extension-MD property}.
\end{defn}

Similar to Proposition \ref{criterion}, we will rely on the following characterization of this property.

\begin{prop} \label{weakMDcriterion}
    The action $\Gamma \acts^a (X, \mu)$ is extension-MD if and only if $(X, \mu)$ is a standard probability space and, whenever $\Gamma \acts^b (Y, \nu)$ is a {\pmp} action extending $\Gamma \acts^a (X, \mu)$, say via $\phi : Y \rightarrow X$, $p, q \in \bb N$, $\alpha : Y \to p$, and $\beta : X \to q$ are measurable, $F \subseteq \Gamma$ is finite, and $\epsilon > 0$, there is a {\pmp} action $\Gamma \acts^{b'} (Y', \nu')$ extending $\Gamma \acts^a (X, \mu)$, say via $\phi' : Y' \rightarrow X$, and a measurable map $\alpha' : Y' \to p$ such that $\phi'$ is almost-everywhere finite-to-one and
    \begin{equation} \label{eqn:weakMDcriterion}
        \Big| \nu \Big(\alpha_F^{-1}(\pi) \cap \phi^{-1}(\beta^{-1}(j)) \Big) - \nu' \Big((\alpha')_F^{-1}(\pi) \cap (\phi')^{-1}(\beta^{-1}(j)) \Big) \Big| < \epsilon
    \end{equation}
    for all $\pi \in p^F$ and $j \in q$.
\end{prop}

\begin{proof}
    If $\Gamma \acts^a (X, \mu)$ is extension-MD, then $(X, \mu)$ is a standard probability space and it is easy to see that it is then enough to check the above condition in the cases where $(Y, \nu)$ is a standard non-atomic probability space. In this case, every $(\phi', b')$ in some open neighborhood of $(\phi, b)$ will satisfy (\ref{eqn:weakMDcriterion}) using the same map $\alpha : Y \to p$ (in this case, take care to note that $\alpha_F$ depends on the action being considered and can be written $\alpha_{b(F)}$ and $\alpha_{b'(F)}$ for clarity). The assumption that the action $a$ is extension-MD implies that this open set contains a finite-to-one extension of $a$. Conversely, suppose the above condition holds. Consider a standard non-atomic probability space $(Y, \nu)$, an $a$-extension $(\phi, b) \in F_a(\Gamma, Y, X)$, and an open neighborhood $U^{\cal A, \cal B, F, \epsilon}_a(\phi, b)$. Choose $\alpha : Y \to p$ and $\beta : X \to q$ so that $\cal A = \{\alpha^{-1}(i) : i \in p\}$ and $\cal B = \{\beta^{-1}(j) : j \in q\}$. Let $\Gamma \acts^{b'} (Y', \nu')$, $\phi' : Y' \rightarrow X$, and $\alpha' : Y' \to p$ satisfy (\ref{eqn:weakMDcriterion}) for all $\pi \in p^F$ and $j \in q$. We can assume that $(Y', \nu')$ is non-atomic, and by passing to a factor of $Y'$ for which $\phi'$ and $\alpha'$ remain measurable, we may further assume that $(Y', \nu')$ is a standard probability space. Notice that the sets appearing in (\ref{eqn:weakMDcriterion}) partition $Y$ and $Y'$ as $\pi \in p^F$ and $j \in q$ vary. Since (\ref{eqn:weakMDcriterion}) implies that the measures of these pieces are within $\epsilon$ of one another, there is an isomorphism of probability spaces $S : (Y', \nu') \rightarrow (Y, \nu)$ that matches the respective pieces of the partitions, for each $\pi \in p^F$ and $j \in q$, up to an error (symmetric difference) of measure $\epsilon$. The resulting pair $(S \cdot \phi', S \cdot b')$ will belong to $U^{\cal A, \cal B, F, \epsilon}_a(\phi, b)$, and since $\phi'$ is finite-to-one and $S$ is an isomorphism, $S \cdot \phi'$ will be finite-to-one as well.
\end{proof}

The above proposition allows us to extend the definition of extension-MD to actions on non-standard probability spaces: we say that a {\pmp} action $\Gamma \acts^a (X, \mu)$ is \textbf{extension-MD} if it satisfies the condition of the above proposition. It is easily seen then that if $\Gamma$ has the extension-MD property as defined in Definition \ref{def:weakMD} then all free {\pmp} actions of $\Gamma$ (on both standard and non-standard spaces) are extension-MD.

A trivial consequence of the above proposition and Proposition \ref{criterion} is that every e.c.\ action is automatically extension-MD. The following is also clear.

\begin{lem}\label{weakMDlemma}
An extension-MD action is e.c.\ if and only if it is e.c.\ for finite-to-one extensions. 
\end{lem}

Before ending this subsection, we make an additional observation about extension-MD actions and show that free groups have the extension-MD property.

\begin{lem} \label{lem:exMD}
    Let $\Gamma \acts^a (X, \mu)$ be a {\pmp} action. Let $G = \Aut([0,1], \lambda)$, where $\lambda$ is Lebesgue measure, equipped with the weak topology, and let $\rho_1$ be a metric on $C^1(\Gamma, G^\Gamma)$ compatible with its product topology. Write $Z^1(a, G)_{fin}$ for the set of all measurable cocycles on $a$ that map to a finite subgroup of $G$. Then the following are equivalent:
    \begin{enumerate}
        \item \label{item:exMD2} $Z^1(a, G)_{fin}$ is dense in $Z^1(a, G)$;
        \item \label{item:exMD3} for every measurable cocycle $\sigma : \Gamma \times X \rightarrow G$, every finite set $F \subseteq \Gamma$, every $\epsilon > 0$, and every finite tuple $(C_i)_{i \in n}$ of Borel subsets of $[0,1]$, there is a finite subgroup $H \leq G$ and a measurable cocycle $\sigma' : \Gamma \times X \rightarrow H$ such that, for all $\gamma \in F$ and all $i \in n$:
            \[ \int \lambda \big(\sigma(\gamma,x)^{-1}(C_i) \symd \sigma'(\gamma, x)^{-1}(C_i) \big) \ d \mu < \epsilon;\]
        \item \label{item:exMD4} for every measurable equivariant map $z : X \rightarrow Z^1(\Gamma, G^\Gamma)$ and every $\epsilon > 0$, there is a finite subgroup $H \leq G$ and a measurable equivariant map $z' : X \rightarrow Z^1(\Gamma, H^\Gamma)$ satisfying $\int \rho_1(z(x),z'(x)) \ d \mu < \epsilon$.
    \end{enumerate}
    Moreover, the above properties imply that the action $\Gamma \acts^a (X, \mu)$ is extension-MD.
\end{lem}

\begin{proof}
    The equivalence of (\ref{item:exMD2}), (\ref{item:exMD3}), and (\ref{item:exMD4}) is immediate from the definitions. So it will be enough to show that (\ref{item:exMD3}) implies that $\Gamma \acts^a (X, \mu)$ is extension-MD. We will do this by checking the condition in Proposition \ref{weakMDcriterion}.

    Let $\Gamma \acts^b (Y, \nu)$ be a {\pmp} extension of $\Gamma \acts^a (X, \mu)$, say via the map $\phi : (Y, \nu) \rightarrow (X, \mu)$. By disintegrating $\nu$ with respect to $\phi$ we obtain a measurable map $x \in X \mapsto \nu_x \in \Prob(Y)$ satisfying $\nu_x(\phi^{-1}(x)) = 1$ for $\mu$-almost-every $x$ and $\int \nu_x \ d \mu = \nu$. Since, in verifying the criterion in Proposition \ref{weakMDcriterion}, we could let $\Gamma$ act trivially on $([0, 1], \lambda)$, where $\lambda$ is the Lebesgue measure, and lift any $\alpha : Y \to p$ to the direct product action $\Gamma \acts (Y \times [0,1], \nu \times \lambda)$, we see that without loss of generality we may assume that $\nu_x$ is non-atomic for $\mu$-almost-every $x \in X$. Then by the Rokhlin skew-product theorem we can assume that $(Y, \nu) = (X \times [0,1], \mu \times \lambda)$, that $\phi$ is the projection map to $X$, and that there is a measurable cocycle $\sigma : \Gamma \times X \rightarrow \Aut([0,1], \lambda)$ so that the action $b$ is the skew-product action given by the formula
	\[\gamma^b \cdot (x, r) = (\gamma^a \cdot x, \sigma(\gamma, x)(r)).\]

    Now consider a pair of measurable maps $\alpha : X \times [0,1] \rightarrow p$ and $\beta : X \rightarrow q$, a finite $F \subseteq \Gamma$ and an $\epsilon > 0$. Since $\Malg([0,1], \lambda)$ is complete, we can find a finite algebra $\cal C$ of Borel subsets of $[0,1]$ and a measurable function $\tilde \alpha : X \times [0,1] \rightarrow p$ satisfying $\{r \in [0,1] : \tilde \alpha(x,r) = i\} \in \cal C$ for every $x \in X$ and $(\mu \times \lambda)(\{(x,r) : \alpha(x,r) \neq \tilde \alpha(x,r)\}) < \epsilon / (2|F|)$. Notice that for every $\pi \in p^F$ and $j \in q$
    \begin{equation} \label{eqn:exMD1}
        |(\mu \times \lambda)(\alpha_F^{-1}(\pi) \cap (\beta \circ \phi)^{-1}(j)) - (\mu \times \lambda)(\tilde \alpha_F^{-1}(\pi) \cap (\beta \circ \phi)^{-1}(j)) < \frac{\epsilon}{2}
    \end{equation}
    Define $A_i^x = \{r \in [0,1] : \tilde \alpha(x,r) = i\} \in \cal C$ for each $x \in X$ and $i \in p$.

    Let $\sigma' : \Gamma \times X \rightarrow G$ be a measurable cocycle that takes values in a finite subgroup of $G$ and satisfies
    \[\sum_{\gamma \in F^{-1}} \sum_{C \in \cal C} \int \lambda \big(\sigma(\gamma,x)^{-1}(C) \symd \sigma'(\gamma, x)^{-1}(C) \big) \ d \mu < \frac{\epsilon}{2}.\]
    Let $b'$ be the skew-product action of $\Gamma$ on $(X \times [0,1], \mu \times \lambda)$ given by $\sigma'$, that is, $\gamma^{b'} \cdot (x, r) = (\gamma^a \cdot x, \sigma'(\gamma, x)(r))$, and regard $b'$ as an extension of $a$ via the projection map $\phi$. We will write $\tilde \alpha_{b(F)}$ and $\tilde \alpha_{b'(F)}$ in place of $\tilde \alpha_F$ in order to clarify the action being considered. Then for every $\pi \in p^F$ and $j \in q$ we have
    \begin{align*}
        & \big|(\mu \times \lambda)\big(\tilde \alpha_{b(F)}^{-1}(\pi) \cap (\beta \circ \phi)^{-1}(j)\big) - (\mu \times \lambda)\big(\tilde \alpha_{b'(F)}^{-1}(\pi) \cap (\beta \circ \phi)^{-1}(j)\big)\big|\\
        & \leq (\mu \times \lambda)\big((\tilde \alpha_{b(F)}^{-1}(\pi) \symd \tilde \alpha_{b'(F)}^{-1}(\pi)) \cap (\beta \circ \phi)^{-1}(j)\big)\\
        & = \int_{\beta^{-1}(j)} \lambda \left(\Bigg(\bigcap_{\gamma \in F^{-1}} \sigma(\gamma,x)^{-1}(A_{\pi(\gamma^{-1})}^{\gamma^a \cdot x}) \Bigg) \symd \Bigg(\bigcap_{\gamma \in F^{-1}} \sigma'(\gamma, x)^{-1}(A_{\pi(\gamma^{-1})}^{\gamma^a \cdot x}) \Bigg) \right) \ d \mu\\
        & \leq \sum_{\gamma \in F^{-1}} \int_X \lambda \Big( \sigma(\gamma,x)^{-1}(A_{\pi(\gamma^{-1})}^{\gamma^a \cdot x}) \symd \sigma'(\gamma,x)^{-1}(A_{\pi(\gamma^{-1})}^{\gamma^a \cdot x} \Big) \ d \mu < \frac{\epsilon}{2}.
    \end{align*}
    Consequently, for every $\pi \in p^F$ and $j \in q$ we have
        \[|(\mu \times \lambda)(\alpha_{b(F)}^{-1}(\pi) \cap (\beta \circ \phi)^{-1}(j)) - (\mu \times \lambda)(\tilde \alpha_{b'(F)}^{-1}(\pi) \cap (\beta \circ \phi)^{-1}(j)) < \epsilon.\]
    Finally, say $\sigma'$ takes values in the finite subgroup $H \leq G$, and let $\cal C'$ be the finite $H$-invariant algebra generated by $\cal C$. Let $(U, \rho)$ be a finite probability space and $\psi : ([0,1], \lambda) \rightarrow (U, \rho)$ a measure-preserving map such that the $\psi$-preimage of the powerset of $U$ is $\cal C'$. Since $H$ descends to the group of measure-preserving automorphisms of $(U, \rho)$, the skew-product action $b'$ descends to an intermediary extension $\Gamma \acts^c (X \times U, \mu \times \rho)$ of $a$, and the map $\tilde \alpha$ descends as well. Since $U$ is finite, this is a finite-to-one extension of $a$ that, together with the descended map $\tilde \alpha$, approximates the extension $b$ and the map $\alpha$ relative to the parameters $F$ and $\epsilon$, in accordance with Proposition \ref{weakMDcriterion}.
\end{proof}

We remark that we do not know if the converse of the previous lemma holds. If $\Gamma \acts^a (X, \mu)$ is extension-MD, then one can show that, for every measurable cocycle $\sigma : \Gamma \times X \rightarrow \Aut([0,1], \lambda)$, every finite $F \subseteq \Gamma$, every $\epsilon > 0$, every finite tuple $(C_i)_{i \in n}$ of Borel subsets of $[0,1]$, and every finite tuple $(A_k)_{k \in m}$ of Borel subsets of $X$, there is a finite subgroup $H \leq \Aut([0,1], \lambda)$, and a measurable cocycle $\sigma' : \Gamma \times X \rightarrow H$ such that for all $\gamma \in F$, $i,j \in n$ and $k \in m$
\[\left| \int_{A_k} \lambda(C_j \cap \sigma(\gamma, x)^{-1}(C_i)) \ d \mu - \int_{A_k} \lambda(C_j \cap \sigma'(\gamma, x)^{-1}(C_i)) \ d \mu \right| < \epsilon.\]
Compare this with (\ref{item:exMD3}) above.

\begin{lem}\label{freeweakMD}
    If $\bb F$ is a free group, then $\bb F$ has the extension-MD property.
\end{lem}

\begin{proof}
    Suppose that $\bb F$ is freely generated by the set $S$, and let $\bb F \acts^a (X, \mu)$ be a free {\pmp} action. Every measurable cocycle $\sigma : \bb F \times X \rightarrow \Aut([0,1], \lambda)$ is uniquely determined by its restriction to $S \times X$, and conversely every measurable map from $S \times X$ to $\Aut([0,1], \lambda)$ uniquely determines a cocycle. Moreover, if we topologize the collection of measurable functions from $S \times X$ to $\Aut([0,1], \lambda)$ so that $f_n \rightarrow f$ if and only if $f_n(s, \cdot)$ converges to $f(s, \cdot)$ in measure for every $s \in S$, then this correspondence is a homeomorphism. Since $\Aut([0,1], \lambda)$ admits an increasing sequence of finite subgroups whose union is dense (that is, consider the subgroup $H_n$ of automorphisms that fix $1$ and permute the intervals $[\frac{i}{2^n},\frac{i+1}{2^n})$ via order-preserving isometries), it is immediate that the measurable maps from $S \times X$ to finite subgroups of $\Aut([0,1], \lambda)$ are dense in the space of all measurable maps from $S \times X$ to $\Aut([0,1], \lambda)$. Therefore condition (\ref{item:exMD2}) of Lemma \ref{lem:exMD} is satisfied.
\end{proof}

\subsection{The axiomatization}

In this section, we combine the ideas from the previous subsections to show that, for a certain class of groups, the model companion $T_\Gamma^*$ exists and has a concrete set of axioms that are ergodic-theoretic in nature. 

\begin{thm}\label{mainaxioms1}
Suppose that $\Gamma$ has the extension-MD property and that $T_{\Gamma,free}$ has the definable cocycle property.  Then $T_\Gamma^*$ exists.
\end{thm}

\begin{proof}
For each $k,q\geq 1$ and each finite $F\subseteq \Gamma$, let $\theta_{k,q,F}$ be the $T_{\Gamma,free}$-sentence $$\sup_{A\in \operatorname{Part}_q}\sup_{B\in Z(\Cocy_{\operatorname{Sym}(k)})}\inf_{C\in \operatorname{Part}_k}\max(\varphi_1(C),\varphi_2(A,C),\varphi_3(B,C)),$$ where:
\begin{enumerate}
    \item $\varphi_1(C)$ is the formula
    $$\max_{i=1,\ldots,k}\left|\mu(C_i)-\frac{1}{k}\right|,$$
    \item $\varphi_2(A,C)$ is the formula
    $$\max_{i=1,\ldots,q}\max_{j=1,\ldots,k}|\mu(A_i\cap C_j)-\mu(A_i)\mu(C_j)|,$$
    \item $\varphi_3(B,C)$ is the formula
    $$\max_{\gamma\in F}\max_{i=1,\ldots,k}\max_{\rho\in\operatorname{Sym}(k)}d\left(\pi_{\gamma,\rho}(B) \cap C_i,\pi_{\gamma,\rho}(B)\cap \gamma^{-1}C_{\rho(i)}\right)$$
\end{enumerate}
Here, $\operatorname{Part}_k$ denotes the $T_\Gamma$-definable set of $k$-tuples that form a partition of the measure space.  Since $T_{\Gamma,free}$ has the definable cocycle property, the above sentences are indeed (equivalent modulo $T_{\Gamma,free}$ to) $T_{\Gamma,free}$-sentences.  

Set $T_\Gamma^*$ to be $T_{\Gamma,free}$ together with all of the sentences $\theta_{k,q,F}$.  Theorem \ref{concretetheorem}, together with Lemma \ref{weakMDlemma}, shows that $T_{\Gamma}^*$ does indeed axiomatize the class of e.c.\ models of $T_\Gamma$.
\end{proof}

Recalling Lemmas \ref{freedefcocycle} and \ref{freeweakMD}, we see that the concrete axioms in the previous theorem yield an alternative axiomatization for $T_{\bb F}^*$ for free groups $\bb F$.

The exact same proof as that of Theorem \ref{mainaxioms1} yields the following theorem:

\begin{thm}\label{mainaxioms2}
If $T_{\Gamma,max}$ has the definable cocycle property and every action of $\Gamma$ that is maximal with respect to weak containment is extension-MD, then $T_\Gamma^*$ exists.  Moreover, the axioms for $T_{\Gamma}^*$ are given by the axioms for $T_{\Gamma,max}$ together with the sentences $\theta_{k,q,F}$ from the proof of Theorem \ref{mainaxioms1}.
\end{thm}

In the remaining three subsections, we will show that if $\Gamma$ is strongly treeable, then $\Gamma$ has the extension-MD property and $T_{\Gamma, free}$ has the definable cocycle property. Thus Theorem \ref{mainaxioms1} yields a concrete axiomatization for $T_\Gamma^*$ when $\Gamma$ is strongly treeable. We will additionally show that when $\Gamma$ is only assumed to be treeable, $T_{\Gamma, max}$ has the definable cocycle property and all actions of $\Gamma$ that are maximal with respect to weak containment are extension-MD. As a result, Theorem \ref{mainaxioms2} provides an axiomatization for $T_\Gamma^*$ when $\Gamma$ is treeable.

\subsection{Trees and group cohomology}\label{subs:groupcoh}

We now turn our attention towards the goal of (partially) generalizing Lemmas \ref{freedefcocycle} and \ref{freeweakMD} to (strongly) treeable groups. The proofs of these prior results for a free group $\bb F$ were based on the simple observation that cocycles $\sigma : \bb F \times X \rightarrow G$ are in one-to-one correspondence with measurable maps $S \times X \rightarrow G$, where $S$ is a free generating set for $\bb F$. In fact, the map that restricts cochains $\theta : \bb F \times X \rightarrow G$ to the domain $S \times X$, when combined with the one-to-one correspondence just mentioned, provides a retraction from the space of cochains to the space of cocycles. (Recall that if $Y$ is a topological space and $A \subseteq Y$, then a map $r : Y \rightarrow A$ is called a retraction if it is continuous and restricts to the identity map on $A$.)

The lemma below generalizes this retraction phenomena to other groups by replacing the role of the free generating set $S$ with the set of edges of a tree $T \in \cal T(\Gamma)$.

\begin{lem} \label{lem:retract}
	Let $\Gamma$ be a countable group and $G$ a Polish group. There is a map assigning each tree $T \in \cal T(\Gamma)$ a retraction $r_T : C^1(\Gamma, G^\Gamma) \rightarrow Z^1(\Gamma, G^\Gamma)$ such that:
	\begin{enumerate}
		\item \label{item:retract1} $r_{\gamma^d \cdot T}(\gamma^t \cdot c) = \gamma^t \cdot r_T(c)$ for all $T \in \cal T(\Gamma)$, $c \in C^1(\Gamma, G^\Gamma)$, and $\gamma \in \Gamma$; and
		\item \label{item:retract2} for every $\alpha, \beta \in \Gamma$ there is a clopen partition $\cal U$ of $\cal T(\Gamma)$ so that $r_T(c)(\beta)(\alpha) = r_{T'}(c)(\beta)(\alpha)$ for all $c \in C^1(\Gamma, G^\Gamma)$ when $T$ and $T'$ belong to a common $V \in \cal V$.
	\end{enumerate}
\end{lem}

\begin{proof}
	For $(T, c) \in \cal T(\Gamma) \times C^1(\Gamma, G^\Gamma)$ define a function $g_T^c : T \cup \bar T \rightarrow G$ by setting, for each $(v, u) \in T$,
	\[g_T^c(v, u) = c(u^{-1} v)(u)\]
	and
	\[g_T^c(u,v) = c(u^{-1} v)(u)^{-1}.\]
	Clearly the map $c \mapsto g_T^c(v, u)$ is continuous for fixed $T$. Also note that for any $\gamma \in \Gamma$ we have that $(v, u) \in T$ if and only if $(\gamma v, \gamma u) \in \gamma^d \cdot T$; moreover when this occurs we have
	\[g_{\gamma^d \cdot T}^{\gamma^t \cdot c}(\gamma v, \gamma u) = (\gamma^t \cdot c)(u^{-1} v)(\gamma u) = c(u^{-1} v)(u) = g_T^c(v, u).\]
	Similarly $g_{\gamma^d \cdot T}^{\gamma^t \cdot c}(\gamma v, \gamma u) = g_T^c(v, u)$ when $(v, u) \in \bar T$.
	
	For $(T,c) \in \cal T(\Gamma) \times C^1(\Gamma, G^\Gamma)$ we define $r_T(c) \in Z^1(\Gamma, G^\Gamma)$ as follows. For each $\alpha, \beta \in \Gamma$, let $v_0, \ldots, v_n$ be the sequence of vertices in the geodesic path from $\alpha$ to $\alpha \beta$ in $T \cup \bar T$ (meaning $v_0 = \alpha$, $(v_{i+1}, v_i) \in T \cup \bar T$ for every $i$, and $v_n = \alpha \beta$) and set
	\begin{equation} \label{eqn:treecoh}
		r_T(c)(\beta)(\alpha) = g_T^c(v_n, v_{n-1}) g_T^c(v_{n-1}, v_{n-2}) \cdots g_T^c(v_1, v_0).
	\end{equation}
	Based on the previous paragraph, it is immediate that $c \mapsto r_T(c)$ is continuous for fixed $T$. Additionally, $r_T(c)(\beta)(\alpha)$ only depends on $T$ in so far as $T$ determines the path $v_0, \ldots, v_n$; hence (\ref{item:retract2}) holds. Note that since $T$ is a tree and $g_T^c(u,v) = g_T^c(v,u)^{-1}$ for all $(v,u) \in T \cup \bar T$, in the above formula one can use any path in $T \cup \bar T$ from $\alpha$ to $\alpha \beta$, not necessarily the geodesic path. Consequently, if $u_0 = \gamma, \ldots, u_\ell = \gamma \alpha$ is a path from $\gamma$ to $\gamma \alpha$ in $T \cup \bar T$ and $v_0 = \gamma \alpha, \ldots, v_m = \gamma \alpha \beta$ is a path from $\gamma \alpha$ to $\gamma \alpha \beta$ in $T \cup \bar T$, then $u_0, \ldots, u_\ell, v_1, \ldots, v_m$ is a path in $T \cup \bar T$ from $\gamma$ to $\gamma \alpha \beta$ and it from equation (\ref{eqn:treecoh}) that, for all $\gamma\in \Gamma$:
	\[r_T(c)(\beta)(\gamma \alpha) r_T(c)(\alpha)(\gamma) = r_T(c)(\alpha \beta)(\gamma).\]
	Thus
	\[r_T(c)(\beta)^\alpha r_T(c)(\alpha) = r_T(c)(\alpha \beta)\]
	and $r_T(c) \in Z^1(\Gamma, G^\Gamma)$ as required.
	
	To see that $r_T$ is a retraction, consider any $z \in Z^1(\Gamma, G^\Gamma)$. For all $\alpha, \beta, \gamma \in \Gamma$ we have $\partial z(\alpha, \beta)(\gamma) = e_G$, meaning
	\[z(\beta)(\gamma \alpha) z(\alpha)(\gamma) = z(\alpha \beta)(\gamma).\]
	By making the substitutions $\alpha = u$, $\beta = u^{-1} v$, and $\gamma = e$ we obtain
	\[z(u^{-1} v)(u)  z(u)(e) = z(v)(e).\]
	Thus when $(v,u) \in T$ we have
	\[g_T^z(v,u) = z(u^{-1} v)(u) = z(v)(e) z(u)(e)^{-1}\]
	and
	\[g_T^z(u,v) = g_T^z(v,u)^{-1} = z(u)(e) z(v)(e)^{-1}.\]
	It now immediately follows from the formula (\ref{eqn:treecoh}) that for all $\alpha, \beta \in \Gamma$
	\[r_T(z)(\beta)(\alpha) = z(\alpha \beta)(e) z(\alpha)(e)^{-1} = z(\beta)(\alpha),\]
	and thus $r_T(z) = z$.
	
	Lastly, when $v_0 = \gamma^{-1} \alpha, \ldots, v_n = \gamma^{-1} \alpha \beta$ is a path in $T \cup \bar T$ from $\gamma^{-1} \alpha$ to $\gamma^{-1} \alpha \beta$, we have $\gamma v_0, \ldots, \gamma v_n$ is a path in $\gamma^d \cdot (T \cup \bar T)$ from $\alpha$ to $\alpha \beta$. Therefore
	\begin{align*}
		r_{\gamma^d \cdot T}(\gamma^t \cdot c)(\beta)(\alpha) & = g_{\gamma^d \cdot T}^{\gamma^t \cdot c}(\gamma v_n, \gamma v_{n-1}) \cdots g_{\gamma^d \cdot T}^{\gamma^t \cdot c}(\gamma v_1, \gamma v_0)\\
		& = g_T^c(v_n, v_{n-1}) \cdots g_T^c(v_1, v_0)\\
		& = r_T(c)(\beta)(\gamma^{-1} \alpha)\\
		& = (\gamma^t \cdot r_T(c))(\beta)(\alpha),
	\end{align*}
	and $r_{\gamma^d \cdot T}(\gamma^t \cdot c) = \gamma^t \cdot r_T(c)$ as required by (\ref{item:retract1}).
\end{proof}

Note that if $H$ is a subgroup of $G$ then $C^1(\Gamma, H^\Gamma)$ and $Z^1(\Gamma, H^\Gamma)$ are subsets of $C^1(\Gamma, G^\Gamma)$ and $Z^1(\Gamma, G^\Gamma)$, respectively; moreover if $r_T$ is the map defined in the previous lemma with respect to $G$, then $r_T$ maps the subset $C^1(\Gamma, H^\Gamma)$ to $Z^1(\Gamma, H^\Gamma)$.

When $\Gamma \acts^a (X, \mu)$ is a free {\pmp} treeable action there is a measurable equivariant map $x \in X \mapsto T(x) \in \cal T(\Gamma)$ from $\Gamma \acts^a (X, \mu)$ to $\Gamma \acts^d (\cal T(\Gamma), \tau)$ for some invariant Borel probability measure $\tau = T_* \mu$. In this case the above lemma provides a technique for turning cochains into cocycles (recall Lemma \ref{lem:CocycleCorr}): given any measurable equivariant map $\theta : X \rightarrow C^1(\Gamma, G^\Gamma)$ we obtain a measurable equivariant map $\phi : X \rightarrow Z^1(\Gamma, G^\Gamma)$ via the formula $\phi(x) = r_{T(x)}(\theta(x))$. Although this does closely resemble the technique we used for verifying the definable cocycle and extension-MD properties for free groups, this new variant of the technique has a significant limitation in that the measure $\tau \in \Prob_\Gamma(\cal T(\Gamma))$ may vary as one considers different treeable actions of $\Gamma$. The constraint of factoring onto $\Gamma \acts^d (\cal T(\Gamma), \tau)$ for some particular $\tau$ poses an obstruction that, for the moment, prevents us from proving that $T_{\Gamma,free}$ has the definable cocyle property when $\Gamma$ is strongly treeable.

\subsection{Weak containment of treeings}\label{subs:treeable}

The purpose of this section is to overcome the limitation described at the end of the previous subsection. We will do this by relaxing the requirement of factoring onto $\Gamma \acts^d (\cal T(\Gamma), \tau)$ to merely weakly containing the action $\Gamma \acts^d (\cal T(\Gamma), \tau)$. We will show that in this case the mapping $T \mapsto r_T$ from Lemma \ref{lem:retract} is still usable in an approximate form.

We begin with a technical lemma.

\begin{lem} \label{lem:reltop}
	Let $X$ and $Z$ be compact Hausdorff spaces, let $\pi_X : X \times Z \rightarrow X$ be the projection map, let $\mu$ be a Borel probability measure on $X$, and let $Y \subseteq Z$ be Borel. Set
	\[\Omega = \{\omega \in \Prob(X \times Z) : \omega(X \times Y) = 1, \ (\pi_X)_* \omega = \mu\},\]
	and equip $\Omega$ with the relative topology inherited from the weak$^*$ topology on $\Prob(X \times Z)$. If $f : X \times Y \rightarrow [0,1]$ is a measurable function having the property that for every $\epsilon > 0$ there is a relatively clopen partition $\cal V$ of $Y$ satisfying
	\[\sup_{x \in X} |f(x,y) - f(x,y')| < \epsilon\]
	whenever $y$ and $y'$ belong to a common $V \in \cal V$, then the map $\omega \in \Omega \mapsto \int f \ d \omega$ is continuous.
\end{lem}

\begin{proof}
	Fix $\epsilon > 0$ and let $\cal V$ be as described. Also fix a point $\omega_0 \in \Omega$. Choose a finite subcollection $\cal V' \subseteq \cal V$ so that $\omega_0(Y \setminus \bigcup \cal V') < \epsilon$, and choose a clopen set $W \subseteq Z$ with $W \cap Y = Y \setminus \bigcup \cal V'$. For each $V \in \cal V'$ pick a point $y_V \in V$ and choose a continuous function $g_V : X \rightarrow [0,1]$ satisfying
	\[\int |g_V(x) - f(x,y_V)| \ d \mu < \frac{\epsilon}{|\cal V'|}.\]
	Pick a collection $\cal U'$ of pairwise disjoint clopen subsets of $Z$ having nonempty intersection with $Y$ and satisfying $\{U \cap Y : U \in \cal U'\} = \cal V'$, and define a continuous function $g : X \times Z \rightarrow [0,1]$ by
	\[g(x,z) = \sum_{U \in \cal U'} g_{U \cap Y}(x) 1_U(z).\]
	For every $\omega \in \Omega$ we have
	\begin{align*}
		\int |f-g| \ d \omega & \leq \omega(W) + \sum_{V \in \cal V'} \int_{X \times V} |f(x,y)-g_V(x)| \ d \omega\\
		& \leq \omega(W) + \epsilon + \sum_{V \in \cal V'} \int_{X \times V} |f(x,y_V) - g_V(x)| \ d \omega\\
		& \leq \omega(W) + \epsilon + \sum_{V \in \cal V'} \int_{X \times Z} |f(x,y_V) - g_V(x)| \ d \omega\\
		& = \omega(W) + \epsilon + \sum_{V \in \cal V'} \int |f(x,y_V) - g_V(x)| \ d \mu\\
		& < \omega(W) + 2 \epsilon.
	\end{align*}
	The set
	\[\left\{\omega \in \Omega : \omega(W) < \epsilon, \ \left|\int g \ d \omega - \int g \ d \omega_0 \right| < \epsilon \right\}\]
	is a relatively open neighborhood of $\omega_0$, and every $\omega$ in this set satisfies
	\[\left| \int f \ d \omega - \int f \ d \omega_0 \right| < \omega(W) + 2 \epsilon + \omega_0(W) + 2 \epsilon + \epsilon = 7 \epsilon.\qedhere\]
\end{proof}

For the lemma below, fix metrics $\rho_1$ and $\rho_2$ on $C^1(\Gamma, G^\Gamma)$ and $C^2(\Gamma, G^\Gamma)$, respectively, that map to $[0,1]$, are compatible with the product topologies, and satisfy the following uniform condition: for every $\epsilon > 0$ there is a finite set $F \subseteq \Gamma$ such that $\forall c, c' \in C^1(\Gamma, G^\Gamma)$:
\[\big( \ c(\beta)(\alpha) = c'(\beta)(\alpha)  \text{ for all }\alpha, \beta \in F \big)\Rightarrow \rho_1(c,c'), \rho_2(\partial c, \partial c') < \epsilon.\]
When $G$ is compact this uniform condition automatically holds for all metrics that are compatible with the topology. In any case, such metrics $\rho_1$ and $\rho_2$ can be constructed by, for example, fixing an enumeration $(\gamma_n)_{n \in \bN}$ of $\Gamma$, picking a $[0,1]$-valued metric $\rho_0$ on $G$ compatible with its topology, and defining
\[\rho_1(c,c') = \sum_{n,m \in \bN} 2^{-n-m} \rho_0 \big(c(\gamma_n)(\gamma_m), c'(\gamma_n)(\gamma_m) \big)\]
and
\[\rho_2(c,c') = \sum_{n,m,k \in \bN} 2^{-n-m-k} \rho_0 \big(c(\gamma_n, \gamma_m)(\gamma_k), c'(\gamma_n, \gamma_m)(\gamma_k) \big).\] 

\begin{lem} \label{lem:weaklift}
	Let $\tau$ be an invariant Borel probability measure on $\cal T(\Gamma)$, let $X$ be a standard Borel space, and let $\Gamma \acts^a (X, \mu)$ be a {\pmp} action that weakly contains $\Gamma \acts^d (\cal T(\Gamma), \tau)$. Then there is an invariant Borel probability measure $\omega$ on $X \times \cal T(\Gamma)$ that pushes forward to $\mu$ and $\tau$ under the projection maps to $X$ and $\cal T(\Gamma)$, respectively,
	and satisfies the following:
	for every measurable equivariant map $\theta : X \rightarrow C^1(\Gamma, G^\Gamma)$ and $\epsilon > 0$ there is a measurable equivariant map $\theta' : X \rightarrow C^1(\Gamma, G^\Gamma)$ satisfying
	\[\int \rho_1(\theta(x), \theta'(x)) \ d \mu < \int \rho_1(\theta(x), r_T(\theta(x))) \ d \omega + \epsilon\]
	and
	\[\int \rho_2( \partial \theta'(x), e_{C^2(\Gamma, G^\Gamma)}) \ d \mu < \epsilon.\]
\end{lem}

\begin{proof}
	We let $\Gamma$ act on $\Pow{\Gamma \times \Gamma}^\Gamma$ by the usual left-shift action
	\[(\gamma^s \cdot y)(\alpha) = y(\gamma^{-1} \alpha) \qquad y \in \Pow{\Gamma \times \Gamma}^\Gamma, \ \alpha, \gamma \in \Gamma\]
	and we define an equivariant lifting map $\ell : \cal T(\Gamma) \rightarrow \cal T(\Gamma)^\Gamma \subseteq \Pow{\Gamma \times \Gamma}^\Gamma$ via the formula
	\[\ell(T)(\gamma) = (\gamma^{-1})^d \cdot T.\]
	
	We will first construct a sequence of measurable equivariant maps $\psi_n : X \rightarrow \cal T(\Gamma)^\Gamma$ such that $(\psi_n)_* \mu$ converges to $\ell_* \tau$ in the weak$^*$ topology on $\Prob(\Pow{\Gamma \times \Gamma}^\Gamma)$. Since $\Pow{\Gamma \times \Gamma}$ is $0$-dimensional, we can pick an increasing sequence $\cal U_n$ of finite clopen partitions of $\Pow{\Gamma \times \Gamma}$ such that $\cal U_{n+1}$ is finer than $\cal U_n$ for every $n$ and such that $\bigcup_n \cal U_n$ is a base for the topology on $\Pow{\Gamma \times \Gamma}$. Also fix an increasing sequence of finite sets $F_n \subseteq \Gamma$ having the property $\bigcup_n F_n = \Gamma$. For each $n$ write $\cal U_n = \{U_n^i : i \in k_n\}$. Define the clopen partition $\cal V_n = \{V_{n,\pi} : \pi \in k_n^{F_n}\}$ of $\Pow{\Gamma \times \Gamma}^\Gamma$ where
	\[V_{n,\pi} = \prod_{\gamma \in F_n} U_n^{\pi(\gamma)}.\]
	Then $(\cal V_n)$ is a sequence of clopen partitions of $\Pow{\Gamma \times \Gamma}^\Gamma$ with $\cal V_{n+1}$ finer than $\cal V_n$ for every $n$ and with $\bigcup_n \cal V_n$ a base for the topology on $\Pow{\Gamma \times \Gamma}^\Gamma$. It easily follows from these properties, as well as the compactness of $\Pow{\Gamma \times \Gamma}^\Gamma$, that the set of all finite linear combinations of characteristic functions of sets in $\bigcup_n \cal V_n$ is uniformly dense in the space of all real-valued continuous functions on $\Pow{\Gamma \times \Gamma}^\Gamma$. So it will be enough for our maps $\psi_n$ to satisfy $(\psi_n)_* \mu(V) \rightarrow \ell_* \tau(V)$ as $n \rightarrow \infty$ for every $V \in \bigcup_m \cal V_m$.
	
	Since $\Gamma \acts^a (X, \mu)$ weakly contains $\Gamma \acts (\cal T(\Gamma), \tau)$, for every $n$ there is a partition $\cal Q_n = \{Q_n^i : i \in k_n\}$ of $X$ such $Q_n^i = \varnothing$ when $U_n^i \cap \cal T(\Gamma) = \varnothing$ and satisfying
	\begin{equation} \label{eq:weaklift}
		\sum_{\pi \in k_n^{F_n}} \left| \mu \left( \bigcap_{\gamma \in F_n} (\gamma^{-1})^a \cdot Q_n^{\pi(\gamma)} \right) - \tau \left( \bigcap_{\gamma \in F_n} (\gamma^{-1})^d \cdot U_n^{\pi(\gamma)} \right) \right| < \frac{1}{2^n}.
	\end{equation}
	For every $n$ and every $i \in k_n$ with $U_n^i \cap \cal T(\Gamma) \neq \varnothing$, pick a point $u_n^i \in U_n^i \cap \cal T(\Gamma)$. Define $\psi_n : X \rightarrow \cal T(\Gamma)^\Gamma$ by setting $\psi_n(x)(\gamma) = u_n^i$ when $(\gamma^{-1})^a \cdot x \in Q_n^i$. Then we have that for every $n$
	\begin{align*}
		& \sum_{\pi \in k_n^{F_n}} \Big| (\psi_n)_* \mu (V_{n,\pi}) - \ell_*\tau (V_{n,\pi}) \Big|\\
		& = \sum_{\pi \in k_n^{F_n}} \left| (\psi_n)_* \mu \left( \prod_{\gamma \in F_n} U_n^{\pi(\gamma)} \right) - \ell_* \tau \left( \prod_{\gamma \in F_n} U_n^{\pi(\gamma)} \right) \right|\\
		& = \sum_{\pi \in k_n^{F_n}} \left| \mu \left( \bigcap_{\gamma \in F_n} (\gamma^{-1})^a \cdot Q_n^{\pi(\gamma)} \right) - \tau \left( \bigcap_{\gamma \in F_n} (\gamma^{-1})^d \cdot U_n^{\pi(\gamma)} \right) \right| < \frac{1}{2^n}.
	\end{align*}
	It follows that $(\psi_n)_* \mu$ weak$^*$ converges to $\ell_* \tau$ as desired.
	
	We will now define $\omega$. Without loss of generality, we can assume $X$ is uncountable. Since $X$ is a standard Borel space, we can pick a compact Hausdorff topology on $X$ that is compatible with its Borel $\sigma$-algebra. Since $\Pow{\Gamma \times \Gamma}^\Gamma$ is also compact, the space of all Borel probability measures on $X \times \Pow{\Gamma \times \Gamma}^\Gamma$ is compact in the weak$^*$ topology. So there is a Borel probability measure $\widehat \omega$ on $X \times \Pow{\Gamma \times \Gamma}^\Gamma$ that is a subsequential limit of the measures $(\id \times \psi_n)_* \mu$. Since each $\psi_n$ is equivariant, $\widehat \omega$ is $\Gamma$-invariant, and since the projection map from $X \times \Pow{\Gamma \times \Gamma}^\Gamma$ to $\Pow{\Gamma \times \Gamma}^\Gamma$ is continuous, the pushforward of $\widehat \omega$ with respect to the projection must be equal to $\ell_* \tau$. Also note that $\widehat \omega$ and every measure $(\id \times \psi_n)_* \mu$ belong to the set $\Omega$ of Borel probability measures on $X \times \Pow{\Gamma \times \Gamma}^\Gamma$ that assign measure $1$ to $X \times \cal T(\Gamma)^\Gamma$ and pushforward to $\mu$ under the projection map to $X$.
    
    We let $\omega$ be the measure obtained as the pushforward of $\widehat \omega$ under the map $(x, y) \mapsto (x, y(e))$, and note that the pushforwards of $\omega$ to $X$ and $\cal T(\Gamma)$ are $\mu$ and $\tau$. Also notice that $\widehat \omega$ is similarly obtained as the pushforward of $\omega$ under the map $(x,T) \mapsto (x, \ell(T))$.
	
	Finally, let $\theta : X \rightarrow C^1(\Gamma, G^\Gamma)$ be any equivariant measurable map. Let $T \mapsto r_T$ be the map from Lemma \ref{lem:retract}. For $y \in \cal T(\Gamma)^\Gamma$ and $c \in C^1(\Gamma, G^\Gamma)$ define $\hat r_y(c) \in C^1(\Gamma, G^\Gamma)$ by the formula
	\[\hat r_y(c)(\beta)(\alpha) = r_{y(\alpha)}((\alpha^{-1})^t \cdot c)(\beta)(e).\]
	We claim that, using $Z = \Pow{\Gamma \times \Gamma}^\Gamma$ and $Y = \cal T(\Gamma)^\Gamma$, both of the functions
	\[(x,y) \mapsto \rho_1(\theta(x), \hat r_y(\theta(x))) \quad \text{and} \quad (x, y) \mapsto \rho_2(\partial \hat r_y(\theta(x)), e_{C^2(\Gamma, G^\Gamma)})\]
	satisfy the assumptions stated in Lemma \ref{lem:reltop}. To see this, let $\epsilon > 0$. Pick a finite set $F \subseteq \Gamma$ such that, for all $c, c' \in C^1(\Gamma, G^\Gamma)$:
	\[ \big( \ c(\beta)(\alpha) = c'(\beta)(\alpha) \text{ for all }\alpha, \beta \in F\big) \Rightarrow \rho_1(c,c'), \rho_2(\partial c, \partial c') < \epsilon.\]
	By Lemma \ref{lem:retract} there is a clopen partition $\cal U$ of $\cal T(\Gamma)$ so that for every $U \in \cal U$ we have $r_T(c)(\beta)(e) = r_{T'}(c)(\beta)(e)$ for all $c \in C^1(\Gamma, G^\Gamma)$, $\beta \in F$, and $T, T' \in U$. Letting $\cal V$ be the clopen partition of $\cal T(\Gamma)^\Gamma$ whereby $y, y' \in \cal T(\Gamma)^\Gamma$ belong to the same piece of $\cal V$ if and only if $y(\alpha)$ and $y'(\alpha)$ both belong to a common piece of $\cal U$ for every $\alpha \in F$, it follows that when $y, y' \in V \in \cal V$ we have that for all $c \in C^1(\Gamma, G^\Gamma)$
	\[\hat r_{y}(c)(\beta)(\alpha) = r_{y(\alpha)}((\alpha^{-1})^t \cdot c)(\beta)(e) = r_{y'(\alpha)}((\alpha^{-1})^t \cdot c)(\beta)(e) = \hat r_{y'}(c)(\beta)(\alpha)\]
	for all $\alpha, \beta \in F$ and hence $\rho_1(\hat r_y(c), \hat r_{y'}(c)), \ \rho_2(\partial \hat r_y(c), \partial \hat r_{y'}(c)) < \epsilon$. Therefore when $y, y' \in V \in \cal V$ we have that
	\[\sup_{x \in X} \big| \rho_1(\theta(x), \hat r_y(\theta(x))) - \rho_1(\theta(x), \hat r_{y'}(\theta(x))) \big| \leq \epsilon\]
	and
	\[\sup_{x \in X} \big| \rho_2(\partial \hat r_y(\theta(x)), e_{C^2(\Gamma, G^\Gamma)}) - \rho_2(\partial \hat r_{y'}(\theta(x)), e_{C^2(\Gamma, G^\Gamma)}) \big| \leq \epsilon.\]
	This verifies our claim. Thus Lemma \ref{lem:reltop} applies to these two functions and the measures $\widehat \omega$ and $(\id \times \psi_n)_* \mu$.
	
	Notice that the map $(y, c) \mapsto \hat r_y(c)$ is jointly equivariant since
	\begin{align*}
		\hat r_{\gamma^s \cdot y}(\gamma^t \cdot c)(\beta)(\alpha) & = r_{(\gamma^s \cdot y)(\alpha)}((\alpha^{-1} \gamma)^t \cdot c)(\beta)(e)\\
		& = r_{y(\gamma^{-1}\alpha)}((\alpha^{-1} \gamma)^t \cdot c)(\beta)(e)\\
		& = \hat r_y(c)(\beta)(\gamma^{-1} \alpha)\\
		& = (\gamma^t \cdot \hat r_y(c))(\beta)(\alpha).
	\end{align*}
	Therefore $x \mapsto \hat r_{\psi_n(x)}(\theta(x))$ is a measurable equivariant map from $X$ to $C^1(\Gamma, G^\Gamma)$ for every $n$. Additionally, for every $T \in \cal T(\Gamma)$ we have
	\begin{align*}
		\hat r_{\ell(T)}(c)(\beta)(\alpha) & = r_{\ell(T)(\alpha)}((\alpha^{-1})^t \cdot c)(\beta)(e)\\
		& = r_{(\alpha^{-1})^d \cdot T}((\alpha^{-1})^t \cdot c)(\beta)(e)\\
		& = ((\alpha^{-1})^t \cdot r_T(c))(\beta)(e)\\
		& = r_T(c)(\beta)(\alpha),
	\end{align*}
	and thus $\hat r_{\ell(T)} = r_T$. Therefore
	\[\int \rho_1(\theta(x), r_T(\theta(x))) \ d \omega = \int \rho_1(\theta(x), \hat r_y(\theta(x))) \ d \widehat \omega\]
	and, since each $r_T$ maps to $Z^1(\Gamma, G^\Gamma)$,
	\[0 = \int \rho_2(\partial r_T(\theta(x)), e_{C^2(\Gamma, G^\Gamma)}) \ d \omega = \int \rho_2(\partial \hat r_y(\theta(x)), e_{C^2(\Gamma, G^\Gamma)}) \ d \widehat \omega.\]
	By applying Lemma \ref{lem:reltop} we conclude that for any $\epsilon > 0$ there is an $n$ with
	\[\int \rho_1(\theta(x), \hat r_y(\theta(x))) \ d (\id \times \psi_n)_* \mu < \int \rho_1(\theta(x), r_T(\theta(x))) \ d \omega + \epsilon\]
	and
	\[\int \rho_2(\partial \hat r_y(\theta(x)), e_{C^2(\Gamma, G^\Gamma)}) \ d (\id \times \psi_n)_* \mu < \epsilon.\]
	Defining $\theta'(x) = \hat r_{\psi_n(x)}(\theta(x))$ for any such value of $n$ completes the proof.
\end{proof}

\subsection{Cocycles on actions of treeable groups} \label{subs:cocytree}

The next proposition is a ``definable cocycle property'' for the class of actions that weakly contain a particular treeable action.

\begin{prop} \label{prop:weaktree}
	Let $\Gamma$ be a treeable group, let $\tau$ be an invariant Borel probability measure on $\cal T(\Gamma)$, let $G$ be a compact metrizable group, and let $\rho_1$ and $\rho_2$ be metrics on $C^1(\Gamma, G^\Gamma)$ and $C^2(\Gamma, G^\Gamma)$, respectively, that are compatible with their product topologies and map to $[0,1]$.
	
	Then for every $\epsilon > 0$ there is $\delta(\epsilon) > 0$ with the following property: whenever $\Gamma \acts^a (X, \mu)$ is a {\pmp} action that weakly contains $\Gamma \acts^d (\cal T(\Gamma), \tau)$ and $\theta: X \rightarrow C^1(\Gamma, G^\Gamma)$ is an equivariant measurable map satisfying
	\[\int \rho_2(\partial \theta(x), e_{C^2(\Gamma, G^\Gamma)}) \ d \mu < \delta(\epsilon),\]
	there is an equivariant measurable map $\phi : X \rightarrow Z^1(\Gamma, G^\Gamma)$ satisfying
	\[\int \rho_1(\theta(x), \phi(x)) \ d \mu < \epsilon.\]
\end{prop}

\begin{proof}
    If needed, we can replace $\Gamma \acts^a (X,\mu)$ with a factor that is an action on a standard Borel space such that $\theta$ descends to this factor and this factor still weakly contains $\Gamma \acts^d (\cal T(\Gamma), \tau)$. So without loss of generality, throughout the proof we always assume that $X$ is a standard Borel space.
    
    We first show that under the stated assumptions there is $\delta(\epsilon) > 0$ so that whenever $\Gamma \acts^a (X, \mu)$ is a {\pmp} action that weakly contains $\Gamma \acts^d (\cal T(\Gamma), \tau)$ and $\theta : X \rightarrow C^1(\Gamma, G^\Gamma)$ is a measurable equivariant map with
	\[\int \rho_2(\partial \theta(x), e_{C^2(\Gamma, G^\Gamma)}) \ d \mu < \delta(\epsilon),\]
	then there is an $\epsilon' \in (0, \epsilon)$ so that for every $\delta' > 0$ there is an equivariant measurable map $\theta' : X \rightarrow C^1(\Gamma, G^\Gamma)$ satisfying
	\[\int \rho_1(\theta(x), \theta'(x)) \ d \mu < \epsilon' \quad \text{and} \quad \int \rho_2(\partial \theta'(x), e_{C^2(\Gamma, G^\Gamma)}) \ d \mu < \delta'.\]
	
	Let $P$ be the set of all invariant Borel probability measures on $\Pow{\Gamma \times \Gamma} \times C^1(\Gamma, G^\Gamma)$ whose pushforward under the projection map to $\Pow{\Gamma \times \Gamma}$ is equal to $\tau$. Since $G$ is compact and the pushforward map associated with the projection map is continuous, it follows that $P$ is a compact subset of $\Prob(\Pow{\Gamma \times \Gamma} \times C^1(\Gamma, G^\Gamma))$.
	
	For each $T \in \cal T(\Gamma)$ let $r_T : C^1(\Gamma, G^\Gamma) \rightarrow Z^1(\Gamma, G^\Gamma)$ be the retraction described in Lemma \ref{lem:retract}, and note that for every $\epsilon > 0$ there is a clopen partition $\cal W$ of $\cal T(\Gamma)$ with
    \[\sup_{c \in C^1(\Gamma, G^\Gamma)} \big|\rho_1(c,r_T(c))-\rho_1(c,r_{T'}(c)) \big| < \epsilon\]
    for all $W \in \cal W$ and $T, T' \in W$. For $\epsilon, \delta > 0$ set
	\[V_\epsilon = \left\{\omega \in P : \int \rho_1(c, r_T(c)) \ d \omega < \epsilon\right\}\]
	and
	\[U_\delta = \left\{\omega \in P : \int \rho_2(\partial c, e_{C^2(\Gamma, G^\Gamma)}) \ d \omega \leq \delta\right\}.\]
	Since each $r_T$ is a retraction to $Z^1(\Gamma, G^\Gamma)$, we have that $\bigcap_{\delta > 0} U_\delta \subseteq V_\epsilon$. Additionally, using $C^1(\Gamma, G^\Gamma)$, $\Pow{\Gamma \times \Gamma}$, and $\cal T(\Gamma)$ for $X$, $Z$, $Y$ in Lemma \ref{lem:reltop}, respectively, we see that $V_\epsilon$ is an open subset of $P$. Also, each $U_\delta$ is compact since $\rho_2(\partial c, e_{C^2(\Gamma, G^\Gamma)})$ is a continuous function on $\Pow{\Gamma \times \Gamma} \times C^1(\Gamma, G^\Gamma)$. It follows that for every $\epsilon > 0$ there is some $\delta(\epsilon) > 0$ such that $U_{\delta(\epsilon)} \subseteq V_\epsilon$, and since $V_\epsilon$ is covered by the open sets $V_{\epsilon'}$, $\epsilon' < \epsilon$, there is an $\epsilon' < \epsilon$ with $U_{\delta(\epsilon)} \subseteq V_{\epsilon'}$.
	
	Now fix $\epsilon > 0$ and suppose that $\Gamma \acts^a (X, \mu)$ is a {\pmp} action that weakly contains $\Gamma \acts^d (\cal T(\Gamma), \tau)$ and that $\theta: X \rightarrow C^1(\Gamma, G^\Gamma)$ is a measurable equivariant map satisfying
	\[\int \rho_2(\partial \theta(x), e_{C^2(\Gamma, G^\Gamma)}) \ d \mu < \delta(\epsilon).\]
	Letting $\omega \in \Prob(X \times \cal T(\Gamma))$ be the measure obtained from Lemma \ref{lem:weaklift}, we have
	\[(\pi_{\cal T(\Gamma)} \times \theta)_* \omega \in U_{\delta(\epsilon)} \subseteq V_{\epsilon'}\]
	for some $\epsilon' < \epsilon$, where $\pi_{\cal T(\Gamma)}$ denotes the projection map. It follows from Lemma \ref{lem:weaklift} that for every $\delta' > 0$ there is a measurable equivariant map $\theta' : X \rightarrow C^1(\Gamma, G^\Gamma)$ satisfying
	\[\int \rho_1(\theta(x), \theta'(x)) \ d \mu < \epsilon' \quad \text{and} \quad \int \rho_2(\partial \theta'(x), e_{C^2(\Gamma, G^\Gamma)}) \ d \mu < \delta'.\]
	This establishes our claim. We will now use this claim to complete the proof.
	
	Let $\epsilon > 0$, let $\Gamma \acts^a (X, \mu)$ be a {\pmp} action that weakly contains $\Gamma \acts^d (\cal T(\Gamma), \tau)$ and let $\theta : X \rightarrow C^1(\Gamma, G^\Gamma)$ be a measurable equivariant map satisfying $$\int \rho_2(\partial \theta(x), e_{C^2(\Gamma, G^\Gamma)}) \ d \mu < \delta(\epsilon).$$ Set $\theta_0 = \theta$ and $\epsilon_0 = \epsilon$. Inductively suppose that $\theta_n$ and $\epsilon_n$ have been defined and satisfy
	\[\int \rho_2(\partial \theta_n(x), e_{C^2(\Gamma, G^\Gamma)}) \ d \mu < \min(\delta(\epsilon_n), 2^{-n+1})\]
	(this holds when $n = 0$ since $\rho_2$ is bounded by $1$). Applying our above claim gives us an $\epsilon' \in (0, \epsilon_n)$. Choose any $0 < \epsilon_{n+1} < \min(\epsilon_n - \epsilon', 2^{-n})$. Then by applying our above claim with $\delta' = \min(\delta(\epsilon_{n+1}), 2^{-(n+1)+1})$ we obtain a measurable equivariant map $\theta_{n+1} : X \rightarrow C^1(\Gamma, G^\Gamma)$ satisfying
	\[\int \rho_1(\theta_n(x), \theta_{n+1}(x)) \ d \mu < \epsilon' < \epsilon_n - \epsilon_{n+1} < \epsilon_n\]
	and
	\[\int \rho_2(\partial \theta_{n+1}(x), e_{C^2(\Gamma, G^\Gamma)}) \ d \mu < \delta' = \min(\delta(\epsilon_{n+1}), 2^{-(n+1)+1}).\]
	This allows the inductive construction to proceed.
	
	Since $\sum_n \epsilon_n < \infty$, the sequence $(\theta_n)$ is Cauchy with respect to the metric $\int \rho_1(\cdot, \cdot) \ d \mu$. Since $G$ is Polish, $C^1(\Gamma, G^\Gamma)$ is Polish as well and therefore both $\rho_1$ and $\int \rho_1(\cdot, \cdot) \ d \mu$ are complete. So $(\theta_n)$ converges $\mu$-almost-everywhere to a measurable function $\phi : X \rightarrow C^1(\Gamma, G^\Gamma)$. Since each $\theta_n$ is equivariant, $\phi$ is as well (if needed, we can redefine $\phi$ to have value $e_{C^1(\Gamma, G^\Gamma)}$ on the null set of orbits for which it fails to be equivariant). Also, since the coboundary map $\partial$ is continuous we have that $(\partial \theta_n)$ converges to $\partial \phi$ almost-everywhere. So $\int \rho_2(\partial \phi(x), e_{C^2(\Gamma, G^\Gamma)}) \ d \mu = 0$ and by redefining $\phi$ on a $\Gamma$-invariant null set if necessary we have that $\phi$ is a measurable equivariant map from $X$ to $Z^1(\Gamma, G^\Gamma)$. Finally,
	\[\int \rho_1(\theta(x), \phi(x)) \ d \mu \leq \sum_{n=0}^\infty \int \rho_1(\theta_n(x), \theta_{n+1}(x)) \ d \mu < \sum_{n=0}^\infty \epsilon_n-\epsilon_{n+1} = \epsilon_0 = \epsilon.\qedhere\]
\end{proof}

We can now conclude the following, generalizing Lemma \ref{freedefcocycle}:

\begin{thm}\label{treeabledefcocyle} \
	\begin{enumerate}
		\item If $\Gamma$ is strongly treeable, then $T_{\Gamma,free}$ has the definable cocycle property.
		\item If $\Gamma$ is treeable, then $T_{\Gamma,max}$ has the definable cocycle property.
	\end{enumerate}
\end{thm}

\begin{proof}
	The theorem follows immediately from the previous proposition by noting that, in either case, a model of the respective theory weakly contains $(\cal T(\Gamma),\tau)$ for some invariant Borel probability measure $\tau$ on $\cal T(\Gamma)$.  Indeed, this is immediate in the second case by assumption.  In the first case, by a result of Abert-Weiss \cite{AW13}, every free {\pmp} action of $\Gamma$ weakly contains the Bernoulli shift action $\Gamma \acts^s ([0,1], \lambda)$. Since $\Gamma$ is strongly treeable, the action $\Gamma \acts^s ([0,1], \lambda)$ must be treeable, meaning it admits a factor map to $\Gamma \acts^d (\cal T(\Gamma), \tau)$ for some invariant Borel probability measure $\tau$. It follows that $\Gamma \acts^s ([0,1], \lambda)$ weakly contains $\Gamma \acts^d (\cal T(\Gamma), \tau)$, and since weak containment is transitive, every free {\pmp} action of $\Gamma$ weakly contains $\Gamma \acts (\cal T(\Gamma), \tau)$.
\end{proof}

We similarly generalize Lemma \ref{freeweakMD}:

\begin{thm}\label{stronglytreeableweakMD}
	If $\Gamma$ is strongly treeable, then $\Gamma$ has the extension-MD property.  More generally, if $\Gamma$ is a treeable group, and $\Gamma \acts^a (X, \mu)$ is a {\pmp} action that weakly contains a free treeable {\pmp} action of $\Gamma$, then $\Gamma \acts^a (X, \mu)$ is extension-MD.
\end{thm}

\begin{proof}
	Recalling that every free {\pmp} action of a strongly treeable group weakly contains a free treeable {\pmp} action (see the proof of Theorem \ref{treeabledefcocyle}), it suffices to prove the second statement. So let $\tau$ be an invariant Borel probability measure on $\cal T(\Gamma)$ and suppose that $\Gamma \acts^a (X, \mu)$ is a {\pmp} action that weakly contains $\Gamma \acts^d (\cal T(\Gamma), \tau)$. It will suffice to verify condition (\ref{item:exMD4}) of Lemma \ref{lem:exMD}. So let $z : X \rightarrow Z^1(\Gamma, G^\Gamma)$ be a measurable equivariant map, where $G = \Aut([0,1], \lambda)$ equipped with the weak topology, and let $\epsilon > 0$. Also let $\rho_1$ and $\rho_2$ be metrics on $C^1(\Gamma, G^\Gamma)$ and $C^2(\Gamma, G^\Gamma)$ satisfying the conditions stated prior to Lemma \ref{lem:weaklift}. As in the proof of Proposition \ref{prop:weaktree}, without loss of generality we will assume that $X$ is a standard Borel space.
	
	For any subgroup $H \leq G$ we can build an equivariant map $\theta : X \rightarrow C^1(\Gamma, H^\Gamma)$ approximating $z$ by choosing, independently for each $\beta \in \Gamma$, a measurable map $x \in X \mapsto \theta(x)(\beta)(e) \in H$ that approximates $x \mapsto z(x)(\beta)(e)$, and then define $\theta(x)(\beta)(\alpha) = \theta((\alpha^{-1})^a \cdot x)(\beta)(e)$ for $\alpha, \beta \in \Gamma$. It is easily checked that any such $\theta$ will indeed be equivariant. Upon noting that the map $c \in C^1(\Gamma, G^\Gamma) \mapsto \rho_1(c, r_T(c))$ is continuous for every $T \in \cal T(\Gamma)$, that $\rho_1(z(x), r_T(z(x))) = 0$ for $\omega$-almost-every $(x,T)$, and that $G$ admits a dense subset that is an increasing union of finite subgroups, we see that we can find a finite subgroup $H$ and construct a measurable equivariant map $\theta : X \rightarrow C^1(\Gamma, H^\Gamma)$ that approximates $z$ sufficiently well so that
	\[\int \rho_1(z(x), \theta(x)) \ d \mu < \epsilon / 3 \quad \text{and} \quad \int \rho_1(\theta(x), r_T(\theta(x))) \ d \omega < \epsilon / 3.\]
	Let $\delta(\epsilon/3) > 0$ be as given by Proposition \ref{prop:weaktree} for $\tau$ and the compact group $H$. Now apply first Lemma \ref{lem:weaklift} to get a measurable equivariant map $\theta' : X \rightarrow C^1(\Gamma, H^\Gamma)$ satisfying
	\[\int \rho_1(\theta(x), \theta'(x)) \ d \mu < \epsilon/3\]
	and
	\[\int \rho_2(\partial \theta'(x), e_{C^2(\Gamma, G^\Gamma)}) \ d \mu < \delta(\epsilon/3),\]
	and next apply Proposition \ref{prop:weaktree} to $\theta'$ to obtain a measurable equivariant map $z' : X \rightarrow Z^1(\Gamma, H^\Gamma)$ with $\int \rho_1(\theta'(x), z'(x)) \ d \mu < \epsilon / 3$. Since
	\[\rho_1(z(x),z'(x)) \leq \rho_1(z(x), \theta(x)) + \rho_1(\theta(x), \theta'(x)) + \rho_1(\theta'(x), z'(x)),\]
	we have $\int \rho_1(z(x), z'(x)) \ d \mu < \epsilon$.
\end{proof}

The following is the main conclusion of this section:

\begin{thm}\label{treeableecchar}
    If $\Gamma$ is treeable, then $T_\Gamma^*$ exists and the axioms for $T_\Gamma^*$ are given by the axioms for $T_{\Gamma,\max}$ together with the sentences $\theta_{k,q,F}$ from the proof of Theorem \ref{mainaxioms1}. If, in addition, $\Gamma$ is strongly treeable, then we may replace the axioms for $T_{\Gamma,\max}$ by the axioms for $T_{\Gamma,free}$. 
\end{thm}

Furthermore, we obtain a natural-to-state characterization of the e.c.\ actions of treeable groups.

\begin{thm} \label{thm:simple}
	Let $\Gamma$ be a treeable group. Then a {\pmp} action $\Gamma \acts^a (X, \mu)$ is e.c.\ if and only if all of the following hold:
	\begin{enumerate}
		\item \label{item:simple1} $\Gamma \acts^a (X, \mu)$ weakly contains a free treeable action of $\Gamma$;
		\item \label{item:simple2} $\Gamma \acts^{a \times \id} (X \times [0,1], \mu \times \lambda)$ is an e.c.\ extension of $\Gamma \acts^a (X, \mu)$;
		\item \label{item:simple3} $B^1(a, \operatorname{Sym}(k))$ is dense in $Z^1(a, \operatorname{Sym}(k))$ for every $k \in \bb N$.
	\end{enumerate}
\end{thm}

Notice that when $\Gamma$ is strongly treeable (\ref{item:simple1}) can be replaced with the requirement that the action $a$ is free.

\begin{proof}[Proof of Theorem \ref{thm:simple}]
	The forward implication follows from Lemmas \ref{lem:locuniv} and \ref{eccocycle} and the definition of being e.c., and the reverse implication follows from Corollary \ref{cor:simple}, Lemma \ref{weakMDlemma} and Theorem \ref{stronglytreeableweakMD}.
\end{proof}

Since conditions (\ref{item:simple1}) and (\ref{item:simple2}) are obviously necessary but not sufficient, the key condition in the above characterization is (\ref{item:simple3}). This reveals that being e.c.\ (for treeable groups) is closely tied to cohomological properties of the action. We also remark that when $\Gamma \acts^a (X, \mu)$ is ergodic, condition (\ref{item:simple2}) is equivalent to $\Gamma \acts^a (X, \mu)$ not being strongly ergodic (the forward implication holds since $a$ weakly contains the non-ergodic action $a \times \id$, and the reverse implication can be seen from the (very short) proof of \cite[Theorem 3]{AW13}).

\section{Approximately treeability and existence of the model companion}\label{approximatelytreeable}

In this final section, we seek to establish the existence of $T_\Gamma^*$ for as large a class of groups as we are able, but without insisting on finding explicit axioms for $T_\Gamma^*$ (though the general set of axioms described in Subsection \ref{openmappingsub} will always apply). Ultimately, we use the open mapping characterization of the existence of the model companion (Theorem \ref{prop:openmap}) to show that $T_\Gamma^*$ exists whenever $\Gamma$ is an approximately treeable group.

\subsection{Approximately treeable groups}

We first give some constructions for and examples of approximately treeable groups. Recall that every treeable group is necessarily approximately treeable.

\begin{lem} \label{lem:union_approx_tree}
    If $\Gamma$ is an increasing union of approximately treeable groups, then $\Gamma$ is approximately treeable.
\end{lem}

\begin{proof}
   Say $\Gamma = \bigcup_n \Gamma_n$, with $\Gamma_n \leq \Gamma_{n+1}$ and $\Gamma_n$ approximately treeable for every $n$. Let $H \subseteq \Gamma$ be finite and let $\epsilon > 0$. Pick any $n$ with $H \subseteq \Gamma_n$. Let $\nu$ be a $\Gamma_n$-invariant Borel probability measure on $\cal F(\Gamma_n)$ with $\nu(\{F \in \cal F(\Gamma_n) : H \times H \not\subseteq E_F\}) < \epsilon$. Viewing $\cal F(\Gamma_n)$ as a subset of $\cal F(\Gamma)$, the pushforward measure $\gamma^d_* \nu$ depends only on the coset $\gamma \Gamma_n \in \Gamma / \Gamma_n$ and not on the particular representative $\gamma$. We can view the product $\prod_{\gamma \Gamma_n \in \Gamma / \Gamma_n} \gamma^d \cdot \cal F(\Gamma_n)$ as a subset of $\cal F(\Gamma)$, and by letting $\mu$ be the measure on $\cal F(\Gamma)$ obtained as the product of the measures $\gamma^d_* \nu$ for $\gamma \Gamma_n \in \Gamma / \Gamma_n$, we see that $\mu$ is a $\Gamma$-invariant measure satisfying $\mu(\{F \in \cal F(\Gamma) : H \times H \not\subseteq E_F\}) < \epsilon$. Thus $\Gamma$ is approximately treeable.
\end{proof}

\begin{lem} \label{lem:some_approx_tree}
	If $\Lambda \lhd \Gamma$, $\Lambda$ is approximately treeable, and $\Gamma / \Lambda$ is amenable, then $\Gamma$ is approximately treeable.  
\end{lem}

\begin{proof}
    Fix a choice of representatives $r : \Gamma / \Lambda \rightarrow \Gamma$ for the cosets of $\Lambda$ in $\Gamma$ with $r(\Lambda) = e$. Define the cocycle $\rho : \Gamma \times (\Gamma / \Lambda) \rightarrow \Lambda$ by $\rho(\gamma, a \Lambda) = r(\gamma a \Lambda)^{-1} \gamma r(a \Lambda)$. Let $\Lambda \acts^b (Y, \nu)$ be a free approximately treeable {\pmp} action. Define the action $\Gamma \acts^{b'} (Y^{\Gamma / \Lambda}, \nu^{\Gamma / \Lambda})$ by
    \[(\gamma^{b'} \cdot \bar y)(a \Lambda) = (\rho(\gamma^{-1}, a \Lambda)^{-1})^b \cdot \bar y(\gamma^{-1} a \Lambda).\]
    Since $\Lambda \acts^b (Y, \nu)$ is free and measure-preserving, it is easily seen that $\Gamma \acts^{b'} (Y^{\Gamma / \Lambda}, \nu^{\Gamma / \Lambda})$ is free and measure-preserving as well.

    Let $m$ denote Lebesgue measure on $[0,1]$ and let $\Gamma \acts^c ([0,1]^{\Gamma / \Lambda}, m^{\Gamma / \Lambda})$ be the left-shift action given by
    \[(\gamma^c \cdot z)(a \Lambda) = z(\gamma^{-1} a \Lambda).\]
    Let $(X, \mu)$ be the product $(Y^{\Gamma / \Lambda} \times [0,1]^{\Gamma / \Lambda}, \nu^{\Gamma / \Lambda} \times m^{\Gamma / \Lambda})$ and let $\Gamma \acts^a (X, \mu)$ be the direct product action $a = b' \times c$. Since the action $\Gamma \acts^a (X, \mu)$ is free and measure-preserving, it suffices to show that it is approximately treeable. Towards this end, fix a finite set $H_0 \subseteq \Gamma$ and an $\epsilon > 0$.

    Since $\Lambda$ is normal in $\Gamma$, almost-every $z \in [0,1]^{\Gamma / \Lambda}$ has stabilizer equal to $\Lambda$. This means that $\Gamma \acts^c ([0,1]^{\Gamma / \Lambda}, m^{\Gamma / \Lambda})$ descends to a free action of $\Gamma / \Lambda$. Since $\Gamma / \Lambda$ is amenable, the orbit equivalence relation $\cal R_c = \{(z, \gamma^c \cdot z) : z \in [0,1]^{\Gamma / \Lambda}, \ \gamma \in \Gamma\}$ must be $m^{\Gamma / \Lambda}$-hyperfinite \cite{OW80}, meaning there is a sequence of measurable equivalence relations $E_n \subseteq \cal R_c$ such that each class of each $E_n$ is finite, $E_n \subseteq E_{n+1}$ for all $n$, and $\bigcup_n E_n$ coincides with $\cal R_c$ on a $\Gamma / \Lambda$-invariant conull set. This implies that we can pick $n$ large enough so that $F = E_n$ satisfies $m^{\Gamma / \Lambda}(Z_0) > 1 - \epsilon/3$ where
    \[Z_0 = \{z \in Z : \forall h \in H_0 \ (h^c \cdot z, z) \in F\}.\]
    Since $F$ is a Borel equivalence relation whose classes are all finite, there exists a Borel set $D \subseteq [0,1]^{\Gamma / \Lambda}$ that contains precisely one point from every $F$-class \cite[Theorem 12.16]{K95}. Let $d : [0,1]^{\Gamma / \Lambda} \rightarrow \Gamma$ be a function satisfying $d(z)^c \cdot z \in D$ and $(d(z)^c \cdot z, z) \in F$ for all $z \in Z$. By fixing an enumeration of the elements of $\Gamma$ and letting $d(z)$ be the least element satisfying these conditions, we have that $d$ is measurable. Pick a finite set $H_1 \subseteq \Gamma$ large enough that $m^{\Gamma / \Lambda}(Z_1) > 1 - \epsilon/3$ where
    \[Z_1 = \{z \in [0,1]^{\Gamma / \Lambda} : \forall h \in H_0 \ d(h^c \cdot z) \in H_1\}.\]

    Using the fact that $\Lambda$ is approximately treeable, pick a measurable directed graph $\cal H \subseteq \cal R_b$ having no cycles such that the equivalence relation $\cal R_\cal H$ given by the $\cal H$-connected components satisfies $\nu(Y_1) > 1 - \epsilon / (3|H_1|)$ where
    \[Y_1 = \{y \in Y : \forall h \in H_1 H_0 H_1^{-1} \cap \Lambda \ (h^b \cdot y, y) \in \cal R_\cal H\}.\]
    Now define a measurable directed graph $\cal H' \subseteq \{((\bar y, z), (\lambda^a \cdot \bar y, z)) : \bar y \in Y^{\Gamma / \Lambda}, \ z \in D, \ \lambda \in \Lambda\}$ via the rule
    \[((\bar y_1, z), (\bar y_2,z)) \in \cal H' \Leftrightarrow z \in D \wedge (\bar y_1(e), \bar y_2(e)) \in \cal H.\]
    Since the map $\pi : Y^{\Gamma / \Lambda} \rightarrow Y$ given by $\pi(\bar y) = y(e)$ is $\Lambda$-equivariant and since the action $\Lambda \acts^b (Y, \nu)$ is free, we see that $\cal H'$ has no cycles.
    
    Finally, define
    \[\cal G = \cal H' \cup \{((\bar y, z), d(z)^a \cdot (\bar y, z)) : \bar y \in Y^{\Gamma / \Lambda}, \ z \in [0,1]^{\Gamma / \Lambda} \setminus D\}.\]
    When $(\bar y, z) \in Y^{\Gamma / \Lambda} \times ([0,1]^{\Gamma / \Lambda} \setminus D)$ we have that $((\bar y, z), d(z)^a \cdot (\bar y, z))$ is the only edge in $\cal G$ leaving $(\bar y, z)$ and there are no edges in $\cal G$ pointing towards $(\bar y, z)$. Therefore $\cal G$ remains acyclic.

    Now suppose $(\bar y, z) \in X_0$ where
    \[X_0 = \left( \bigcap_{h \in H_1} (h^{-1})^{b'} \cdot \pi^{-1}(Y_1) \right) \times (Z_0 \cap Z_1),\]
    and let $h_0 \in H_0$. Then $(h_0^c \cdot z, z) \in F$ which implies that $z' = d(z)^c \cdot z \in D$ is equal to $(d(h_0^c \cdot z) h_0)^c \cdot z$. Also note that
    \[((\bar y, z), d(z)^a \cdot (\bar y, z)) \in \cal G \quad \text{and} \quad (h_0^a \cdot (\bar y, z), (d(h_0^c \cdot z) h_0)^a \cdot (\bar y, z) \in \cal G.\]
    Since $z \in Z_1$ we have that $h_1 = d(h_0^c \cdot z) h_0 d(z)^{-1}$ belongs to $H_1 H_0 H_1^{-1}$. Moreover, $h_1^c \cdot z' = z'$. Since the stabilizer of $z'$ must be $\Lambda$ (excluding a null set), we have $h_1 \in H_1 H_0 H_1^{-1} \cap \Lambda$. Finally, since $\pi(d(z)^{b'} \cdot \bar y) \in Y_1$ we have that the points
    \[d(z)^a \cdot (\bar y, z) = (d(z)^{b'} \cdot \bar y, z')\]
    and
    \[(d(h_0^c \cdot z) h_0)^a \cdot (\bar y, z) = (h_1 d(z))^a \cdot (\bar y, z) = h_1^a \cdot (d(z)^{b'} \cdot \bar y, z')\]
    belong to the same connected component of $\cal H'$. Thus $h_0^a \cdot (\bar y, z)$ and $(\bar y, z)$ belong to the same connected component of $\cal G$. This completes the proof since $\mu(X_0) > 1 - \epsilon$.
\end{proof}

\begin{cor}
    Universally free groups are approximately treeable.
\end{cor}

\begin{proof}
    By Lemma \ref{lem:union_approx_tree} we may assume that $\Gamma$ is finitely generated, that is, that $\Gamma$ is a limit group.  By the aforementioned theorem of Kochloukova \cite{Kochloukova}, $\Gamma$ has a normal coamenable free subgroup $\Lambda$. Since $\Lambda$ is a free group, it is strongly treeable. Therefore $\Gamma$ is approximately treeable by Lemma \ref{lem:some_approx_tree}.
\end{proof}

\subsection{Outline of the proof}

Recall from Subsection \ref{treeabletypes} that $\Gamma$ is approximately treeable if for every weak$^*$ open neighborhood $U$ of the point-mass $\delta_{\Gamma \times \Gamma} \in \Prob(\cal E(\Gamma))$, there is an invariant Borel probability measure $\mu$ on $\cal F(\Gamma)$ so that the pushforward of $\mu$ under the map $F \in \cal F(\Gamma) \mapsto E_F \in \cal E(\Gamma)$ belongs to $U$.

For any set $A$, we let $s$ denote the left-shift action of $\Gamma$ on $A^\Gamma$ given by the formula $(\gamma^s \cdot x)(\delta) = x(\gamma^{-1} \delta)$ for $\gamma, \delta \in \Gamma$ and $x \in A^\Gamma$. When $C \subseteq \Gamma$ and $\gamma \in \Gamma$, we also write $\gamma^s$ to denote the map from $A^C$ to $A^{\gamma C}$ given by the same formula (where now $\delta \in \gamma C$).

The following is the main theorem of this section.

\begin{thm} \label{thm:approxtree}
	If $\Gamma$ is an approximately treeable group, then $T_\Gamma^*$ exists.
\end{thm}

We provide a brief sketch of the proof. As mentioned earlier, our proof of this theorem will use the open mapping characterization of the existence of the model companion given in Theorem \ref{prop:openmap}. Towards that end, consider $p,q\in \bN$, nonempty open $\Lambda \subseteq \Prob_\Gamma(q^\Gamma \times p^\Gamma)$, and $\lambda \in \Lambda$. Let $\nu \in \Prob_\Gamma(q^\Gamma)$ be the pushforward of $\lambda$ under the canonical projection map, and let $y \in q^\Gamma \mapsto \lambda_y \in \Prob(p^\Gamma)$ be the disintegration of $\lambda$ over $\nu$, so that $\lambda = \int \delta_y \times \lambda_y \ d \nu$ where $\delta_y$ is the point-mass measure at $y$.

For $y \in q^\Gamma$ and $F \in \cal F(\Gamma)$, we can construct a Borel probability measure $\phi(\lambda_y, E_F)$ on $p^\Gamma$ by taking the independent product over the classes $C \in \Gamma / E_F$ of the pushforward of $\lambda_y$ with respect to the projection $p^\Gamma \rightarrow p^C$. In other words, $\phi(\lambda_y, E_F)$ is obtained from $\lambda_y$ by independently re-randomizing $\lambda_y$ over each of the classes $C \in \Gamma / E_F$. A simple but important observation is that if $E_F = \Gamma \times \Gamma$ is the indiscrete equivalence relation then $\phi(\lambda_y, E_F) = \lambda_y$, and moreover if $E_F$ is sufficiently close to $\Gamma \times \Gamma$ then $\phi(\lambda_y, E_F)$ will be close to $\lambda_y$. The fact that $\Gamma$ is approximately treeable allows us to pick an invariant measure $\mu \in \Prob(\cal F(\Gamma))$ so that $E_F$ is close to $\Gamma \times \Gamma$ with probability close to $1$. This can be done so that when we average the measures $\delta_y \times \phi(\lambda_y, E_F)$ over $\nu$ and $\mu$ we obtain an invariant probability measure contained in $\Lambda$.

In the next step in the proof, which is the more technically challenging portion, we choose a continuous function $\kappa : q^\Gamma \rightarrow \Prob(p^\Gamma)$ that approximates the function $y \in q^\Gamma \mapsto \lambda_y$ sufficiently well in $\nu$-measure. For $y \in q^\Gamma$ and $F \in \cal F(\Gamma)$ we build a measure on $(p^\Gamma)^\Gamma$ whose projection to $(p^\Gamma)^{\{\gamma\}}$ is $\kappa((\gamma^{-1})^s \cdot y) \approx \lambda_{(\gamma^{-1})^s \cdot y}$ for each $\gamma \in \Gamma$. When $(\delta, \gamma) \in F$ the projections of the measure to $(p^\Gamma)^{\{\delta\}}$ and $(p^\Gamma)^{\{\gamma\}}$ will be coupled via an isomorphism from $(p^\Gamma, \kappa((\gamma^{-1})^s \cdot y))$ to $(p^\Gamma, \kappa((\delta^{-1})^s \cdot y))$ that is close to the shift map $(\delta^{-1} \gamma)^s$ with probability close to $1$. Additionally, as in the previous paragraph the measure will have independent projections to the sets $(p^\Gamma)^C$ for $C \in \Gamma / E_F$. Using the map $z \in (p^\Gamma)^\Gamma \mapsto z(\cdot)(e) \in p^\Gamma$ we pushforward this measure to $p^\Gamma$ and call the resulting measure $\theta(\kappa_y, F)$. The fact that the couplings associated with the edges in $F$ are close to the respective shift maps, together with the fact that $E_F$ is close to $\Gamma \times \Gamma$ with probability close to $1$, will imply that $\theta(\kappa_y, F)$ is close to $\kappa(y)$ with probability close to $1$. Finally, after taking the average of $\theta(\kappa_y, F)$ with respect to $\nu$ and $\mu$ we obtain an invariant measure close to the one we constructed before and therefore still belonging to $\Lambda$.

The final step is to observe that the construction of the previous paragraph can be repeated for any invariant probability measure $\nu'$ on $q^\Gamma$. Moreover, since $\kappa$ and the overall construction are continuous, the measure obtained will depend continuously on $\nu'$. Therefore if $\nu'$ is sufficiently close to $\nu$ the construction will yield an invariant measure in $\Lambda$ whose projection to $q^\Gamma$ is $\nu'$, completing the proof.

We now proceed to develop all of the necessary details. Throughout the discussion below, fix $p \in \bb N$. 

\subsection{Map-measure pairs}

Let $M$ be the set of all pairs $(h, \omega)$ where $h : p^\Gamma \rightarrow p^\Gamma$ is Borel measurable and $\omega$ is a Borel probability measure on $p^\Gamma$. We equip $M$ with the weakest topology making the map $(h, \omega) \in M \mapsto (h \times \id)_* \omega \in \Prob(p^\Gamma \times p^\Gamma)$ continuous. Equivalently, since $p^\Gamma \times p^\Gamma$ is compact and zero-dimensional, the topology on $M$ is the weakest topology making the maps $(h, \omega) \mapsto \omega(A \cap h^{-1}(B))$ continuous for all pairs of clopen sets $A, B \subseteq p^\Gamma$.

We call a pair $((h_1, \omega_1), (h_0, \omega_0))$ of elements of $M$ \emph{composable} if $\omega_1 = (h_0)_* \omega_0$. We write $M_2$ for the set of all composable pairs and define a composition operation from $M_2$ to $M$, which we denote simply by $\cdot$, by the rule $(h_1, \omega_1) \cdot (h_0, \omega_0) = (h_1 \circ h_0, \omega_0)$.

\begin{lem} \label{lem:comp}
	The composition function $((h_1, \omega_1),(h_0,\omega_0)) \in M_2 \mapsto (h_1 \circ h_0, \omega_0) \in M$ is continuous.
\end{lem}

\begin{proof}
	Let $((h_1, \omega_1), (h_0, \omega_0)) \in M_2$, let $A, B \subseteq p^\Gamma$ be clopen, and let $\epsilon > 0$. Pick a clopen set $B_1 \subseteq p^\Gamma$ such that $\omega_1(B_1 \symd h_1^{-1}(B)) < \epsilon$. Now consider any $((g_1, \zeta_1),(g_0,\zeta_0)) \in M_2$ that is close to $((h_1, \omega_1), (h_0, \omega_0))$ in the sense that
	\[\zeta_1(B_1 \symd g_1^{-1}(B)) < \epsilon \text{ and } |\zeta_0(A \cap g_0^{-1}(B_1)) - \omega_0(A \cap h_0^{-1}(B_1))| < \epsilon.\]
	Since each pair is composable, we have that $\zeta_1 = (g_0)_* \zeta_0$ and $\omega_1 = (h_0)_* \omega_0$. Therefore
	\begin{align*}
		& |\zeta_0(A \cap g_0^{-1}(g_1^{-1}(B))) - \omega_0(A \cap h_0^{-1}(h_1^{-1}(B)))|\\
		& \leq |\zeta_0(A \cap g_0^{-1}(B_1))-\omega_0(A \cap h_0^{-1}(B_1))|\\
		& \qquad + \zeta_0(g_0^{-1}(B_1 \symd g_1^{-1}(B)))+\omega_0(h_0^{-1}(B_1 \symd h_1^{-1}(B)))\\
		& = |\zeta_0(A \cap g_0^{-1}(B_1))-\omega_0(A \cap h_0^{-1}(B_1))|+\zeta_1(B_1 \symd g_1^{-1}(B)) + \omega_1(B_1 \symd h_1^{-1}(B))\\
		& < 3 \epsilon.\qedhere
	\end{align*}
\end{proof}

Let $C \subseteq \Gamma$ and write $M^{(C)}$ for the set of all pairs $((h_\gamma)_{\gamma \in C}, \omega)$ where $h_\gamma : p^\Gamma \rightarrow p^\Gamma$ is a Borel measurable function for every $\gamma \in C$ and $\omega$ is a Borel probability measure on $p^\Gamma$. We give $M^{(C)}$ the weakest topology so that the map $((h_\gamma)_{\gamma \in C}, \omega) \in M^{(C)} \mapsto (h_\gamma, \omega) \in M$ is continuous for every $\gamma \in C$. For $((h_\gamma)_{\gamma \in C}, \omega) \in M^{(C)}$ we write $\prod_{\gamma \in C} h_\gamma$ for the function from $p^\Gamma$ to $(p^\Gamma)^C$ given by $$(\prod_{\gamma \in C} h_\gamma)(x)(\delta)(\beta) = h_\delta(x)(\beta)$$ for $x \in p^\Gamma$, $\delta \in C$, and $\beta \in \Gamma$.

\begin{lem} \label{lem:mfold}
	The map
	\[((h_\gamma)_{\gamma \in C}, \omega) \in M^{(C)} \mapsto \left(\prod_{\gamma \in C} h_\gamma \right)_* \omega \in \Prob((p^\Gamma)^C)\]
	is continuous.
\end{lem}

\begin{proof}
	Choose any cylinder set of $(p^\Gamma)^C$, say $A = \prod_{\gamma \in C} B_\gamma$ where $B_\gamma \subseteq p^\Gamma$ is nonempty and clopen for every $\gamma \in C$ and where, for some finite set $D \subseteq C$, $B_\gamma = p^\Gamma$ for all $\gamma \in C \setminus D$. Fix any $((h_\gamma)_{\gamma \in C}, \omega) \in M^{(C)}$ and $\epsilon > 0$. For each $\gamma \in D$ pick a clopen set $B_\gamma' \subseteq p^\Gamma$ satisfying $\omega(B_\gamma' \symd h_\gamma^{-1}(B_\gamma)) < \epsilon / (2|D|+2)$. Then, for any $((g_\gamma)_{\gamma \in C}, \zeta) \in M^{(C)}$ satisfying $|\zeta(\bigcap_{\gamma \in D} B_\gamma') - \omega(\bigcap_{\gamma \in D} B_\gamma')| < \epsilon/(|D|+1)$ and $\zeta(B_\gamma' \symd g_\gamma^{-1}(B_\gamma)) < \epsilon / (2|D|+2)$ for every $\gamma \in D$, we have
	\begin{align*}
		& \left| \left(\prod_{\gamma \in C} g_\gamma\right)_*\zeta(A) - \left(\prod_{\gamma \in C} h_\gamma\right)_*\omega(A) \right|\\
		& = \left| \zeta \left(\bigcap_{\gamma \in D} g_\gamma^{-1}(B_\gamma) \right) - \omega \left( \bigcap_{\gamma \in D} h_\gamma^{-1}(B_\gamma) \right)\right|\\
		& \leq \left| \zeta \left(\bigcap_{\gamma \in D} B_\gamma' \right) - \omega \left(\bigcap_{\gamma \in D} B_\gamma'\right)\right| + |D| \frac{\epsilon}{|D|+1} < \epsilon.\qedhere
	\end{align*}
\end{proof}

Let $P^*$ denote the set of all pairs $(\kappa_1, \kappa_0)$ of Borel probability measures on $p^\Gamma$ having the property that either $\kappa_1 = \kappa_0$ or else both $\kappa_0$ and $\kappa_1$ are non-atomic. We equip $P^*$ with the subspace topology inherited from the product space $\Prob(p^\Gamma) \times \Prob(p^\Gamma)$.

\begin{lem} \label{lem:pairmap}
	There exists a function $\psi$ assigning to each pair $(\kappa_1, \kappa_0) \in P^*$ a Borel measurable function $\psi(\kappa_1, \kappa_0) : p^\Gamma \rightarrow p^\Gamma$ such that
	\begin{enumerate}
		\item \label{item:pmap1} $\kappa_1 = \psi(\kappa_1, \kappa_0)_* \kappa_0$;
		\item \label{item:pmap2} $\psi(\kappa_0, \kappa_1) \circ \psi(\kappa_1, \kappa_0)$ is equal to the identity $\kappa_0$-almost-everywhere;
		\item \label{item:pmap3} the map $(\kappa_1, \kappa_0) \in P^* \mapsto (\psi(\kappa_1, \kappa_0), \kappa_0) \in M$ is continuous; and
		\item \label{item:pmap4} $(\psi(\kappa_1, \kappa_0), \kappa_0) \rightarrow (\id, \kappa)$ as $(\kappa_1, \kappa_0) \rightarrow (\kappa, \kappa)$.
	\end{enumerate}
\end{lem}

\begin{proof}
	Let $\preceq$ denote the (non-strict) lexicographical ordering on $p^\Gamma$ obtained from some fixed enumeration of the elements of $\Gamma$, and for $z \in p^\Gamma$ set $L_z = \{z' \in p^\Gamma : z' \preceq z\}$. For $(\kappa_1, \kappa_0) \in P^*$ define $\psi(\kappa_1, \kappa_0) : p^\Gamma \rightarrow p^\Gamma$ by
	\[\psi(\kappa_1, \kappa_0)(z) = \inf_\prec \{z' \in p^\Gamma : \kappa_1(L_{z'}) \geq \kappa_0(L_z)\}.\]
	The above set will always be nonempty since the function on $\Gamma$ having constant value $p-1$ is $\preceq$-maximal. Consequently $\psi$ is well-defined.
	
	The function $\psi(\kappa_1, \kappa_0) : p^\Gamma \rightarrow p^\Gamma$ is Borel measurable since it is $\prec$-monotone increasing. Specfically, the sets $L_{z'}$, $z' \in p^\Gamma$, generate the Borel $\sigma$-algebra of $p^\Gamma$ and each preimage $\psi(\kappa_1,\kappa_0)^{-1}(L_{z'})$ is necessarily Borel (it is either empty or else equal to $L_z$ or $L_z \setminus \{z\}$ where $z = \sup_\prec \psi(\kappa_1,\kappa_0)^{-1}(L_{z'})$).
	
	For $\kappa \in \Prob(p^\Gamma)$ set
	\[Z_\kappa = \{z \in p^\Gamma : \forall z' \prec z \ \kappa(L_z \setminus L_{z'}) > 0\}.\]
	Since $p^\Gamma$ contains a countable set that is dense in the $\prec$-ordering, we have that $\kappa(Z_\kappa) = 1$. From these definitions it is clear that $\psi(\kappa, \kappa)$ restricts to the identity map on $Z_\kappa$. Thus conditions (\ref{item:pmap1}) and (\ref{item:pmap2}) hold when $(\kappa_1, \kappa_0) = (\kappa, \kappa)$. Additionally, since $\psi(\kappa,\kappa)$ is $\kappa$-almost-everywhere equal to $\id$, no open set in $M$ can separate the points $(\psi(\kappa,\kappa), \kappa)$ and $(\id, \kappa)$ and therefore condition (\ref{item:pmap4}) will be implied by condition (\ref{item:pmap3}). When $\kappa_1 \neq \kappa_0$ are non-atomic it is easy to see from the definitions that $\psi(\kappa_1, \kappa_0)$ maps $Z_{\kappa_0}$ bijectively onto $Z_{\kappa_1}$ and $\psi(\kappa_0, \kappa_1) \circ \psi(\kappa_1, \kappa_0)$ restricts to the identity on $Z_{\kappa_0}$ (in fullfillment of condition (\ref{item:pmap2})). Finally, since Borel probability measures on $p^\Gamma$ are uniquely determined from their values on the sets $L_z$, $z \in p^\Gamma$, condition (\ref{item:pmap1}) easily follows.
	
	To conclude we verify condition (\ref{item:pmap3}). First observe that, while $L_z$ is always closed, it is open precisely for those countably many $z \in p^\Gamma$ which evaluate to $p-1$ at all but finitely many elements of $\Gamma$. Write $L_\varnothing = \varnothing$ and define $Z$ to be the set of all $z \in p^\Gamma \cup \{\varnothing\}$ for which $L_z$ is clopen. Extend $\prec$ to $Z$ by declaring $\varnothing$ to be $\prec$-minimum. For any nonempty clopen sets $A, B \subseteq p^\Gamma$ we can choose $\prec$-increasing sequences $z_0, \ldots, z_{2n-1} \in Z$ and $z_0', \ldots, z_{2m-1}' \in Z$ so that $A$ and $B$ are the disjoint unions
	\[A = \bigsqcup_{k \in n} L_{z_{2k+1}} \setminus L_{z_{2k}} \quad \text{and} \quad B = \bigsqcup_{\ell \in m} L_{z'_{2\ell+1}} \setminus L_{z'_{2\ell}}.\]
	Then, by a simple inclusion-exclusion computation, $\kappa_0(A \cap \psi(\kappa_1,\kappa_0)^{-1}(B))$ is equal to
	\[\sum_{(k, \ell) \in n \times m} \sum_{i,j \in 2} (-1)^{i+j} \kappa_0 \left(L_{z_{2k+i}} \cap \psi(\kappa_1,\kappa_0)^{-1}(L_{z'_{2\ell+j}}) \right).\]
	Since $\psi(\kappa_1, \kappa_0)$ is $\prec$-monotone increasing, one of the two sets $\psi(\kappa_1,\kappa_0)^{-1}(L_{z'_{2\ell+j}})$ and $L_{z_{2k+i}}$ must contain the other, meaning the measure of their intersection is the smaller of their two measures. Combining this observation with condition (\ref{item:pmap1}), we obtain
	\[\kappa_0(A \cap \psi(\kappa_1,\kappa_0)^{-1}(B)) = \sum_{(k, \ell) \in n \times m} \sum_{i,j \in 2} (-1)^{i+j} \min \left(\kappa_0(L_{z_{2k+i}}), \kappa_1(L_{z'_{2\ell+j}})\right).\]
	Since each of the sets $L_{z_{2k+i}}$ and $L_{z'_{2\ell+j}}$ are clopen, the above is a continuous function of $(\kappa_0, \kappa_1) \in P^*$.
\end{proof}

\subsection{Measure constructions}

Recall the space $\cal E(\Gamma)$ of equivalence relations on $\Gamma$ and the space $\cal F(\Gamma)$ of directed forests on $\Gamma$ from Subsection \ref{treeabletypes}.

\begin{lem} \label{lem:srmap}
	There is a map $\phi : \Prob(p^\Gamma) \times \cal E(\Gamma) \rightarrow \Prob(p^\Gamma)$ such that:
	\begin{enumerate}
		\item \label{item:srmap1} $\phi$ is continuous;
		\item \label{item:srmap3} $\phi(\omega, \Gamma \times \Gamma) = \omega$ for all $\omega \in \Prob(p^\Gamma)$;
		\item \label{item:srmap2} $\gamma^s_* \phi(\omega, E) = \phi(\gamma^s_* \omega, \gamma^d \cdot E)$ for all $\gamma \in \Gamma$, $\omega \in \Prob(p^\Gamma)$ and $E \in \cal E(\Gamma)$.
	\end{enumerate}
\end{lem}

\begin{proof}
	For $C \subseteq \Gamma$ let $\pi_C : p^\Gamma \rightarrow p^C$ be the projection map. Notice that $\beta^s \circ \pi_C = \pi_{\beta C} \circ \beta^s$ for $\beta \in \Gamma$.
	
	For $E \in \cal E(\Gamma)$ and $\omega \in \Prob(p^\Gamma)$ define
	\[\phi(\omega, E) = \prod_{C \in \Gamma / E} (\pi_C)_* \omega.\]
	While the above is formally a product measure indexed by the $E$-classes in $\Gamma$, since $(\pi_C)_* \omega$ is a probability measure on $p^C$ and the sets $C \in \Gamma / E$ partition $\Gamma$ we will identify $\phi(\omega, E)$ as a Borel probability measure on $p^\Gamma$.
	
	Clause (\ref{item:srmap3}) follows immediately from the definition. Also (\ref{item:srmap2}) holds since for $\beta \in \Gamma$ we have
	\begin{align*}
		\beta^s_* \phi(\omega, E) & = \beta^s_* \prod_{C \in \Gamma / E} (\pi_C)_* \omega\\
		& = \prod_{\beta C \in \Gamma / \beta^d \cdot E} \beta^s_* \circ (\pi_C)_* \omega\\
		& = \prod_{\beta C \in \Gamma / \beta^d \cdot E} (\pi_{\beta C})_* \beta^s_* \omega = \phi(\beta^s_* \omega, \beta^d \cdot E).\\
	\end{align*}
	
	Lastly, suppose that $\varnothing \neq B_\gamma \subseteq p$ for $\gamma \in \Gamma$ and that $B_\gamma = p$ for all $\gamma \in \Gamma \setminus D$ where $D \subseteq \Gamma$ is finite. Consider the cylinder set $A = \prod_{\gamma \in \Gamma} B_\gamma$. Then $\phi(\omega, E)(A)$ is a continuous function of $(\omega, E)$ since the family of sets $\cal D_E = \{D \cap C : C \in \Gamma / E\}$ is a locally constant function of $E$ and when $\cal D_E = \cal D$
	\[\phi(\omega, E)(A) = \prod_{C' \in \cal D} \omega \left( \prod_{\gamma \in C'} B_\gamma \times p^{\Gamma \setminus C'} \right)\]
	is a continuous function of $\omega$. It follows that $\phi$ is continuous, in fulfillment of (\ref{item:srmap1}).
\end{proof}

Recall that $P^*$ denotes the set of all pairs $(\kappa_1, \kappa_0) \in \Prob(p^\Gamma) \times \Prob(p^\Gamma)$ such that either $\kappa_1 = \kappa_0$ or else both $\kappa_1$ and $\kappa_0$ are non-atomic. Let $P^{(\Gamma)}$ be the set of all $\omega \in \Prob(p^\Gamma)^\Gamma$ satisfying:
\begin{enumerate}
	\item $\omega(\gamma)$ is non-atomic for every $\gamma \in \Gamma$; or
	\item $\gamma^s_* \omega(\gamma) = \omega(e)$ for every $\gamma \in \Gamma$.
\end{enumerate}
In other words, $\omega \in \Prob(p^\Gamma)^\Gamma$ belongs to $P^{(\Gamma)}$ if and only if $((\delta^{-1} \gamma)^s_* \omega(\gamma), \omega(\delta)) \in P^*$ for all $\delta, \gamma \in \Gamma$.

\begin{lem} \label{lem:fmap}
	There is a map $\theta : P^{(\Gamma)} \times \cal F(\Gamma) \rightarrow \Prob(p^\Gamma)$ such that:
	\begin{enumerate}
		\item[(a)] $\theta(\omega, F)$ is a Borel function of $(\omega, F)$ and is a continuous function of $\omega$;
		\item[(b)] $\theta(\omega, F) = \phi(\omega(e), E_F)$ whenever $\omega \in P^{(\Gamma)}$ satisfies $\gamma^s_* \omega(\gamma) = \omega(e)$ for all $\gamma \in \Gamma$;
		\item[(c)] $\gamma^s_* \theta(\omega, F) = \theta(\gamma^s \cdot \omega, \gamma^d \cdot F)$ for all $\gamma \in \Gamma$, $\omega \in P^{(\Gamma)}$ and $F \in \cal F(\Gamma)$.
	\end{enumerate}
\end{lem}

\begin{proof}
	Let $\psi$ be the function described in Lemma \ref{lem:pairmap}. For $\delta, \gamma \in \Gamma$ and $\omega \in P^{(\Gamma)}$ define Borel functions $g_{\delta, \gamma}^\omega, \bar g_{\delta, \gamma}^\omega : p^\Gamma \rightarrow p^\Gamma$ by
	\[g_{\delta, \gamma}^\omega = \psi(\omega(\delta), (\delta^{-1} \gamma)^s_* \omega(\gamma)) \circ (\delta^{-1} \gamma)^s\]
	and
	\[\bar g_{\delta, \gamma}^\omega = (\delta^{-1} \gamma)^s \circ \psi((\gamma^{-1} \delta)^s_* \omega(\delta), \omega(\gamma)).\]
	The definition of $P^{(\Gamma)}$ ensures that $\psi$ is defined in both of the expressions above. With these definitions, $g_{\delta, \gamma}^\omega$ first performs the shift on $p^\Gamma$ by $\delta^{-1} \gamma$ and then, according to $\psi$, applies a map from $p^\Gamma$ to itself that pushes $(\delta \gamma^{-1})_* \omega(\gamma)$ forward to $\omega(\delta)$. Similarly, $\bar g_{\delta, \gamma}^\omega$ first applies a map from $p^\Gamma$ to itself that pushes $\omega(\gamma)$ forward to $(\gamma^{-1} \delta)^s_* \omega(\delta)$ and then performs the shift on $p^\Gamma$ by $\delta^{-1} \gamma$.
	
	We observe that the following statements hold for all $\delta, \gamma \in \Gamma$:
	\begin{enumerate}
		\item \label{proofitem:fmap1} Each map $g_{\delta,\gamma}^\omega$ and $\bar g_{\delta, \gamma}^\omega$ pushes $\omega(\gamma)$ forward to $\omega(\delta)$. This is clear from the definitions together with clause (\ref{item:pmap1}) of Lemma \ref{lem:pairmap}.
		\item \label{proofitem:fmap2} \indent Both compositions $\bar g_{\gamma, \delta}^\omega \circ g_{\delta, \gamma}^\omega$ and $g_{\gamma, \delta}^\omega \circ \bar g_{\delta, \gamma}^\omega$ are equal to the identity $\omega(\gamma)$-almost-everywhere. This follows from the definitions and clause (\ref{item:pmap2}) of Lemma \ref{lem:pairmap}.
		\item \label{proofitem:fmap3} The pairs $(g_{\delta, \gamma}^\omega, \omega(\gamma)), (\bar g_{\delta, \gamma}^\omega, \omega(\gamma)) \in M$ are continuous functions of $\omega$. Indeed each pair can be expressed as the result of a composition $M_2 \rightarrow M$:
		\[(g_{\delta, \gamma}^\omega, \omega(\gamma)) = (\psi(\omega(\delta), (\delta^{-1} \gamma)^s_* \omega(\gamma)), (\delta^{-1} \gamma)^s_* \omega(\gamma)) \cdot ((\delta^{-1} \gamma)^s, \omega(\gamma))\]
		\[(\bar g_{\delta, \gamma}^\omega, \omega(\gamma)) = ((\delta^{-1} \gamma)^s, (\gamma^{-1} \delta)^s_* \omega(\delta)) \cdot (\psi((\gamma^{-1} \delta)^s_* \omega(\delta), \omega(\gamma)), \omega(\gamma)).\]
		Continuity is a consequence of Lemma \ref{lem:comp}, Lemma \ref{lem:pairmap}.(\ref{item:pmap3}), and the continuity of the push-forward maps $\gamma^s_* : \Prob(p^\Gamma) \rightarrow \Prob(p^\Gamma)$ for $\gamma \in \Gamma$.
		\item \label{proofitem:fmap4} When $\delta^s_* \omega(\delta) = \gamma^s_* \omega(\gamma)$ both $g_{\delta, \gamma}^\omega$ and $\bar g_{\delta, \gamma}^\omega$ are equal to $(\delta^{-1} \gamma)^s$.
		\item \label{proofitem:fmapa} $g_{\beta \delta, \beta \gamma}^{\beta^s \cdot \omega} = g_{\delta, \gamma}^\omega$ and $\bar g_{\beta \delta, \beta \gamma}^{\beta^s \cdot \omega} = \bar g_{\delta, \gamma}^\omega$ for all $\beta \in \Gamma$. This is immediate from the definitions together with the facts that $(\beta^s \cdot \omega)(\beta \delta) = \omega(\delta)$, $(\beta^s \cdot \omega)(\beta \gamma) = \omega(\gamma)$, and $(\beta \delta)^{-1} (\beta \gamma) = \delta^{-1} \gamma$.
	\end{enumerate}
	
	Now for $\omega \in P^{(\Gamma)}$ and $F \in \cal F(\Gamma)$ we define a function $\sigma_F^\omega$ assigning to each pair $(\delta, \gamma) \in E_F$ a Borel measurable function $\sigma_F^\omega(\delta, \gamma) : p^\Gamma \rightarrow p^\Gamma$ as follows. First, when $\delta = \gamma$ we set $\sigma_F^\omega(\delta, \gamma)$ equal to the identity function. Next, if $(\delta, \gamma) \in F \cup \bar F$ then we set
	\[\sigma_F^\omega(\delta, \gamma) = \begin{cases}
		g_{\delta, \gamma}^\omega & \text{if } (\delta, \gamma) \in F\\
		\bar g_{\delta, \gamma}^\omega & \text{if } (\delta, \gamma) \in \bar F.
	\end{cases}\]
	In general for $(\delta, \gamma) \in E_F$, consider the unique sequence of vertices $\gamma_0, \gamma_1, \ldots, \gamma_n$ forming a path in $F \cup \bar F$ from $\gamma$ to $\delta$, specifically $\gamma_0 = \gamma$, $\gamma_n = \delta$, and $(\gamma_{i+1}, \gamma_i) \in F \cup \bar F$ for all $0 \leq i < n$, and set
	\[\sigma_F^\omega(\delta, \gamma) = \sigma_F^\omega(\gamma_n, \gamma_{n-1}) \circ \cdots \circ \sigma_F^\omega(\gamma_1, \gamma_0).\]
	
	The following properties hold for all $\omega \in P^{(\Gamma)}$, $F \in \cal F(\Gamma)$, and $(\delta, \gamma) \in E_F$.
	\begin{enumerate}\setcounter{enumi}{5}
		\item \label{proofitem:fmap5} $\sigma_F^\omega(\delta, \gamma)$ pushes $\omega(\gamma)$ forward to $\omega(\delta)$. This is an immediate consequence of (\ref{proofitem:fmap1}) above.
		\item \label{proofitem:fmap6} If $\beta$ lies on the path from $\gamma$ to $\delta$ then $\sigma_F^\omega(\delta, \gamma) = \sigma_F^\omega(\delta, \beta) \circ \sigma_F^\omega(\beta, \gamma)$. This is immediate from the definition.
		\item \label{proofitem:fmap7} $\sigma_F^\omega(\gamma, \delta) \circ \sigma_F^\omega(\delta, \gamma)$ is equal to the identity $\omega(\gamma)$-almost-everywhere. When $(\delta, \gamma) \in F \cup \bar F$ this fact follows immediately from (\ref{proofitem:fmap2}) above. The general case follows by induction on the length of the path from $\gamma$ to $\delta$. Specifically, if $\beta$ lies on the path from $\gamma$ to $\delta$ then by (\ref{proofitem:fmap6})
		\[\sigma_F^\omega(\gamma, \delta) \circ \sigma_F^\omega(\delta, \gamma) = \sigma_F^\omega(\gamma, \beta) \circ \sigma_F^\omega(\beta, \delta) \circ \sigma_F^\omega(\delta, \beta) \circ \sigma_F^\omega(\beta, \gamma).\]
		By induction we can assume $\sigma_F^\omega(\beta, \delta) \circ \sigma_F^\omega(\delta, \beta)$ is equal to the identity $\omega(\beta)$-almost-everywhere. Since $\sigma_F^\omega(\beta, \gamma)$ pushes $\omega(\gamma)$ forward to $\omega(\beta)$ by (\ref{proofitem:fmap5}), it follows that $\omega(\gamma)$-almost-everywhere $\sigma_F^\omega(\gamma, \delta) \circ \sigma_F^\omega(\delta, \gamma)$ is equal to $\sigma_F^\omega(\gamma, \beta) \circ \sigma_F^\omega(\beta, \gamma)$, and by induction this latter function is equal to the identity $\omega(\gamma)$-almost-everywhere.
		\item \label{proofitem:fmap8} Whenever $\beta$ is in the same connected component as $\delta$ and $\gamma$, $\sigma_F^\omega(\delta, \gamma)$ is $\omega(\gamma)$-almost-everywhere equal to $\sigma_F^\omega(\delta, \beta) \circ \sigma_F^\omega(\beta, \gamma)$. To see this, let $\beta'$ be the unique point belonging to all three of the paths between the points $\delta$, $\gamma$, and $\beta$. Two applications of (\ref{proofitem:fmap6}) yield the two equations:
		\begin{align*}
			\sigma_F^\omega(\delta, \beta) \circ \sigma_F^\omega(\beta, \gamma) & = \sigma_F^\omega(\delta, \beta') \circ \sigma_F^\omega(\beta', \beta) \circ \sigma_F^\omega(\beta, \beta') \circ \sigma_F^\omega(\beta', \gamma)\\
			\sigma_F^\omega(\delta, \gamma) & = \sigma_F^\omega(\delta, \beta') \circ \sigma_F^\omega(\beta', \gamma).
		\end{align*}
		The right-sides of these two equations are $\omega(\gamma)$-almost-everywhere equal since $\sigma_F^\omega(\beta', \gamma)$ pushes $\omega(\gamma)$ forward to $\omega(\beta')$ by (\ref{proofitem:fmap5}) and $\sigma_F^\omega(\beta', \beta) \circ \sigma_F^\omega(\beta, \beta')$ is $\omega(\beta')$-almost-everywhere equal to the identity by (\ref{proofitem:fmap7}). This establishes the claim.
		\item \label{proofitem:fmap10} If $\gamma^s_* \omega(\gamma) = \omega(e)$ for all $\gamma \in \Gamma$ then $\sigma_F^\omega(\delta, \gamma) = (\delta^{-1} \gamma)^s$ for all $(\delta, \gamma) \in E_F$. This follows from (\ref{proofitem:fmap4}).
		\item \label{proofitem:fmapb} $\sigma_{\beta^d \cdot F}^{\beta^s \cdot \omega}(\beta \delta, \beta \gamma) = \sigma_F^\omega(\delta, \gamma)$ for all $\beta \in \Gamma$. This is a consequence of (\ref{proofitem:fmapa}) together with the fact that if $\gamma_0, \ldots, \gamma_n$ is the sequence of vertices forming a path in $F$ from $\gamma$ to $\delta$ then $\beta \gamma_0, \ldots, \beta \gamma_n$ is the sequence of vertices for the path from $\beta \gamma$ to $\beta \delta$ in $\beta^d \cdot F$.
	\end{enumerate}
	Additionally, we observe the following facts:
	\begin{enumerate}\setcounter{enumi}{10}
		\item \label{proofitem:fmapc} The function $(\omega, F) \in P^{(\Gamma)} \times U_{\delta, \gamma} \mapsto (\sigma^\omega_F(\delta, \gamma), \omega(\gamma)) \in M$ is continuous, where $U_{\delta, \gamma}$ is the open set $\{F \in \cal F(\Gamma) : (\delta, \gamma) \in E_F\}$. Specifically, we can partition $U_{\delta, \gamma}$ into a collection $\cal U$ of clopen sets so that for every $U \in \cal U$ the path from $\delta$ to $\gamma$ in $F$ is constant over all $F \in U$. It's enough to observe that for each $U \in \cal U$ the map $(\omega, F) \mapsto (\sigma^\omega_F(\delta, \gamma), \omega(\gamma))$ is continuous on $P^{(\Gamma)} \times U$. Say $\gamma = \gamma_0, \gamma_1, \ldots, \gamma_n = \delta$ is the path from $\gamma$ to $\delta$ in $F$ for all $F \in U$. Then from the definition of $\sigma^\omega_F$ and by (\ref{proofitem:fmap5}) we see that $(\sigma^\omega_F(\delta, \gamma), \omega(\gamma)) \in M$ is the composition of the elements $(\sigma^\omega_F(\gamma_{i+1}, \gamma_i), \omega(\gamma_i)) \in M$ when $F \in U$. It then follows from (\ref{proofitem:fmap3}) and Lemma \ref{lem:comp} that $(\sigma^\omega_F(\delta, \gamma), \omega(\gamma))$ is a continuous function of $(\omega, F) \in P^{(\Gamma)} \times U$.
		\item \label{proofitem:fmap9} If $C \subseteq \Gamma$ is nonempty, $\gamma \in \Gamma$, and $C \cup \{\gamma\}$ is contained in a single $E_F$-class for all $F$ in a set $H \subseteq \cal F(\Gamma)$, then the function $(\omega, F) \in P^{(\Gamma)} \times H \mapsto ((\sigma_F^\omega(\delta, \gamma))_{\delta \in C}, \omega(\gamma)) \in M^{(C)}$ is continuous. Indeed, for every $\delta \in C$ the map $(\omega, F) \in P^{(\Gamma)} \times H \mapsto (\sigma_F^\omega(\delta, \gamma), \omega(\gamma)) \in M$ is continuous by (\ref{proofitem:fmapc}) since $H \subseteq U_{\delta, \gamma}$.
	\end{enumerate}
	
	Let $C \in \Gamma / E_F$ be a connected component and fix any $\gamma_C \in C$. From $\omega$ we define a Borel probability measure on $(p^\Gamma)^C$ by the formula
	\[\left( \prod_{\delta \in C} \sigma_F^\omega(\delta, \gamma_C) \right)_* \omega(\gamma_C).\]
	Notice that this measure depends continuously on $\omega$ by (\ref{proofitem:fmap9}) and Lemma \ref{lem:mfold}. Additionally, the measure we obtain does not depend upon the choice of $\gamma_C \in C$ since for any other element $\gamma_C' \in C$ properties (\ref{proofitem:fmap5}) and (\ref{proofitem:fmap8}) imply
	\begin{align*}
		\left(\prod_{\delta \in C} \sigma_F^\omega(\delta, \gamma_C) \right)_* \omega(\gamma_C) & = \left(\prod_{\delta \in C} \sigma_F^\omega(\delta, \gamma_C) \right)_* \sigma_F^\omega(\gamma_C, \gamma_C')_* \omega(\gamma_C')\\
		& = \left( \prod_{\delta \in C} \sigma_F^\omega(\delta, \gamma_C) \circ \sigma_F^\omega(\gamma_C, \gamma_C') \right)_* \omega(\gamma_C')\\
		& = \left( \prod_{\delta \in C} \sigma_F^\omega(\delta, \gamma_C') \right)_* \omega(\gamma_C').
	\end{align*}
	
	For each $C \in \Gamma / E_F$ pick some $\gamma_C \in C$. Also let $f : (p^\Gamma)^\Gamma \rightarrow p^\Gamma$ be the flattening map given by the formula $f(z)(\gamma) = z(\gamma)(e)$. We define the Borel probability measure $\theta(\omega, F) \in \Prob(p^\Gamma)$ by
	\[\theta(\omega, F) = f_* \prod_{C \in \Gamma / E_F} \left( \prod_{\delta \in C} \sigma_F^\omega(\delta, \gamma_C) \right)_* \omega(\gamma_C).\]
	Formally the expression above to the right of $f_*$ is a product measure on a product space indexed by the set $\Gamma / E_F$, however since the measure associated to the coordinate $C \in \Gamma / E_F$ is a measure on $(p^\Gamma)^C$, we identify the expression to the right of $f_*$ as a Borel probability measure on $(p^\Gamma)^\Gamma$ (and therefore the application of $f_*$ is defined).
	
	We first check (b). Denote by $\pi_C : p^\Gamma \rightarrow p^C$ the projection map for $C \subseteq \Gamma$, and recall that
	\[\phi(\omega(e), E_F) = \prod_{C \in \Gamma / E_F} (\pi_C)_* \omega(e).\]
	Let us also write $f$ for the map from $(p^\Gamma)^C$ to $p^C$ given by the same formula as before: $f(z)(\gamma) = z(\gamma)(e)$ for $\gamma \in C$. Since $\prod_{\gamma \in C} (\gamma^{-1})^s$ maps $p^\Gamma$ to $(p^\Gamma)^C$, it is easily checked that $\pi_C = f \circ \prod_{\gamma \in C} (\gamma^{-1})^s$. If $\gamma^s_* \omega(\gamma) = \omega(e)$ for all $\gamma \in \Gamma$ then using (\ref{proofitem:fmap10}) we obtain
	\begin{align*}
		\prod_{C \in \Gamma / E_F} (\pi_C)_* \omega(e) & = \prod_{C \in \Gamma / E_F} f_* \left( \prod_{\delta \in C} (\delta^{-1})^s \right)_* \omega(e)\\
		& = \prod_{C \in \Gamma / E_F} f_* \left( \prod_{\delta \in C} (\delta^{-1})^s \right)_* (\gamma_C)^s_* \omega(\gamma_C)\\
		& = \prod_{C \in \Gamma / E_F} f_* \left( \prod_{\delta \in C} (\delta^{-1} \gamma_C)^s \right)_* \omega(\gamma_C)\\
		& = f_* \prod_{C \in \Gamma / E_F} \left( \prod_{\delta \in C} \sigma^\omega_F(\delta, \gamma_C) \right)_* \omega(\gamma_C) = \theta(\omega, F),\\
	\end{align*}
	confirming (b).
	
	Next we check (a). Let $\Delta \subseteq \cal F(\Gamma) \times \cal E(\Gamma)$ be the Borel set $\Delta = \{(F, E_F) : F \in \cal F(\Gamma)\}$. Since the map $F \in \cal F(\Gamma) \mapsto (F, E_F) \in \Delta$ is Borel, it suffices to show that $\theta(F, \omega)$ is a continuous function of $(\omega, F, E_F) \in P^{(\Gamma)} \times \Delta$. In fact it is enough to show that $\theta(\omega, F)(A)$ is a continuous function of $(\omega, F, E_F)$ when $A$ is a cylinder set. Say $A = \prod_{\gamma \in \Gamma} B_\gamma$ where $\varnothing \neq B_\gamma \subseteq p$ for all $\gamma \in \Gamma$ and $B_\gamma = p$ for all $\gamma \in \Gamma \setminus D$ where $D \subseteq \Gamma$ is finite. Partition $\Delta$ into a collection $\cal V$ of clopen sets so that for each $V \in \cal V$ the restriction of $E_F$ to $D$ is constant over all $(F, E_F) \in V$. It will be enough to check that for each $V \in \cal V$ the map $(\omega, F, E_F) \mapsto \theta(\omega, F)(A)$ is continuous on $P^{(\Gamma)} \times V$. Say every $(F, E_F) \in V$ partitions $D$ into the family of sets $\cal D$. When $(F, E_F) \in V$ and $C \in \Gamma / E_F$ we have that
	\[\left(f_* \left( \prod_{\delta \in C} \sigma^\omega_F(\delta, \gamma_C) \right)_* \omega(\gamma_C) \right) \left(\prod_{\gamma \in C} B_\gamma\right)\]
	is equal to $1$ if $C \cap D = \varnothing$ and is equal to
	\[\left(f_* \left( \prod_{\delta \in C'} \sigma^\omega_F(\delta, \gamma_{C'}) \right)_* \omega(\gamma_{C'}) \right) \left(\prod_{\gamma \in C'} B_\gamma \right)\]
	when $C \cap D = C' \in \cal D$ and $\gamma_{C'}$ is any element of $C'$ (changing $\gamma_C$ to $\gamma_{C'}$ does not change the resulting value, as explained three paragraphs prior). Now notice that the map
	\[(\omega, F, E_F) \in P^{(\Gamma)} \times V \mapsto f_* \left( \prod_{\delta \in C'} \sigma^\omega_F(\delta, \gamma_{C'}) \right)_* \omega(\gamma_{C'}) \in \Prob(p^{C'})\]
	is continuous by (\ref{proofitem:fmap9}), Lemma \ref{lem:mfold}, and the fact that $f_*$ is continuous. Therefore when $(\omega, F, E_F) \in P^{(\Gamma)} \times V$ we have that
	\[\theta(\omega, F)(A) = \prod_{C' \in \cal D} \left(f_* \left( \prod_{\delta \in C'} \sigma^\omega_F(\delta, \gamma_{C'}) \right)_* \omega(\gamma_{C'}) \right) \left(\prod_{\gamma \in C'} B_\gamma \right)\]
	is a continuous function of $(\omega, F, E_F)$ as claimed.
	
	Lastly, to check (c) consider any $\beta \in \Gamma$. Since left-multiplication by $\beta$ provides a bijection from the $E_F$-classes to the $E_{\beta^d \cdot F}$-classes, for each $C \in \Gamma / E_F$ we can use the point $\widehat \gamma_{\beta C} = \beta \gamma_C$ as the representative for $\beta C$ in computing $\theta(\beta^s \cdot \omega, \beta^d \cdot F)$. Then by (\ref{proofitem:fmapb}) and the fact that the map $f$ is $\Gamma^s$-equivariant we have
	\begin{align*}
		\beta^s_* \theta(\omega, F) & = \beta^s_* f_* \prod_{C \in \Gamma / E_F} \left( \prod_{\delta \in C} \sigma_F^\omega(\delta, \gamma_C) \right)_* \omega(\gamma_C)\\
		& = f_* \prod_{\beta C \in \Gamma / E_{\beta^d \cdot F}} \left( \prod_{\beta \delta \in \beta C} \sigma_F^\omega(\delta, \gamma_C) \right)_* \omega(\gamma_C)\\
		& = f_* \prod_{\beta C \in \Gamma / E_{\beta^d \cdot F}} \left( \prod_{\beta \delta \in \beta C} \sigma_{\beta^d \cdot F}^{\beta^s \cdot \omega}(\beta \delta, \beta \gamma_C) \right)_* \omega(\gamma_C)\\
		& = f_* \prod_{\beta C \in \Gamma / E_{\beta^d \cdot F}} \left( \prod_{\delta \in \beta C} \sigma_{\beta^d \cdot F}^{\beta^s \cdot \omega}(\delta, \hat \gamma_{\beta C}) \right)_* \omega(\beta^{-1} \hat \gamma_{\beta C})\\
		& = f_* \prod_{C \in \Gamma / E_{\beta^d \cdot F}} \left( \prod_{\delta \in C} \sigma_{\beta^d \cdot F}^{\beta^s \cdot \omega}(\delta, \hat \gamma_C) \right)_* (\beta^s \cdot \omega)(\hat \gamma_C) = \theta(\beta^s \cdot \omega, \beta^d \cdot F).\qedhere
	\end{align*}
\end{proof}

\subsection{Existence of the model companion}

We are now ready to prove that $T_\Gamma^*$ exists when $\Gamma$ is approximately treeable:

\begin{proof}[Proof of Theorem \ref{thm:approxtree}]
	We verify the criterion from Theorem \ref{prop:openmap}. Fix $p,q\in \bN$, nonempty open $\Lambda \subseteq \Prob_\Gamma(q^\Gamma \times p^\Gamma)$, and $\lambda \in \Lambda$. Let $\nu \in \Prob_\Gamma(q^\Gamma)$ be the pushforward of $\lambda$ under the canonical projection map, and let $y \in q^\Gamma \mapsto \lambda_y \in \Prob(p^\Gamma)$ be the disintegration of $\lambda$ over $\nu$, so that $\lambda = \int \delta_y \times \lambda_y \ d \nu$ where $\delta_y$ is the point-mass measure at $y$.
	
	Let $\phi : \Prob(p^\Gamma) \times \cal E(\Gamma) \rightarrow \Prob(p^\Gamma)$ be the map from Lemma \ref{lem:srmap}. Since $\phi$ is continuous, a routine calculation shows that the map $E \in \cal E(\Gamma) \mapsto \int \delta_y \times \phi(\lambda_y, E) \ d \nu \in \Prob(q^\Gamma \times p^\Gamma)$ is continuous. It follows from this that the function $\zeta : \Prob(\cal E(\Gamma)) \rightarrow \Prob(q^\Gamma \times p^\Gamma)$ defined by the formula
	\[\zeta(\mu) = \int \int \delta_y \times \phi(\lambda_y, E) \ d \nu(y) \ d \mu(E)\]
	is continuous. Moreover, if $\mu$ is an invariant Borel probability measure on $\cal F(\Gamma)$ and $\bar \mu$ is the pushforward of $\mu$ with respect to the map $F \in \cal F(\Gamma) \mapsto E_F \in \cal E(\Gamma)$, then $\zeta(\bar \mu)$ is also an invariant measure since $\phi$ as well as each of the maps $y \mapsto \delta_y$, $y \mapsto \lambda_y$, and $F \mapsto E_F$ are $\Gamma$-equivariant. By Lemma \ref{lem:srmap} we have that $\zeta(\delta_{\Gamma \times \Gamma}) = \lambda$ belongs to the open set $\Lambda$, so by continuity and the fact that $\Gamma$ is approximately treeable there is an invariant Borel probability measure $\mu \in \Prob(\cal F(\Gamma))$ satisfying $\zeta(\bar \mu) \in \Lambda$.
	
	Recall the sets $P^* \subseteq \Prob(p^\Gamma) \times \Prob(p^\Gamma)$ and $P^{(\Gamma)} \subseteq \Prob(p^\Gamma)^\Gamma$ defined earlier. Define $\kappa_\infty : q^\Gamma \rightarrow \Prob(p^\Gamma)$ by $\kappa_\infty(y) = \lambda_y$. Pick a sequence of continuous functions $\kappa_n$, $n \in \bN$, mapping $q^\Gamma$ to the space of non-atomic Borel probability measures on $p^\Gamma$ such that $(\kappa_n)$ converges $\nu$-almost-everywhere to $\kappa_\infty$. For $n \in \bN \cup \{\infty\}$ and $y \in q^\Gamma$ define $\omega_{n,y} \in \Prob(p^\Gamma)^\Gamma$ by
	\[\omega_{n,y}(\gamma) = \kappa_n((\gamma^{-1})^s \cdot y).\]
	Observe that for $\beta \in \Gamma$ we have $\beta^s \cdot \omega_{n,y} = \omega_{n, \beta^s \cdot y}$ since for all $\gamma \in \Gamma$
	\[(\beta^s \cdot \omega_{n,y})(\gamma) = \omega_{n,y}(\beta^{-1} \gamma) = \kappa_n((\gamma^{-1} \beta)^s \cdot y) = \omega_{n, \beta^s \cdot y}(\gamma).\]
	Our definitions ensure that $\omega_{n,y} \in P^{(\Gamma)}$ for all $n \in \bN$ and all $y \in q^\Gamma$, and that $\omega_{\infty,y} \in P^{(\Gamma)}$ for $\nu$-almost-every $y \in q^\Gamma$. This fact allows us to define for each $n \in \bN \cup \{\infty\}$ a Borel probability measure on $q^\Gamma \times p^\Gamma$ by the formula
	\[\int \int \delta_y \times \theta(\omega_{n,y}, F) \ d \nu(y) \ d \mu(F).\]
	Notice that the above measure is in fact $\Gamma$-invariant since the measures $\nu$ and $\mu$ are $\Gamma$-invariant and the maps $y \mapsto \omega_{n,y}$, $y \mapsto \delta_y$, and $\theta$ are $\Gamma$-equivariant. Now observe that for every $\gamma \in \Gamma$
	\[\gamma^s_* \omega_{\infty,y}(\gamma) =  \gamma^s_* \kappa_\infty((\gamma^{-1})^s \cdot y) = \gamma^s_* \lambda_{(\gamma^{-1})^s \cdot y}\]
	is equal to
	\[\gamma^s_* (\gamma^{-1})^s_* \lambda_y = \lambda_y = \omega_{\infty, y}(e)\]
	for $\nu$-almost-every $y \in q^\Gamma$. Additionally, $\omega_{n,y}$ converges to $\omega_{\infty, y}$ as $n \rightarrow \infty$ for $\nu$-almost-every $y \in q^\Gamma$. Combining these two facts with Lemma \ref{lem:fmap} yields
	\begin{align*}
		& \lim_{n \rightarrow \infty} \int \int \delta_y \times \theta(\omega_{n,y}, F) \ d \nu(y) \ d \mu(F)\\
		& = \int \int \delta_y \times \theta(\omega_{\infty,y}, F) \ d \nu(y) \ d \mu(F)\\
		& = \int \int \delta_y \times \phi(\lambda_y, E_F) \ d \nu(y) \ d \mu(F) = \zeta(\bar \mu) \in \Lambda.
	\end{align*}
	We fix from this point forward an $n \in \bN$ satisfying $\int \int \delta_y \times \theta(\omega_{n,y}, F) \ d \nu(y) \ d \mu(F) \in \Lambda$.
	
	Since $\omega_{n,y} \in P^{(\Gamma)}$ for all $y \in q^\Gamma$, we may define a function $\xi$ taking invariant Borel probability measures on $q^\Gamma$ to invariant Borel probability measures on $q^\Gamma \times p^\Gamma$ by the formula
	\[\xi(\nu') = \int \int \delta_y \times \theta(\omega_{n,y}, F) \ d \nu'(y) \ d \mu(F).\]
	Of course, by construction we have $\xi(\nu) \in \Lambda$. On the other hand, we chose $\kappa_n$, and hence the map $y \mapsto \omega_{n,y}$, to be continuous. Consequently the map $y \in q^\Gamma \mapsto \int \delta_y \times \theta(\omega_{n,y}, F) \ d \mu(F)$ is continuous as well. It follows that $\xi$ is continuous and therefore $\xi^{-1}(\Lambda)$ is an open set of invariant Borel probability measures containing $\nu$. Every $\nu' \in \xi^{-1}(\Lambda)$ is immediately seen from the definition to be the pushforward of $\xi(\nu')$ with respect to the projection map to $q^\Gamma$. We conclude that the pushforward of $\Lambda$ with respect to the projection map is open and thus $T_\Gamma^*$ exists.
\end{proof}

\end{document}